\documentclass[a4paper,12pt]{article}
\usepackage{amsmath,amsfonts,amsthm,amscd,amssymb,latexsym,comment,eufrak}
\usepackage{graphicx,psfrag,subfig}%,jchangebar,oldgerm}
\usepackage[shortlabels]{enumitem}
\usepackage[mathscr]{eucal}
\usepackage[usenames]{color}
\usepackage{tikz}
\usetikzlibrary{matrix}
\usepackage{url}
\includecomment{vnr} 
\excludecomment{pap2}
\excludecomment{full}
\excludecomment{comment}
\specialcomment{pap2}{\begin{color}{blue}}{\end{color}}
\specialcomment{full}{\begin{color}{cyan}}{\end{color}}
\specialcomment{comment}{\begin{color}{red}}{\end{color}}
\excludecomment{comment}
\excludecomment{pap2}
\excludecomment{full}
\title{Automorphisms of Partially Commutative Groups II: Combinatorial Subgroups
}
\author{ \textsf{Andrew J. Duncan}
 \and \textsf{Vladimir N. Remeslennikov}}

\def\nul{\emptyset }
\def\D{\Delta }
\def\d{\delta }
\def\b{\beta }

\def\i{\iota }

\def\l{\lambda }

\def\e{\varepsilon }
\def\G{\Gamma }
\def\g{\gamma }
\def\a{\alpha }
\newcommand{\al}{\nu}

\def\s{\sigma }

\def\W{\Omega}
\def\w{\omega }

\newcommand{\vat}{\breve}
\def\cD{{\cal{D}}}

\def\cH{{\cal{H}}}
\def\cJ{{\cal{J}}}
\def\cK{{\cal{K}}}
\def\cL{{\cal{L}}}
\def\cP{{\cal{P}}}
\def\cR{{\cal{R}}}
\def\cS{{\cal{S}}}
\def\cW{{\cal{W}}}
\def\cQ{{\cal{Q}}}
\newtheorem{theorem}{Theorem}[section]
\newtheorem{lemma}[theorem]{Lemma}
\newtheorem{corol}[theorem]{Corollary}
\newtheorem{prop}[theorem]{Proposition}

\newtheorem{defn}[theorem]{Definition}
\newtheorem*{defn*}{Definition}

\newtheorem{exam}[theorem]{Example}
\newenvironment{expl}{\begin{exam} \rm}{\end{exam}}
\newtheorem{remk}[theorem]{Remark}
\newenvironment{rem}{\begin{remk} \rm}{\end{remk}}

\newtheorem{que}[theorem]{Question}

\newtheoremstyle{citing}% Name
{6pt}{6pt}% Space above and below
{\itshape}% Body font
{0\parindent}{\bfseries}% Heading indent and font
{.}% Punctuation after heading
{ }% Space after head (" " = normal interword space)
{\thmnote{#3}}% Typeset note only, if present
\theoremstyle{citing}
\newtheorem*{varthm}{}% all text is supplied in the option
\numberwithin{equation}{section}
\numberwithin{figure}{section}
\newcommand{\gd}{\mathop{{\rm gd}}}

\newcommand{\FR}{\operatorname{FR}}
\newcommand{\WH}{\operatorname{WH}}

\newcommand{\sym}{{\operatorname{symm}}}

\newcommand{\Iso}{\operatorname{Aut}}
\newcommand{\Aut}{\operatorname{Aut}}
\newcommand{\Isol}{\operatorname{Dom}}
\newcommand{\Sol}{\operatorname{Sol}}

\newcommand{\out}{{\operatorname{{out}}}}
\newcommand{\oAut}{\Aut^{*}}
\newcommand{\Out}{\operatorname{Out}}
\newcommand{\oOut}{\overline{\oAut}}

\newcommand{\act}{{\operatorname{act}}}

\renewcommand{\AA}{\ensuremath{\mathbb{A}}}
\newcommand{\ZZ}{\ensuremath{\mathbb{Z}}}

\newcommand{\la}{\langle}
\newcommand{\ra}{\rangle}
\newcommand{\cl}{\operatorname{cl}}
\newcommand{\ca}{\operatorname{cl}_\ad}

\newcommand{\maps}{\rightarrow}

\newcommand{\bs}{\backslash}
\newcommand{\St}{\operatorname{St}}
\newcommand{\conj}{\operatorname{conj}}
\newcommand{\cSt}{\St^{\conj}}

\newcommand{\Conj}{\operatorname{Conj}}

\newcommand{\iConj}{\Conj_{\operatorname{I}}}
\newcommand{\aConj}{\Conj_{\operatorname{A}}}
\newcommand{\CConj}{\Conj_{\operatorname{C}}}
\newcommand{\NConj}{\Conj_{\operatorname{N}}}

\newcommand{\VConj}{\Conj_{\operatorname{V}}}
\newcommand{\aOonj}{\overline{\aConj}}
\newcommand{\sConj}{\Conj_{\operatorname{S}}}
\renewcommand{\int}{{\operatorname{int}}}
\newcommand{\ext}{{\operatorname{ext}}}
\newcommand{\LInn}{{\operatorname{LInn}}}
\newcommand{\LInni}{\LInn_{\int}}
\newcommand{\LInne}{\LInn_{\ext}}

\newcommand{\CLInn}{\operatorname{LInn}_C}
\newcommand{\ILInn}{\operatorname{LInn}_I}
\newcommand{\NLInn}{\operatorname{LInn}_N}
\newcommand{\RLInn}{\operatorname{LInn}_R}
\newcommand{\SLInn}{\operatorname{LInn}_S}
\newcommand{\TLInn}{\operatorname{LInn}_T}
\newcommand{\ULInn}{\operatorname{LInn}_U}
\newcommand{\VLInn}{\operatorname{LInn}_V}
\newcommand{\WLInn}{\operatorname{LInn}_W}
\newcommand{\Inn}{\operatorname{Inn}}

\newcommand{\Tr}{\operatorname{Tr}}
\newcommand{\Tri}{{\Tr_{\int}}}
\newcommand{\Tre}{{\Tr_{\ext}}}

\newcommand{\Trt}{\tilde{\operatorname{Tr}}}
\newcommand{\tr}{\tau}%\operatorname{tr}}
\newcommand{\atl}{\tilde\a}
\newcommand{\trt}{\tilde{\tr}}
\newcommand{\Inv}{\operatorname{Inv}}
\newcommand{\Invi}{\Inv_{\int}}

\newcommand{\cmp}{{\operatorname{{comp}}}}

\newcommand{\GL}{\operatorname{GL}}

\newcommand{\ad}{\mathfrak{a}}

\newcommand{\lk}{\diamondsuit}
\newcommand{\ado}{\operatorname{out}}
\newcommand{\be}{\begin{enumerate}}
\newcommand{\ee}{\end{enumerate}}
\newcommand{\bd}{\begin{description}}
\newcommand{\ed}{\end{description}}
\newlength{\nts}
\newlength{\rts}
\newlength{\lts}
\newenvironment{aj}{\noindent\color{red} AJD }{}

\begin{document}
\maketitle
%\begin{abstract}
%To the Memory of Wilhelm Magnus.
%\end{abstract}
\section{Introduction}
Partially commutative groups are a class of groups widely studied on account 
both of their intrinsically rich structure and their natural appearance in
many diverse branches of mathematics and computer science (see \cite{Charney07} or \cite{EKR} for example.) 
It is therefore natural that the pace of study of  their automorphism groups
should be gaining  momentum, as it has been recently. 

A {\em partially commutative group} $G(\G)$ (also known 
as a {\em right-angled Artin group}, a {\em trace group}, a {\em semi-free group} or a {\em graph group}) is a group given by a finite presentation
$\la X| R\ra$, where $X$ is the vertex set of a simple graph $\G$ and $R$ is
the set consisting of precisely those commutators $[x,y]$ of elements 
of $X$ such that $x$ and $y$ are joined by an edge of $\G$. (A simple graph
is one without multiple edges or self-incident vertices. Our convention
is that $[x,y]=x^{-1}y^{-1}xy$.) 

Initial work by Servatius \cite{servatius89} and Laurence \cite{Laurence95}
established a finite generating set for the automorphism group of a partially
commutative group. In a resurgence of interest over the last few years 
considerably more has been discovered: for 
example, Bux, Charney, Crisp and Vogtmann 
\cite{CharneyCrispVogtmann07, CharneyVogtmann09, BuxCharneyVogtmann09} 
have shown that these groups are virtually torsion-free and have finite 
virtual cohomological dimension and  Day has shown how peak reduction techniques may
be used on certain subsets of the generators and thereby has given
a presentation for the automorphism group \cite{Day09}. 
Moreover these groups have a very rich subgroup structure: Gutierrez, 
Piggott and Ruane \cite{GutierrezPiggotRuane07} have constructed a semi-direct
 product decomposition 
for the more general case of  automorphism groups of graph products of 
groups. Duncan, Remeslennikov and Kazachkov \cite{DKR5} describe several
arithmetic subgroups of the automorphism group of a partially commutative
group; while different arithmetic subgroups have been found by Noskov \cite{Noskov10}. 
Under certain conditions on the graph $\G$,  Charney and Vogtmann have shown 
\cite{CharneyVogtmann10} that the Tits alternative holds for the outer 
automorphism group of $G(\G)$ and moreover 
   Day \cite{Day10} has shown that in all cases this group  contains 
either a finite-index nilpotent subgroup or a
non-Abelian free subgroup.  Minasyan has
shown \cite{Minasyan} that partially commutative groups are conjugacy
separable, from which (loc. cit.) it follows that their outer automorphism
groups are residually finite. By reduction to the compressed word problem in $G(\G)$, 
Lohrey and Schleimer have shown that  the word problem
in $\Aut(G(\G))$ has polynomial time complexity \cite{LohreySchleimer}. 
 Charney and Faber \cite{CharneyFarber}, and subsequently Day \cite{Day11}, 
have studied 
automorphism groups of partially 
commutative groups associated to random graphs, of Erd\"os-R\'enyi type, 
 and found bounds on the edge probabilities so that,  
with probability tending to one as the number of vertices tends to $\infty$, 
such groups have finite outer automorphism groups. 

In this paper we  continue the investigation of \cite{DKR5} into
the structure of the automorphism group and its subgroups.  
We  introduce several standard
automorphisms of a partially commutative group and describe how
an arbitrary automorphism may be decomposed  as a product of these
standard automorphisms. This reduces the study of the automorphism
group to the study of subgroups generated by particular types of 
standard automorphism. 
 We  then define  subgroups of a geometric character 
  and use these to analyse
the group structure. Note that if $R$ is the ring of integers or a field
of characteristic $0$ and $G$  is a partially commutative
group in the class of $2$-nilpotent $R$-groups,  the structure of 
$\Aut(G)$ has been completely described, by Remeslennikov and Treier
\cite{RemeslennikovTreier}; and  decomposes as  an extension of 
an Abelian group by a 
 subgroup of $\GL(n,R)$. 

With this program in mind  we   define certain 
automorphisms, based on the combinatorial
properties of the graph $\G$,  and these  form our stock of standard 
automorphisms. The idea   is to emulate the theory of automorphisms of 
algebraic  and Chevalley groups. There is extensive literature
 on abstract isomorphisms of
the classical linear groups and algebraic groups, 
over fields and special classes of
rings, in which the fundamental results are  theorems on splitting
of  arbitrary automorphisms into special automorphisms (such as algebraic,
semialgebraic, simple, central, etc.) \cite{BorelTits73,Golubchik92} and
representations of the group of automorphisms as products of
 the corresponding subgroups. Similar  splitting theorems have also 
been established for  Chevalley groups.   
Steinberg \cite{steinberg60} and Humphreys \cite{Humphreys69} established such
results for Chevalley groups over 
fields and Bunina \cite{bunina10} has defined several special types
of automorphism (Central, Ring, Inner and Graph automorphism) and shown  that,
if $G$ is a Chevalley group over a commutative local ring (subject
to certain restrictions) then an arbitrary automorphism of 
$G$ decomposes as a product of such automorphisms.  Moreover similar
results have been obtained for Kac-Moody groups (see \cite{caprace} and 
the references therein). 

In \cite{DKR5} we obtained certain decomposition  theorems  
for the automorphism group of a partially commutative grouup which we extend in 
this work.  We use the 
 orthogonalisation operator $Y^\bot$  
and a closure operator
$\cl(Y)$ both defined on subsets $Y \subseteq X$ in \cite{DKR3}. 
In particular, for $x\in X$, the set $\{x\}^\bot$ consists 
 of all vertices incident to $x$, as 
well as $x$ itself: so is the ``star'' of $x$; and the closure $\cl(\{x\})$ of
$\{x\}$ is the intersection of the stars of all elements of $\{x\}^\bot$.   
The closure operator
$\cl$ defines a lattice of ``closed'' subsets $\cL=\cL(X)$ of $X$ and 
the results of \cite{DKR5} were obtained by considering the 
action of automorphisms on this lattice. 
In this paper we consider a similar lattice $\cK=\cK(X)$ 
of ``admissible'' subsets
of $X$ and the action of automorphisms on $\cK$. 
 In particular there is an admissible set $\ad(x)$ associated to each
element of $X$: namely the intersection of the stars of all elements
of $\{x\}^\bot\bs \{x\}$. 
We consider the following
subgroups of the automorphism group 
$\Aut(G)$ of the partially commutative group $G$.
\begin{itemize}
\item The subgroup $\Aut^\G(G)$ of automorphisms induced by automorphisms
of the graph $\G$ (Definition \ref{defn:graphaut}). 
\item The subgroup $\Aut^\G_{\cmp}(G)$ of $\Aut(G)$, which is
isomorphic to the automorphism group of the graph $\G^{\cmp}$, the compressed
graph of $\G$ (Definition \ref{defn:graphaut}). 
\item The subgroup $\Conj (G)$ of basis-conjugating automorphisms: those
which map each generator $x$ to $x^{f_x}$, for some $f_x\in G$ (Definition 
\ref{defn:conj}).
\item The subgroup $\NConj(G)$ of $\Conj(G)$, of automorphisms such that, 
for all $x\in X$, there exists $g_x$ in $G$ with the property that
$z$ maps to $z^{g_x}$, for all $z\in \ad(x)$
(Definition \ref{defn:NConj}).
\item The subgroup $\St(\cK)$, elements of which  stabilise subgroups
generated by subsets $A$, where $A$
is an element of the lattice $\cK$ (Definition \ref{defn:st}).     
\item  The subgroup $\cSt(\cK)$, elements of  which map each subgroup $\la A\ra$,
where $A\in \cK$, 
to $\la A\ra^{g_{A}}$, for some $g_A\in G$ (Definition \ref{defn:stconj}).
\item Various subgroups  illustrated in Figure 
\ref{fig:subgplattice} below.
%Subgroups $\VConj(G), \NConj(G), \aConj(G), \iConj(G),
%\sConj(G)$ and $\CConj(G)$of $\Conj(G)$.
\end{itemize}
(Several of these groups are well-known: some are defined for example
in \cite{Laurence95} and others  in \cite{DKR5}.)

The first step in our decomposition of $\Aut(G)$ is to separate
out the automorphisms induced by automorphisms of the compressed graph.

\begin{varthm}[\textrm{\textbf{Theorem \ref{theorem:Stconj}}}] 
The group $\Aut(G)$
can be decomposed into the internal semi-direct product of the subgroup
$\cSt(\cK)$ and the finite subgroup $\Aut^\G_\cmp(G)$,  
 i.e. 
\[\Aut(G)=\cSt(\cK)\rtimes \Aut^\G_\cmp(G).\]
\end{varthm}
This theorem essentially reduces the problem of studying $\Aut (G(\G))$ to the study
of the group $\cSt(\cK)$. 

We may also decompose the automorphism group
using the connected components of $\G$. If $\G$ has connected 
components $\G_1, \ldots ,\G_n$ then the partially commutative group
determined by $\G$ is the free product of those determined by the $\G_i$. 
The group of
automorphisms of a free product of groups has been  completely  described (from the
point of view of generators and defining relations) in
papers \cite{FR1,FR2,Gilbert87,CollinsGilbert}. We specialise these
results to the case under consideration to give generators and relations
for the full automorphism group in terms of presentations for the 
automorphism groups of the factors. 

However, here we encounter the first of 
 two main obstructions to identifying the 
structure of $\Aut(G(\G))$. 
 The problem arises when there are isolated vertices in the graph $\G$ (vertices of 
valency zero). In this case the automorphism group does not have
a natural semi-direct product decomposition in terms of  the 
automorphism groups of the factors. Nonetheless, in the special case
where there are no isolated vertices the quoted results give  
the following theorem, where $\LInne$ is a subset of $\Conj(G)$, which
 is empty unless $\G$ is disconnected and is 
defined in Definition \ref{defn:ourgens}, 
$\G$ has connected components $\G_1,\ldots ,\G_n$  and 
$G_i=G(\G_i)$. 
\begin{varthm}[\textrm{\textbf{Theorem \ref{theorem:FRker}}} 
\normalfont{({\it cf.} \cite{CollinsGilbert}, Theorem C])}]
Suppose that no component of $\G$ is an isolated vertex. 
Define $\bar G=G_1\times \cdots \times G_n$ and $\FR(G)= \la \LInne\ra$. 
Then 
$\FR(G)$ is the kernel of the
canonical map from $\Aut(G)$ to $\Aut(\bar G)$. Moreover $\FR(G)$ has a normal 
series 
\[
1<P_{n-1}<\cdots <P_2<\FR(G)
\]
such that, setting $\FR_i(G)=\FR(G)/P_i$, 
\be[(i)]
\item $\FR(G)=P_i \rtimes \FR_i(G)$, 
\item $\FR_i(G)=\FR(G_1\ast \cdots \ast G_i)$ and
\item all the $P_i$ are finitely generated.
\ee
%\begin{theorem}\label{th4}[Theorem 2.26]
\end{varthm}
The last theorem reduces analysis of the structure of $\Aut(G)$, in  the 
case when $\G$ has no isolated vertices, to analysis of $\Aut(G(\G_i)), i = 1, \ldots n$,
and 
of the Fouxe-Rabinovitch kernel $\FR(G)$.
%%%%%%%%%%%%%%%%%%%%%%%%%%%%%%%%%%%%%%

In the light of these results we may often reduce to the study
of $\cSt(\cK)$ where $\G$ is a connected graph. 
First of all we have the following theorem.  
\begin{varthm}[\textrm{\textbf{Theorem \ref{theorem:nconjnorm}}}]
%\begin{theorem}\label{th3} 
The subgroup $\NConj(G)$ is a normal subgroup of $\cSt(\cK)$ and
therefore of $\Conj(G)$.
\end{varthm}

The next step might appear to be to give an affirmative answer to the 
following question.
\begin{varthm}[\textrm{\textbf{Question \ref{que:aut*decomp}}}] Let $\G$ be a connected graph. Is 
$\cSt(\cK)=\St(\cK)\NConj(G)$?  
\end{varthm}
However as examples show the answer to this question is negative; and 
this brings us to the 
 second major obstruction to the description of the structure of $\Aut(G)$. 
This is the existence of vertices $x$ and $y$ such that $\{y\}^\perp\bs\{y\}$ is
contained in $\{x\}^\perp$. When this occurs we say that 
$x$ dominates $y$ (Definition \ref{defn:isols}).  
If there are no  such vertices $x$ and  $y$ in $\G$
  then we obtain a clear description
of the structure of $\cSt(\cK)$ in terms of $\Conj(G)$ and the stabiliser
$\St(\cL)$ of the lattice of closed sets (studied in detail in \cite{DKR5}).  
 In fact, in this case $\Conj(G)=\NConj(G)$ and $\St(\cK)=\St(\cL)$.
\begin{varthm}[\textrm{\textbf{Theorem \ref{theorem:nolocisol}}}]
%\begin{theorem}\label{th5} 
The following are equivalent for a graph $\G$. 
\be[(i)]
\item $G$ has no dominated vertices.
\item  
$\cSt(\cK)=\NConj(G)\rtimes \St(\cL).$
\item  
$\cSt(\cK)=\Conj(G)\rtimes \St(\cL).$
\item
 $\cSt(\cK)=\Conj(G)\rtimes \St(\cK)$. 
\ee
\end{varthm}

Therefore, in the case where there are no dominated vertices the structure of $\cSt(\cK)$ is determined by the
structure of $\NConj(G)$ and $\St(\cL)$ and, as we have shown in \cite{DKR5},
$\St(\cL)$ is an arithmetic group for which we have a complete structural
decomposition. 

We conclude by establishing conditions under which 
$\cSt(\cK)=\Conj(G)\St(\cK)$,
even though there are dominated vertices (and this product may not 
be semi-direct).  
In Section \ref{section:bal} we introduce balanced graphs, which include those without
dominated vertices, and prove the following theorem.

\begin{varthm}[\textrm{\textbf{Theorem \ref{theorem:balstcon}}}]
Let $\G$ be a connected graph and $G=G(\G)$. Then $\cSt(\cK)=\St(\cK)\Conj(G)$
if and only if $\G$ is a balanced graph. 
\end{varthm}

Therefore in many cases the structure of $\cSt(\cK)$ is determined
by the structure of  $\St(\cK)$, $\Conj(G)$ and $\NConj(G)$.  
In this paper we 
find generators for (most of) the subgroups discussed above, as well as those
appearing in the diagram below, 
 establish some of their basic properties and 
investigate the decomposition of the automorphism group  of $G$,  
in terms of these subgroups, in the simplest cases, leaving
the case where there are dominated vertices, and the 
structure of $\St(\cK)$ to later papers.

The structure of the paper is as follows. In Section \ref{sec:preliminaries}
we introduce partially commutative groups, admissible sets, the lattices
$\cK$ and $\cL$ and describe an ordering on the vertex set of a graph 
induced from the lattice $\cK$. In Section \ref{sec:generators} we 
turn to the automorphism groups of partially commutative groups,   
show how they may be decomposed using the subgroups $\Aut^\G(G)$ and
$\Aut^\G_\cmp(G)$, mentioned above, and, using the results of Fouxe-Rabinovitch
and Gilbert and Collins, show how the connected components of the 
graph $\G$ determine generators and relations 
for the automorphism group of $G(\G)$. In Section \ref{sec:conj} we 
define a collection of subgroups of the basis-conjugating automorphism
subgroup $\Conj(G)$, to be used to decompose   $\cSt(\cK)$, and show how these relate to 
each other (see Figure \ref{fig:subgplattice} below). In Section \ref{section:st}
we consider the subgroups $\St(\cK)$ and $\cSt(\cK)$, show that 
$\NConj(G)$ is normal in $\cSt(\cK)$, describe the intersection 
$\St(\cK)\cap \NConj(G)$ and give an example to show that, in general,
$\cSt(\cK)\neq \St(\cK)\Conj(G)$. Finally in Section \ref{section:bal} we 
define balanced graphs and show that the equality 
$\cSt(\cK)=\St(\cK)\Conj(G)$ holds for $G=G(\G)$ if and only if 
$\G$ is balanced. In the version of the paper on \texttt{arxiv} we 
include an appendix with details of the construction of a presentation
of $\Aut^\G_\cmp(G)$ and the proof of the Theorem \ref{prop:presentation}.

For reference purposes Figure \ref{fig:subgplattice} shows 
a diagram of the lattice of the main
 subgroups
which we define in the paper. The figure covers the 
case where $\G$ has no isolated vertices (that is vertices of valency zero).
If $\G$ does have isolated vertices then the subgroups $\NConj(G)$ and
$\NConj(G)\cap \St(K)$ are removed from the diagram, which otherwise
remains the same. (In this case $\NConj(G)=\Inn(G)$.)   
Subgroups $\aConj(G)$, $\VConj(G)$, 
$\sConj(G)$ and $\CConj(G)$ are defined in definitions 
\ref{defn:aConj}, \ref{defn:VConj}, \ref{defn:singular} and 
\ref{defn:collectedconj}, respectively.   
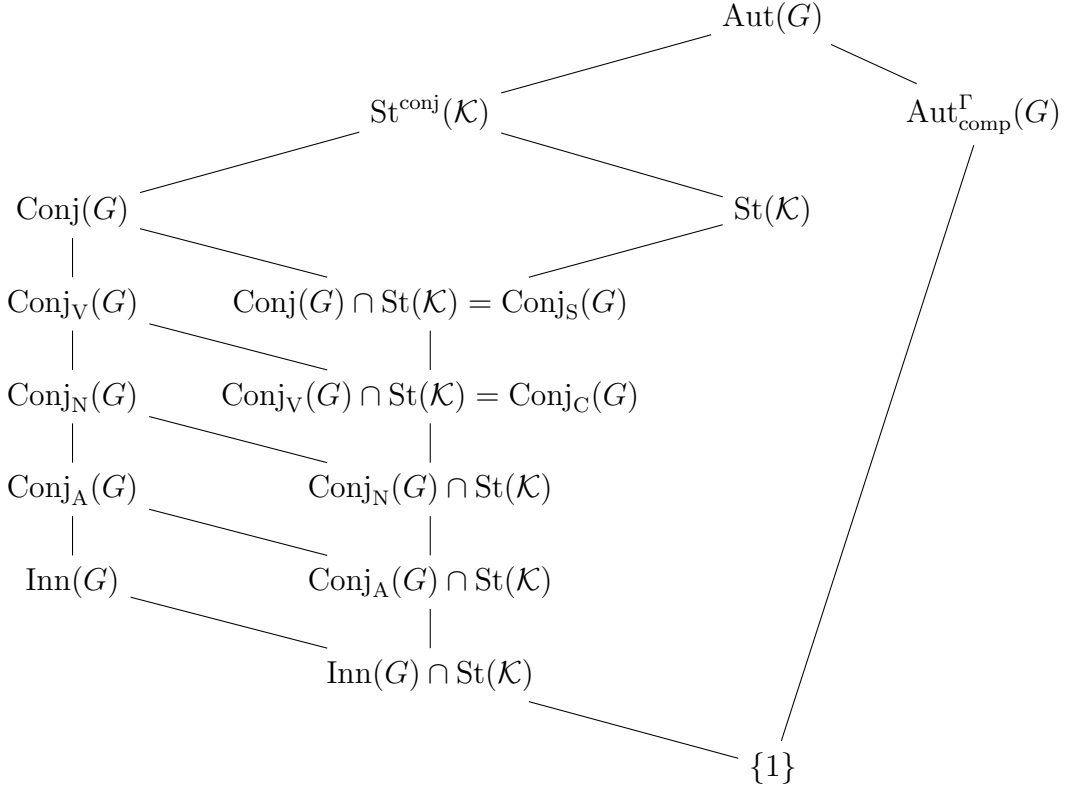
\begin{figure}
\begin{center}
\begin{tikzpicture}
  \matrix (m) [matrix of math nodes,row sep=1.3em,column sep=2em,minimum width=2em]
  {
     & &\Aut(G) &  \\
     &\cSt(\cK)& & \Aut^\G_\cmp(G)\\
\Conj(G)& &\St(\cK)& \\
\VConj(G)&\Conj(G)\cap \St(\cK)=\sConj(G)&&\\
\NConj(G)&\VConj(G)\cap \St(\cK)=\CConj(G)&&\\
\aConj(G)&\NConj(G)\cap \St(\cK)&&\\
\Inn(G) & \aConj(G)\cap\St(\cK)&&\\
&\Inn(G)\cap \St(\cK)&&\\
&&\{1\}&\\
};
  \path%[-stealth]
   (m-1-3) edge  (m-2-4)
   (m-2-2) edge  (m-1-3)
           edge  (m-3-1)
           edge  (m-3-3)
   (m-2-4) edge  (m-9-3)
   (m-3-1) edge  (m-4-1)
           edge  (m-4-2)
   (m-3-3) edge  (m-4-2)
   (m-4-1) edge  (m-5-1)
           edge  (m-5-2)
   (m-5-1) edge  (m-6-1)
           edge  (m-6-2)
   (m-6-1) edge  (m-7-1)
           edge  (m-7-2)
   (m-7-1) edge  (m-8-2)
   (m-8-2) edge  (m-9-3)
   (m-4-2) edge  (m-5-2)
   (m-5-2) edge  (m-6-2)
   (m-6-2) edge  (m-7-2)
   (m-7-2) edge  (m-8-2)
;  
\end{tikzpicture}
\end{center}
%%%%%%%%%%%%%%%%%%%%%%%%%%
\caption{Subgroups of $\Aut(G(\G))$, where $\G$ has no isolated vertices}\label{fig:subgplattice}
\end{figure}

%%% Local Variables: 
%%% mode: latex
%%% TeX-master: "aut2"
%%% End: 

\section{Preliminaries}\label{sec:preliminaries}
Graph will mean undirected, finite, simple graph throughout this paper. 
A subgraph $S$ of a graph $\G$ is called a {\em full} subgraph 
if vertices $a$ and $b$ of $S$ are joined by an edge of $S$ 
whenever they are joined by an edge of $\G$.  
 If $S$ is a subset
of $V(\G)$ we shall write $\G(S)$ for the full subgraph of $\G$ with
vertices $S$.

If $a$ and $b$ are elements of a group then  $[a,b]$ denotes $a^{-1}b^{-1}ab$. If $A$ and $B$ are subsets of a group then
$[A,B]$ denotes $\{[a,b]:a\in A, b\in B\}$.

For the remainder of the paper let
 $\G$ be a finite, undirected, simple graph. Let
$X=V(\G)=\{x_1,\dots, x_n\}$ be the set of vertices of $\G$ and let
$F(X)$ be the free group on $X$. Let
\[
R=\{[x_i,x_j]\in F(X)\mid x_i,x_j\in X \textrm{ and there is an edge of }
\G \textrm{ joining }
x_i \textrm{ to } x_j \}.
\]
We define the {\em partially commutative group} with ({\em commutation}) 
{\em graph}
$\G$ to be the group $G(\G)$ with presentation $ \left< X\mid
R\right>$. 
When the underlying graph is clear from the context we
write simply $G$.

The subgroup generated by a subset $Y\subseteq X$ is called a {\em canonical parabolic subgroup} of $G$
and denoted  $G(Y)$. This subgroup is equal to the partially commutative group
with commutation graph the full subgraph of $\G$ with vertices $Y$ 
(see \cite{baudisch77} or \cite{EKR}).

By a \emph{word} over $X$ is meant an element of the free monoid $(X\cup X^{-1})^\ast$.
We identify elements of $F(X)$ with reduced words (that is those have
no subwords of the form $x^\e x^{-\e}$, where $x\in X$ and $\e=\pm 1$). The
\emph{length} of a word $w$ is its length as an element of $(X\cup X^{-1})^\ast$
and is denoted $|w|$. 
Denote by $\lg(g)$ the minimum of the lengths of words that represent the
element $g$ of $G(X)$. If $w$ is a word representing $g$ and $w$ has length
$\lg(g)$ we call $w$ a {\em minimal form} for $g$. If $w$ is a 
minimal form for some element of $G$ then we say that $w$ is a 
\emph{geodesic word}.  
When the meaning is clear we shall say that $w$ is a 
minimal element of $G$ when we mean that $w$ is a minimal form of an
element of $G$.
We say that   $h \in G$
 is {\em cyclically minimal} if and only if
\[
\lg(g^{-1}hg) \ge \lg(h)
\]
for every $g \in G$.

The  \emph{support} of a word $w$ over $X$ is the set of elements
of $X$ such that $x$ or $x^{-1}$ occurs in $w$. If $u$ and $v$ are 
minimal forms of an element $g\in G$ then both $u$ and $v$ have 
the same support (see for example \cite{EKR}). Therefore we may define
the \emph{support} $\al(g)$ of an element $g\in G$ to be the support
of a minimal form of $g$. If $w\in G$ define $A(w)=G(Y)$, 
where $Y$ is the set of elements
of $X\backslash \al(w)$ which commute with every element of $\al(w)$. If
$S\subseteq G$ then we define $A(S)=\cap_{w\in S}A(w)$.

From now on we regard words as representing elements of $G$, 
so when we write $u=v$, where $u$ and $v$ are words, we mean
that $u$ and $v$ represent the same element of $G$. 
We write $u\circ w$ to express the fact that
$\lg(uw)=\lg(u)+\lg(w)$, where $u,w\in G$.
Let $u$ and $w$ be elements of $G$. We say that $u$ is a {\em left }
({\em right}) {\em divisor} of $w$ 
if there exists $v\in G$ such that $w= u
\circ v$ ($w=  v\circ u$).
We partially order the set of all left (right) divisors of a word $w$ as
follows. We say that $u_2$ is greater than $u_1$ if and only if
$u_1$ is a left (right) divisor of $u_2$.
It is shown in \cite{EKR} that, for any $w\in G$ and $Y\subseteq
X$, there exists a  unique maximal left  divisor   of $w$ 
which belongs to the subgroup $G(Y)$ of $G$: called 
the {\em greatest left divisor} $\gd^{l}_Y(w)$ of $w$ in $Y$.
  The {\em greatest right divisor} of $w$ in $Y$ is defined analogously.
 %We omit the indices when no ambiguity occurs. 

The {\em non-commutation} graph of 
the partially commutative  group $G(\G)$ is the graph $\Delta$,
dual to $\Gamma$, with vertex set $V(\Delta)=X$ and 
an edge connecting $x_i$ and $x_j$ if and only if $\left[x_i, x_j
\right] \ne 1$. The graph $\D$ is the union of its connected
components $\D_1, \ldots , \D_k$ and if $u$ and $v$ are  words
such that $\nu(u)\subseteq \D_i$ and $\nu(v)\subseteq \D_j$, with 
$i\neq j$, then  $u$ and $v$ represent commuting elements of $G$. 
 Thus, if
the vertex set of 
$\D_j$ is $I_j$ and $\G_j=\G(I_j)$, the full subgraph of $\G$ on $I_j$,  
then 
$G=
G(\G_1) \times \cdots \times G(\G_k)$.

Let $g\in G$ and suppose that  the full subgraph $\Delta (\al(g))$ 
of $\Delta$ with vertices $\al(g)$ has connected components
$\D^\prime_1,\ldots, \D^\prime_l$ and let the vertex set of $\D^\prime_j$ be $I^\prime_j$.
If $w$ is a minimal form of $g$ then, 
since $[I^\prime_j,I^\prime_k]=1$, 
we can factor $w$ as a product of commuting words,
$w=w_1\circ \cdots \circ w_l$, where $w_j\in G(\G(I^\prime_j))$, so $[w_j,w_k]=1$
for all $j,k$.
If $g$ is cyclically minimal then
we call this expression for $g$ a
 {\em block decomposition} of $g$ and
say $w_j$ is a {\em block} of $g$, for $j=1,\ldots ,l$. Thus $w$ itself is
 a block if and only if  $\D(\al(w))$ is connected. 
  Moreover it follows (see \cite{EKR} for example)  that if $g$ has
another minimal form $u$ with a block  decomposition 
$u=u_1\circ \cdots\circ u_k$ then $k=l$ (as $\D(\al(u))=\D(\al(g))=
\D(\al(w))$) and after reordering the
$u_s$'s if necessary $w_s=u_s$, for $s=1,\ldots ,l$.

In general let 
$v$ be an element of $G$, not necessarily cyclically minimal. We may write 
$v=u^{-1}\circ w \circ u$, where $w$ is cyclically minimal and then
$w$ has a block decomposition $w=w_1\cdots w_l$, say. 
We call the expression $v=w_1^u \cdots w_l^u$ 
a  {\em block decomposition}
of $v$ and say that $w_j^u$ is a {\em block}
of $v$, for $j=1,\ldots , l$. Note that this definition is slightly different
from that given in \cite{EKR}. 

The {\em centraliser} of a subset $S$ of  $G$  is 
\[C(S)=C_G(S)=\{g\in G: gs=sg,
\textrm{ for all } s\in S\}.\] 
An element $g \in G$ is called a {\em root element} 
if $g$ is not a proper
power of any element of $G$. If $h=g^n$, where $g$ is a root element
and $n\ge 1$, then $g$ is said to be a {\em root} of $h$.
As shown in \cite{baudisch77} (and also \cite{dk93b}) 
every element of the partially commutative group
 $G$ has a unique root, which
we denote $r(g)$. 
Let $w$ be a cyclically
minimal element of $G$ with block decomposition $w=w_1\cdots w_k$
and let $v_i=r(w_i)$. Then from \cite{baudisch77} (and also \cite[Theorem 3.10]{dk93b}),
\begin{equation}\label{eq:centraliser}
C(w)=\langle v_1\rangle \times \cdots \times \langle v_k
\rangle\times A(w).
\end{equation}

 The following lemma will be useful. 
\begin{lemma}\label{lem:comm}
Let $x,y\in X$ and $f,g\in G$ such that $[x,y]=[x^f,y^g]=1$ and 
\be[(i)]
\item\label{it:comm1}
 $x^f=f^{-1}\circ x\circ f$, $y^g=g^{-1}\circ y\circ g$ and 
\item\label{it:comm2} $\gd_X^r(f,g)=1$.
\ee
Then $[\al(f),\al(g)]=\{1\}$ and $[f,y]=[g,x]=1$.
\end{lemma}
\begin{vnr}
\begin{proof}
By hypothesis $[x^{fg^{-1}},y]=1$. 
From condition \ref{it:comm2} $fg^{-1}=f\circ g^{-1}$ and 
from \cite{DKR2}, Lemma 2.3, there exist 
$a,b,u \in G$ such that $fg^{-1}=a\circ b\circ u$, $x^{fg^{-1}}=u^{-1}\circ
 x \circ u$, $a=x^n$, for some $n\in \ZZ$, and $[b,x]=1$. From 
\eqref{eq:centraliser},   
$b=x^m\circ c$, for some $c\in A(x)$; so $a\circ b=x^k\circ c$, for
some $k\in \ZZ$. 
 Now $a\circ b$ is a left divisor of $f\circ g^{-1}$ and it follows
from condition \ref{it:comm1} that $u=u_1\circ u_2$, where 
$f=u_1$ and  $g^{-1}=a\circ b\circ u_2=x^k\circ c\circ u_2$, for
some words $u_1, u_2$ with $[u_1, x^k\circ c]=1$. 
From condition \ref{it:comm1} again, $k=0$, $a=1$ and $b=c\in A(x)$. 
 
From \cite{DKR2} Corollary 2.6, $[x,y]=[x^u,y]=1$ implies
that $u\in C(y)$. As $u$ is a right divisor of $f\circ g^{-1}$ it
follows, from condition \ref{it:comm1} again, that $u_2=1$; so  
$f= u_1=u$, and $g=b$.  It follows, using condition \ref{it:comm2} once more, 
that $[\al(b),\al(u)]=1$ and this gives the result.
\end{proof}
\end{vnr}

\subsection{Admissible sets}\label{subsec:admiss}

In this section we establish some properties of graphs
 which we shall apply to the study of the automorphism group
of $G(\G)$.   
If $x$ and $y$ are vertices of a  graph $\G$ then we define the {\em distance} 
$d(x,y)$ from $x$ to $y$ to be the minimum of the lengths of all paths 
from $x$ to $y$ in $\G$. 
Given a subset
$Y$ of $X$  the {\em orthogonal complement}
of $Y$ is defined to be
\[Y^\perp=\{u\in X|d(u,y)\le 1, \textrm{ for all } y\in Y\}.\]
For a set $\{x\}$ of one element we write $x^\perp$ instead of $\{x\}^\perp$
 and in general often write $x$ in place of $\{x\}$. 
For any set $Y\subseteq X$ we write $Y^{\perp \perp}$ for $(Y^\perp)^\perp$. 
By convention we set $\nul^\perp=X$. 

We define the {\em closure} of $Y$ to be
$\cl(Y)=Y^{\perp\perp}$.
The closure operator in $\G$ satisfies, among others, the properties that
$Y\subseteq \cl(Y)$,
$\cl(Y^\perp)=Y^\perp$ and
$\cl(\cl(Y))=\cl(Y)$ \cite[Lemma 2.4]{DKR3}.  Moreover if $Y_1\subseteq Y_2\subseteq X$
then $\cl(Y_1)\subseteq \cl(Y_2)$.
\begin{defn}
A subset $Y$ of $X$ is called {\em closed} {\rm(}with respect to $\G${\rm)}
if $Y=\cl(Y)$.
Denote by $\cL=\cL(\G)$ the set of all closed subsets of $X$.

For non-empty 
$Y\subseteq X$ define $\ad(Y)=\cap_{y\in Y}(y^\perp\backslash y)^\perp$. 
Define $\ad(\nul)=X$. 
Subsets of the form $\ad(Y)$, where $Y\subseteq X$ are called 
{\em admissible} sets.  Let $\cK=\cK(\G)$ denote the set of admissible subsets of $X$. 
\end{defn}

Properties of the set $\cL$ are considered in detail in \cite{DKR3} and applied
to the study of centralisers and 
automorphisms of partially commutative groups in \cite{DKR2}, \cite{DKR4} and 
\cite{DKR5}.
 We shall see, in Section \ref{sec:generators}, that distinct elements
$x$ and $y$ of $X$, such that $x^\perp\bs x\subseteq y^\perp$, give
rise to a particular type of automorphism of $G$. 
 The motivation for the definition of an admissible set is then clear
from the first part of the following lemma. 
\begin{lemma}\label{lem:admot}
For all $x\in X$, 
\be[(i)]
\item \label{it:admot1}
the set  $\ad(x)=\{y\in X: x^\perp\bs x\subseteq y^\perp\}$ and 
\item \label{it:admot2}
$y\in \ad(x)$ if and only if $\cl(y)\subseteq \ad(x)$, for all $y\in X$.  
\ee
\end{lemma}
\begin{proof}
\be[(i)]
\item
$y\in \ad(x)$ if and only if  $[y,v]=1$, for all $v\in x^\perp\bs x$, if
and only if $x^\perp\bs x\subseteq y^\perp$. 
\item
For all $y\in X$ we have $y\in \cl(y)$, so the ``if'' clause 
follows. On the other hand if  
$y\in \ad(x)$ then, from \ref{it:admot1}, $x^\perp\bs x\subseteq y^\perp$;
so $y^{\perp \perp}\subseteq (x^\perp\bs x)^\perp$, as required.
\ee
\end{proof}

\begin{exam}\label{ex:adex}
In the graph $\G$ of Figure \ref{subf:adexa}
\begin{itemize}
\item
$\ad(a)=\{b,c,d,e,g,h,i\}^\perp=\{a\}=\cl(a)$; 
 \item
$d^\perp=g^\perp=\{a,c,d,e,g,h\}$ and $\ad(d)=\ad(g)=\{a,d,g\}=\cl(d)=\cl(g)$;
\item
$\cl(b)=\{a,b,c,h\}^\perp=\{a,b\}$, $\cl(i)=\{a,c,h,i\}^\perp=\{a,i\}$, 
$i^\perp\bs i=b^\perp\bs b$ and 
$\ad(i)=\ad(b)=\{a,c,h\}^\perp=\{a,b,d,g,i\}= \cl(b)\cup \cl(d)\cup \cl(i)$;
\item
$\cl(c)=\{a,c\}$, $\cl(h)=\{a,h\}$, $c^\perp\bs c=h^\perp\bs h$ and 
$\ad(c)=\ad(h)=\{a,c,h\}=\cl(c)\cup \cl(h)$; 
\item 
$\ad(e)=\{a,d,f,g\}^\perp=\{e\}=\cl(e)$ and 
\item
$\cl(f)=\{e,f\}^\perp=\{e,f\}$ and 
$\ad(f)=\{e\}^\perp=\{a,d,e,f,g\}=\cl(d)\cup \cl(f)$.  
\end{itemize}
\end{exam}
\begin{figure}
\begin{center}
\psfrag{a}{$a$}
\psfrag{b}{$b$}
\psfrag{c}{$c$}
\psfrag{d}{$d$}
\psfrag{e}{$e$}
\psfrag{f}{$f$}
\psfrag{g}{$g$}
\psfrag{h}{$h$}
\psfrag{i}{$i$}
\psfrag{aa}{$\ad(a)$}
\psfrag{ba}{$\ad(b)$}
\psfrag{ca}{$\ad(c)$}
\psfrag{da}{$\ad(d)$}
\psfrag{ea}{$\ad(e)$}
\psfrag{fa}{$\ad(f)$}
\psfrag{ga}{$\ad(g)$}
\psfrag{ha}{$\ad(\nul)$}
\psfrag{ia}{$\ad(X)$}
\subfloat[A graph $\G$\label{subf:adexa}]{
\includegraphics[scale=0.4]{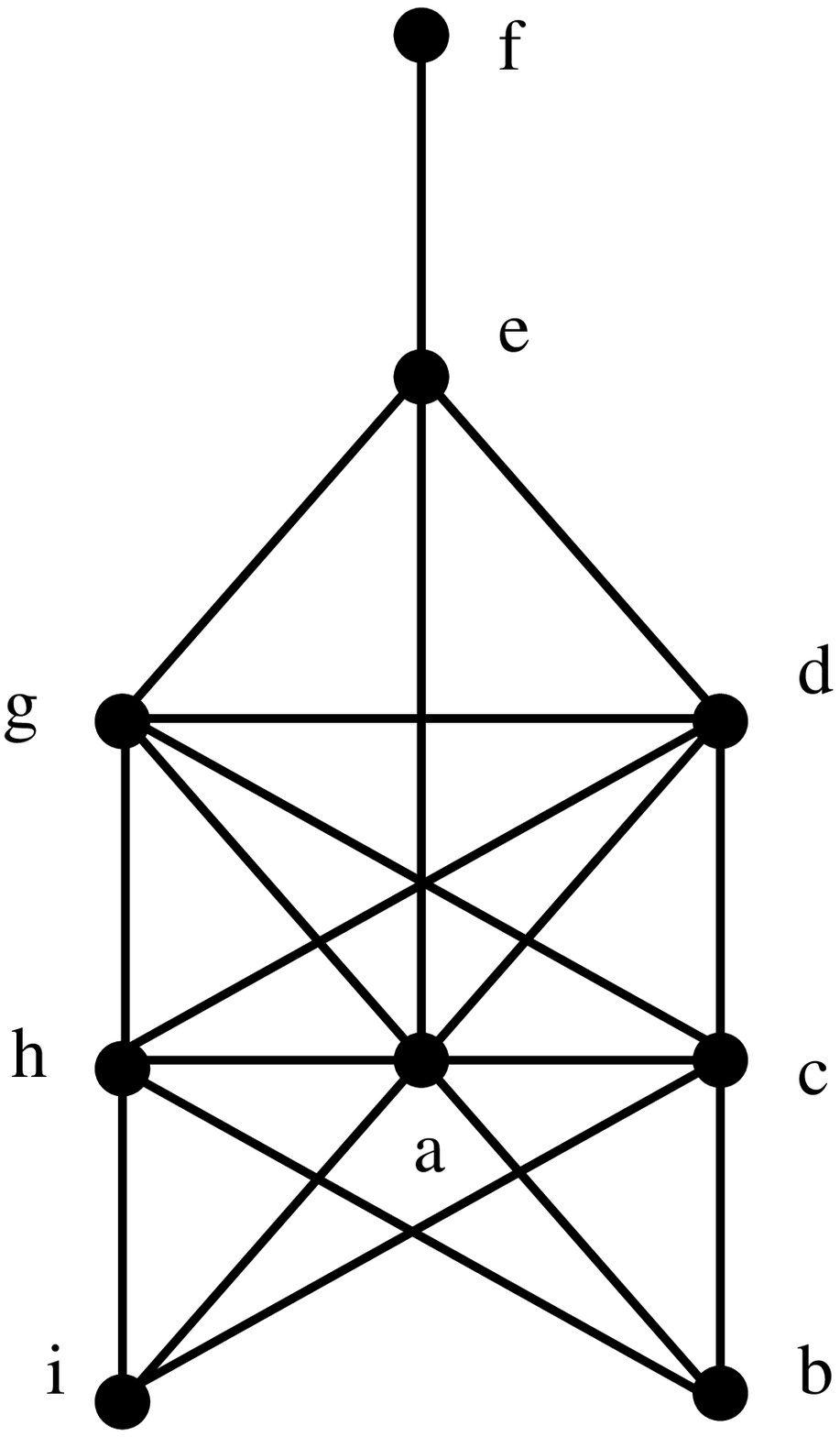}
}
\quad\quad
\subfloat[The lattice $\cK(\G)$\label{subf:adexb}]{
\includegraphics[scale=0.4]{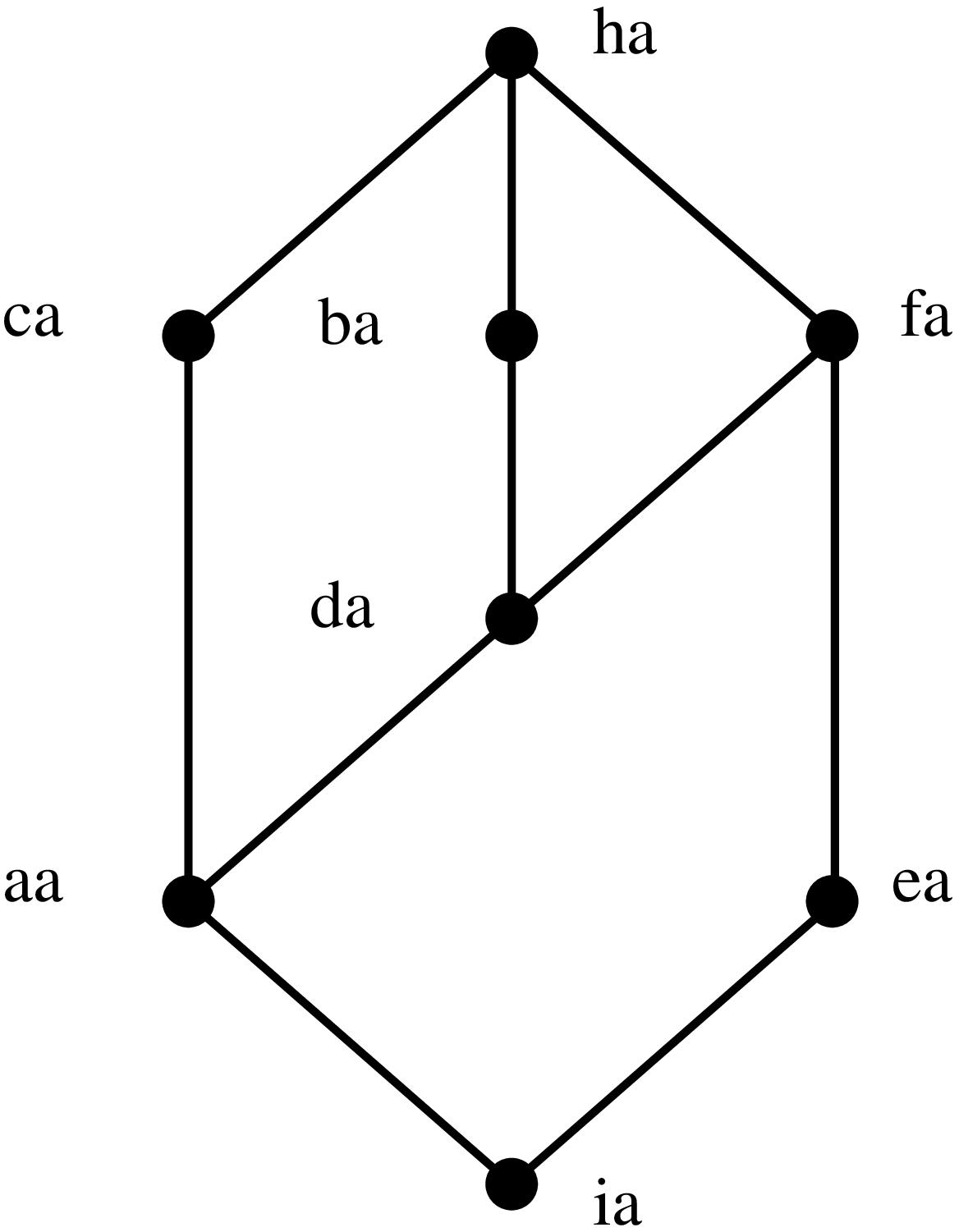}
}
\end{center}
\caption{A graph and it's lattice of admissible sets}\label{fig:adex}
\end{figure}
For sets $U, V$ we write $U<V$ to indicate that $U \subseteq V$ and 
$U\neq V$. 
 A subset $Y$ of $X$ is called a 
{\em simplex} if the full subgraph of $\G$ with vertices $Y$ is isomorphic 
to  a
complete graph.
\begin{lemma}\label{lem:ad0}
 For $x\neq z\in X$  
and subsets $U$ and $V$ of $X$ the 
following hold. 
\be[(i)]
\item\label{it:ad1} If $U\subseteq V$ then $\ad(V)\subseteq \ad(U)$.
\item\label{it:ad10} $\ad(U)\cap \ad(V)=\ad(U\cup V)$. 
\item\label{it:ad2} $\cl(x)=\ad(x) \cap x^\perp$ so 
$\ad(x)=\cl(x)$ if and only if $\ad(x) \subseteq x^\perp$.
\item\label{it:ad3} $x^\perp\subseteq \ad(x)$ if and only if $x^\perp$ 
generates
a complete subgraph.
\item\label{it:ad5} If $x^\perp\backslash x\subseteq z^\perp\backslash z$ 
then $\ad(z)\subseteq \ad(x)$.
\item\label{it:ad6} If $x^\perp\subseteq z^\perp$ then 
$\ad(z)\subseteq \ad(x)$.
\item\label{it:ad12}  $\ad(z)\subseteq \ad(x)$ if and only if
$x^\perp\bs x \subset z^\perp$.
\item\label{it:ad7} $\ad(x)=\ad(z)$ if and only if either $x^\perp=z^\perp$ or 
$x^\perp\bs x=z^\perp\bs z$. 
\item\label{it:ad8} If $z\in \ad(x)$ then $\ad(z)\subseteq \ad(x)$.
\item\label{it:ad9} $\ad(U)=\cup_{y\in \ad(U)} \ad(y)$.
%%%%%\item\label{it:ad12} If $\cl(y)<\cl(x)$ then $\ad(x)\nleq \ad(y)$.
\item\label{it:ad13} If $\cl(x)=\ad(x)$ then $\cl(y)=\ad(y)$, for
all $y\in \ad(x)$. 
%\item\label{it:ad14} $\ad(y)\subseteq \ad(x)$ if and only if $\cl(y)
%\subseteq\cl(x)$ and the first inclusion is strict if and only if
%the second one is. 
\item\label{it:ad11} If $[x,z]=1$ then $[G(\ad(x)),G(\ad(z))]=1$.
\ee
\end{lemma}
\begin{vnr}
\begin{proof}
Statements \ref{it:ad1} to \ref{it:ad5} 
follow directly 
from the definitions and the fact that if $S\subseteq T$ then
$T^\perp \subseteq S^\perp$, for all subsets $S,T$ of $X$.  
For \ref{it:ad6}  note that in
this case $z\in x^\perp$, so as $x\neq z$,  
$\ad(x)
=(x^\perp\backslash x)^\perp
=((x^\perp\bs\{x,z\})\cup \{z\})^\perp
=(x^\perp\bs\{x,z\})^\perp\cap z^\perp
\supseteq (z^\perp\bs\{x,z\})^\perp\cap x^\perp
=\ad(z)$.

The right to left implication of \ref{it:ad12} is a consequence of  
 \ref{it:ad5} and \ref{it:ad6}, and the fact that 
if $x^\perp\bs x\subseteq z^\perp$
 then $x^\perp \subseteq z^\perp$ or $x^\perp\bs x\subseteq z^\perp\bs z$. 
 To see the opposite implication:  
 if $\ad(z)\subseteq \ad(x)$ then, as $z\in \ad(z)$, we have
$z\in \ad(x)$, so $x^\perp\bs x\subseteq z^\perp$, from Lemma \ref{lem:admot}.

To see \ref{it:ad7} suppose first that $\ad(x)=\ad(z)$. Then, from
\ref{it:ad12}, we have $x^\perp\bs x \subseteq z^\perp$ and 
$z^\perp\bs z\subseteq x^\perp$. If $x\in z^\perp$ then $z\in x^\perp$, and
in this case $x^\perp =z^\perp$. Otherwise $x\notin z^\perp$  and 
$z\notin x^\perp$ in which case $x^\perp\bs x = z^\perp\bs z$. Conversely,
if either $x^\perp =z^\perp$ or  $x^\perp\bs x = z^\perp\bs z$ then it
follows, from \ref{it:ad5} and \ref{it:ad6}, that $\ad(x)=\ad(z)$.

Statement \ref{it:ad8} follows immediately from \ref{it:ad12} and 
Lemma \ref{lem:admot}. 
Statement \ref{it:ad9} follows from
\ref{it:ad8} as if $y\in \ad(U)$ then $\ad(y)\subseteq \ad(U)$.

%Statement \ref{it:ad12} follows from \ref{it:ad3}, as % suppose that 
%$\cl(y)<\cl(x)$ %so 
%implies 
%$x^\perp<y^\perp$. 
%
%If
%$\ad(x)\le \ad(y)$ then $\cl(x)=\ad(x)\cap x^\perp\le \ad(y)\cp y^\perp=\cl(y)$,
%a contradiction.

To see statement \ref{it:ad13} observe that $\cl(x)$ is a simplex so
if $\cl(x)=\ad(x)$ and $y\in \ad(x)$ then $\ad(y)\subseteq \ad(x)$ implies
that $\ad(y)$ is a simplex. Therefore $\ad(y)\subseteq y^\perp$ and 
the result follows from \ref{it:ad2}.

For  \ref{it:ad11} suppose that 
$u\in \ad(x)$ and $v\in \ad(z)$. Since $z\in x^\perp\bs x$ we have
$u\in z^\perp$ and similarly $v\in x^\perp$. Since $[u,y]=1$ for
all $y\in x^\perp$, except possibly $x$, it follows that $u$ commutes
with $v$, unless $v=x$. However if $v=x$ then, since 
$v\in (z^\perp\bs z)^\perp$, $v$  commutes with all elements of $z^\perp$,
including $u$.

\end{proof}
\end{vnr}

%Define 
%\[x^\sh=\ad(x)\backslash\cup\{\ad(y)|y\in \ad(x) \textrm{ and }\ad(y) < \ad(x)\},
%\textrm{ for } x\in X.
%\]
Let $\sim_\perp$ be the relation on $X$ given by $x\sim_\perp y$ if
and only if $x^\perp=y^\perp$ and let $\sim_\lk$ be the relation given by
$x\sim_\lk y$ if
and only if 
$x^\perp\bs x=y^\perp\bs y$. These are equivalence relations
and the equivalence classes of $x$ under $\sim_\perp$ and 
$\sim_\lk$  are denoted by $[x]_\perp$ and $[x]_\lk$, respectively.
Note that if $|[x]_\perp|>1$ then $[x]_\lk=\{x\}$ and 
the same is true on interchanging $\perp$ and $\lk$. Therefore the
relation 
$\sim$,
given by $x\sim y$ if and only if $x\sim_\perp y$ or $x\sim_\lk y$, 
is an equivalence relation. Denote the equivalence class of 
$x$ under $\sim$ by $[x]$. Then $x\sim y$ if and
only if $x\sim_\perp y$ or $x\sim_\lk y$, and $[x]=[x]_\perp\cup[x]_\lk$.
It follows  that 
 $x\sim y$ if and only if $x^\perp\bs \{x,y\}=y^\perp\bs \{x,y\}$. 
\begin{lemma}\label{lem:ad1}
For all $x,z\in X$, 
\be[(i)]
\item\label{it:adcl1} $\ad(x)=\ad(z)$ if and only if  $z\in [x]$, %and 
\item\label{it:adcl2}
$[x]=
\ad(x)\bs(\cup\{\ad(y)|y\in \ad(x) \textrm{ and }\ad(y) < \ad(x)\})$.
\ee 

\end{lemma}
\begin{vnr}
\begin{proof}
%For statements \ref{it:adcl1} and \ref{it:adcl2}
The second statement  follows directly from the first, together with
 \ref{it:ad8} of Lemma \ref{lem:ad0}. Therefore 
it suffices to show that $\ad(x)=\ad(z)$ if and only if $z\in [x]$.
By definition, $z\in [x]$ if and only if $x^\perp=z^\perp$ or $x^\perp\bs x=z^\perp\bs z$. 
 From Lemma \ref{lem:ad0} \ref{it:ad7},
 this holds 
if and only if 
$\ad(x)=\ad(z)$.
\end{proof}
\end{vnr}
\begin{lemma}\label{lem:ad2}
\[[x]=\left\{
\begin{array}{ll}
[x]_\perp,& \textrm{ if } \ad(x)=\cl(x)\\
\left[x\right]_\lk,& \textrm{ if } \ad(x) > \cl(x)
\end{array}
\right.
.
\]
\end{lemma}
\begin{vnr}
\begin{proof}
First suppose $\ad(x)=\cl(x)$. Then, from Lemma \ref{lem:ad0} \ref{it:ad2}, 
 $\ad(x)\subseteq x^\perp$. If $z\in [x]$ then, 
from Lemma \ref{lem:ad1} \ref{it:adcl2}, $z\in [x]\cap x^\perp$ and 
so
 $z\sim_\perp x$.

Now suppose that $\cl(x)<\ad(x)$. If $z\in [x]$, $z\neq x$, and 
 $z^\perp= x^\perp$ then $x\in z^\perp\bs
z$, so $\ad(x)=\ad(z)\subseteq x^\perp$, contradicting Lemma \ref{lem:ad0} 
\ref{it:ad2}.
Hence $z\in [x]$ implies $z\sim_\lk x$. 
\end{proof}
\end{vnr}

In view of  Lemma \ref{lem:ad0} \ref{it:ad10} above, for an arbitrary subset $Y$ of $X$ the
$\ad${\em -closure} of $Y$ may be defined to be the admissible
set \[
\ca(Y)=\cap\{U\subseteq X|Y\subseteq U \textrm{ and } U=\ad(V),
\textrm{ for some } V\subseteq X\}.\]
Then $\ca(Y)$ is the smallest admissible set containing $Y$ and 
$Y$ is admissible if and only if $Y=\ca(Y)$. 
\begin{comment}
\begin{lemma}\label{lem:clcont}
$[x]\cap \cl(x)=[x]_\perp$ and $[x]\cap \ca(x)=[x]$.
\end{lemma} 
\begin{vnr}
\begin{proof} 
%To see statement \ref{it:adcl3} 
If $x^\perp=z^\perp$ then
$\cl(x)=\cl(z)$ so $z\in \cl(x)$ and therefore $[x]_\perp\subseteq 
\cl(x)$. Moreover if $x\in \ad(u)$ then $(u^\perp\bs u)\subseteq x^\perp
=z^\perp$, so also $z\in \ad(u)$. Therefore $[x]_\lk\subseteq 
\ca(x)$. Now if $x^\perp\bs x=z^\perp\bs z$ the same argument
shows that $[x]_\lk \subseteq \ca(x)$
 (since $u$ commutes with $x$ if and only if it commutes with $z$). 
However, if $z\in \cl(x)$ then
$z\in x^\perp$ so $z\notin [x]_\lk$ unless $z=x$.
\end{proof} 
\end{vnr}
\end{comment}

 It is shown in \cite{DKR3} that  $Y\in \cL$ if and only if $Y=U^\perp$,
for some $U\subseteq X$. Therefore 
$\cK \subset \cL$. 
The set $\cK$, partially ordered by inclusion, with infimum $U\wedge V=U\cap V$ and supremum
$U\vee V=\ca(U\cup V)$ forms a lattice. The lattice $\cK$ has
 maximal element $X=\ca(X)=\ad(\nul)$ and  minimal element 
$\ad(X)$. The lattice of admissible sets for the graph of Example \ref{ex:adex}
is shown in Figure \ref{subf:adexb}.

Although the lattice $\cK(\G)$ of the graph $\G$ in Example \ref{ex:adex} 
consists of $\ad(\nul)$, $\ad(X)$ and sets of the form $\ad(x)$, where $x\in X$, this
is atypical. For example consider the path graph of length three: that 
is the tree with vertices $a$, $b$, $c$ and $d$ and edges $\{a,b\}$, 
$\{b,c\}$ and $\{c,d\}$. In this case $\ad(a)=\{a,b,c\}$ and $\ad(b)=\{b\}$.
It follows that  $|\ad(x)|=1$ or $3$, for all vertices $x$. 
However $\ad(\{a,d\})=\{b,c\}$; so $\ad(\{a,d\})\neq \ad(x)$, for all vertices
$x$. Moreover $\ca(\{a,d\})=X\neq \ad(\{a,d\})$. 
We shall be mostly interested here in the admissible sets 
$\ad(x)$, for $x\in X$; for which we can say the following. 
\begin{lemma}
For $x\in X$, $\ca(x)=\ad(x)$. 
\end{lemma}
\begin{proof}
By definition $\ca(x)\subseteq \ad(x)$. If $U$ is admissible 
then $U=\cap_{y\in Y}\ad(y)$, for some $Y\subseteq X$.  If $x\in U$
this  means that 
$x\in \ad(y)$, so $\ad(x)\subseteq \ad(y)$, for all $y\in Y$. Hence
$\ad(x)\subseteq U$. It follows that $\ad(x)\subseteq \ca(x)$. 
\end{proof}

\begin{comment}
\begin{exam}
In the graph of Figure \ref{fig:claex} we have $\ad(a)=b^\perp$ and 
$\ad(b)=\{b\}$. Therefore $|\ad(v)|=1$ or $4$, for all vertices $v$ of the
graph. However $\ad(\{a,d\})=b^\perp\cap c^\perp$ is an admissible
set of  two elements.
Therefore $\ad(\{a,d\})$ is not equal to $\ad(v)$, for all vertices $v$. 
In this example $\ca(\{a,d\})=X$. 
\end{exam}
\begin{figure}\label{fig:claex}
\begin{center}
\psfrag{a}{$a$}
\psfrag{b}{$b$}
\psfrag{c}{$c$}
\psfrag{d}{$d$}
\includegraphics[scale=0.4]{claex.eps}
\caption{}\label{fig:claex}
\end{center}
\end{figure}
\end{comment}

\subsection{Ordering $X$}
The next goal is to define a total ordering on $X$ which reflects
the structure of the lattice $\cK$. 
First define a partial 
order $<_{\cK}$ on $X$ by $x<_{\cK} y$ if and only if $\ad(x)<\ad(y)$. 
If $\ad(x) =\ad(y)$ we write $x=_\cK y$. We say $x$ is $\cK${\em -minimal}
if $y\le_\cK x$ implies $y=_\cK x$. The definition of 
$\cK${\em -maximal} is then the obvious one. 
Recall that in \cite{DKR5} the analogous ordering using $\cL$ instead
of $\cK$ was defined, and the definitions of $\cL$-minimal and $\cL$-maximal
were also defined using the closure operator instead of the $\ad$ operator.
\begin{lemma}\label{lem:admin}
An element $x\in X$ is $\cK$-minimal if and only if  $[x]=\ad(x)$.
If $x$ is $\cK$-minimal then
\be[(i)]
\item\label{it:Kmin2}
$x$ is $\cL$-minimal and
\item\label{it:Kmin3} 
$\cl(y)=[y]_\perp$, for all  $y\in \ad(x)$.
\ee 
\end{lemma}
\begin{vnr}
\begin{proof}
The first statement follows immediately from the definitions
and 
 Lemma \ref{lem:ad1}. For the second suppose $x$ is $\cK$-minimal.
If $[x]=\ad(x)=\cl(x)=[x]_\perp$ then $\cl(y)=\cl(x)$, for
 all $y\in \cl(x)$, so $x$ is $\cL$-minimal and \ref{it:Kmin2} and 
\ref{it:Kmin3} hold. If 
$[x]=\ad(x)>\cl(x)$ then $[x]_\lk=\ad(x)$,  so
$\cl(x)=[x]_\lk\cap x^\perp=\{x\}$, from Lemma \ref{lem:ad0}\ref{it:ad2},
%from Lemma \ref{lem:clcont}, 
% $\cl(x)\cap [x]_\lk=\{x\}$, so
%$\{x\}=\cl(x)=[x]_\perp$; 
 so again $x$ is $\cL$-minimal. 
To see \ref{it:Kmin3} in this case 
note that $y\in \ad(x)$ implies $y\in [x]_\lk=[y]_\lk$ and 
$\cl(y)=y^\perp\cap \ad(y)=y^\perp \cap \ad(x)=y^\perp\cap [y]_\lk=\{y\}=
[y]_\perp$.
\begin{comment}
suppose $x$ is $\cK$-minimal, so from \ref{it:Kmin2}
every $y\in \ad(x)$ is $\cL$-minimal. If $\ad(x)$ has only one element there
is nothing further to prove. Suppose $\ad(x)$ has more than one element. If,
in this case, $[x]=[x]_\perp$ then $[x]_\lk=\{x\}\neq \ad(x)=[x]_\perp$. Then
$\{x\}$ must be $\cK$-minimal and the elements of $\ad(x)$ not equal to $x$
are therefore not $\cK$-minimal, a contradiction. Hence $[x]=[x]_\lk=\ad(x)$ and 
$[y]_\perp=\{y\}=\cl(y)$, for all $y\in \ad(x)$.
\end{comment}
\end{proof}
\end{vnr}
As in Example \ref{ex:ord} below,  there may be
elements which are $\cL$-minimal but are not $\cK$-minimal.

We  now define a total order $\prec$ on $X$, which will have the properties that
\be
    \item\label{it:prec1} if $x<_\cK y$ then $y\prec x$ and
    \item\label{it:prec2} if $z\prec y\prec x$ and $z\in [x]$ then 
$y\in [x]$.
\ee
\begin{vnr}
Define $\cK_X$ to be the subset of $\cK$ consisting of sets
$\ad(x)$, for $x\in X$. 
To begin with let 
\[
B_0=\{Y\in \cK_X:Y=\ad(x), \textrm{ where } x \textrm{ is } \cK\textrm{-minimal}\}.
\]
Suppose that $B_0$ has $k$ elements and choose an ordering
$Y_1<\cdots <Y_k$ of these elements.
If $i\neq j$ then %$Y_i\cap Y_j\in \cK$ and from 
%and the fact that the $Y_i$'s are $\cK$-minimal 
it follows from Lemma \ref{lem:admin}
that $Y_i\cap Y_j=\emptyset$. Therefore we may
define the ordering $\prec$
 on $\cup_{i=1}^k Y_i$ in such a way that if $x_i\in Y_i$ and $x_j\in Y_j$
and $Y_i<Y_j$ then $x_j\prec x_i$: merely by choosing an ordering for
elements of each $Y_i$. (For instance, in Example \ref{ex:ord} we can
choose $f\prec  e\prec d\prec c$.)

We recursively define sets $B_i$ of elements of $\cK_X$, for $i\ge 0$,
as follows.
Assume that we have defined sets
$B_0,\ldots ,B_i$, set $U_i=\cup_{j=0}^i B_j$ and
define $X_i=\{u\in X: u\in Y, \textrm{ for some } Y\in U_i\}.$
If $U_i\neq \cK_X$
define $B_{i+1}$ by
\[B_{i+1}=\{Y=\ad(x)\in \cK_X:
 Y\notin U_i, \textrm{ and } y<_\cK x \textrm{ implies that } \ad(y)\in U_i
\}.\]
If $U_i\neq \cK_X$ then $X_i\neq X$ and $B_{i+1}\neq \nul$.
We assume inductively that we have ordered the set $X_i$ in such a way
that if 
\be[(i)]
\item 
$0\le a<b\le i$,
\item  
 $x_a\in Y_a$ where $Y_a\in B_a$ and
\item 
 $x_b\in Y_b$ where $Y_b\in B_b$;
\ee 
then 
$x_b\prec x_a$. 
%Note that if $[x]\neq [x]_\lk$ then $[x]_\perp=\{x\}$ and so $[x]_\lk
%\in B_0$. This means also that $[x]\subseteq X_0$ and so from 
%Lemma \ref{lem:ad2}, if $\ad(x)$ is not in $B_i$ then?????
If $Y=\ad(x)\in B_{i+1}$
then
\begin{align*}
[x]
&=Y\bs \{u\in \ad(y):y<_\cK x\}\\
&=Y\bs \{u\in X_i\}.
\end{align*}
Therefore we have defined $\prec$ on the set $Y\bs [x]$.
Moreover, if $Y_1\neq Y_2$ and $Y_1,Y_2\in B_{i+1}$ then
$Y_1\cap Y_2\in \cK$ so $z\in Y_1\cap Y_2$ implies $\ad(z)\subseteq Y_1\cap Y_2$.
As $Y_1\neq Y_2$ this implies that $\ad(z)$ is strictly contained in
$Y_i$, $i=1,2$. If $Y_i=\ad(x_i)$ then $z<_\cK x_i$ and so $z\notin [x_i]$,
$i=1,2$. That is, $[x_1]\cap [x_2]=\emptyset$.
Now choose an ordering on the set of elements of $B_{i+1}$:
$Z_1<\cdots <Z_k$ say, where $Z_j=\ad(x_j)$.
Then $Z_j\bs [x_j]\subseteq X_i, j=1,\ldots ,k$.
We can extend the  total
order $\prec$ on $X_i$ to
\[
X_{i+1}=X_i\cup \cup_{j=1}^k Z_{j} =X_i\cup\cup_{j=1}^k[x_j]
\]
as follows.
Assume the order has already been extended to
$X_i\cup_{j=1}^{s-1} [x_{j}]$.
Extend the order further by choosing the ordering $\prec$ on
the elements of $[x_s]$ and then defining the  greatest element
 of $[x_s]$ to be less
than the least element of $X_i\cup_{j=1}^{s-1}
[x_j]$ (as in the last step of Example \ref{ex:ord}).
At the final stage $s=k$ and the order on $X_i$ is extended to $X_{i+1}$.
We continue until $U_i=\cK_X$, at
which point $X=X_i$ and we have the required total order on $X$.
Note that, by construction,
if $x,y\in X$ and $x<_\cK y$ then
$y\prec x$. Also, if
$x\prec y\prec z$ and $[z]=[x]$ then $[y] =[x]$.
\end{vnr}
Thus \ref{it:prec1} and \ref{it:prec2} above hold.
If $\ad(x)$ belongs to $B_i$ we shall say that $x$, $\ad(x)$ and
$[x]$ have {\em height} $i$ and write $h(x)=h(\ad(x))=h([x])
=i$.
\begin{exam}\label{ex:ord}
Let $\G$ be the graph of Figure \ref{fig:ord}. Then 
\begin{align*}
\ad(a)  =\ad(b)&=\{a,b,e,f\},\\
\ad(c)=\ad(d)&=\{c,d\}, \\
\ad(e)=\ad(f)&=\{e,f\}\textrm{ and }\\
\ad(g)&=\{c,d,g\}.
\end{align*}
In the partial order $<_\cK$ we have 
\begin{itemize}
\item 
$x<_\cK a$ and $x<_\cK b$, if $x\in \{e,f\}$, and
\item
$c<_\cK g$ and $d<_\cK g$.
\end{itemize}
Elements $c$, $d$, $e$ and $f$ are $\cK$-minimal while $a$, $b$ and $g$ 
are $\cK$-maximal. (For all $x\in X$, we have $\cl(x)=\{x\}$, so 
$a$, $b$ and $g$ are $\cL$-minimal but not $\cK$-minimal.) 

We have $B_0=\{\ad(c), \ad(e)\}$ and we must choose an order on this
set: say $\ad(c)<\ad(e)$. Next choose orders on $\ad(c)$ and $\ad(e)$:
 say $d\prec c$ and $f\prec e$. The construction now gives the order
\[
f\prec e\prec d\prec c
\]
on $\ad(c)\cup \ad(e)$. 

Now $U_0=B_0$ and $X_0=\ad(c)\cup \ad(e)$; so $B_1=\{\ad(a),\ad(g)\}$. 
 We must
choose an ordering on $B_1$: say $\ad(a)<\ad(g)$. 
\begin{align*}
\ad(a) &= \{a,b\} \cup \{e,f\}\textrm{ and }\\
\ad(g) &= \{g\}\cup \{e,f\},
\end{align*} 
where $\{a,b\}=[a]$, $\{g\}=[g]$ and $\{e,f\}\subseteq X_0$. We must
choose an ordering on $[a]$: say $b\prec a$ and then extend the order
on $X_0$ to $X_0\cup [a]$ by 
\[
b\prec a\prec f\prec e\prec d\prec c.
\] 
Finally extend this order to $[g]$ to obtain 
\[
g\prec b\prec a\prec f\prec e\prec d\prec c.
\]
\end{exam}
\begin{figure}
\begin{center}
\psfrag{a}{$a$}
\psfrag{b}{$b$}
\psfrag{c}{$c$}
\psfrag{d}{$d$}
\psfrag{e}{$e$}
\psfrag{f}{$f$}
\psfrag{g}{$g$}
\includegraphics[scale=0.4]{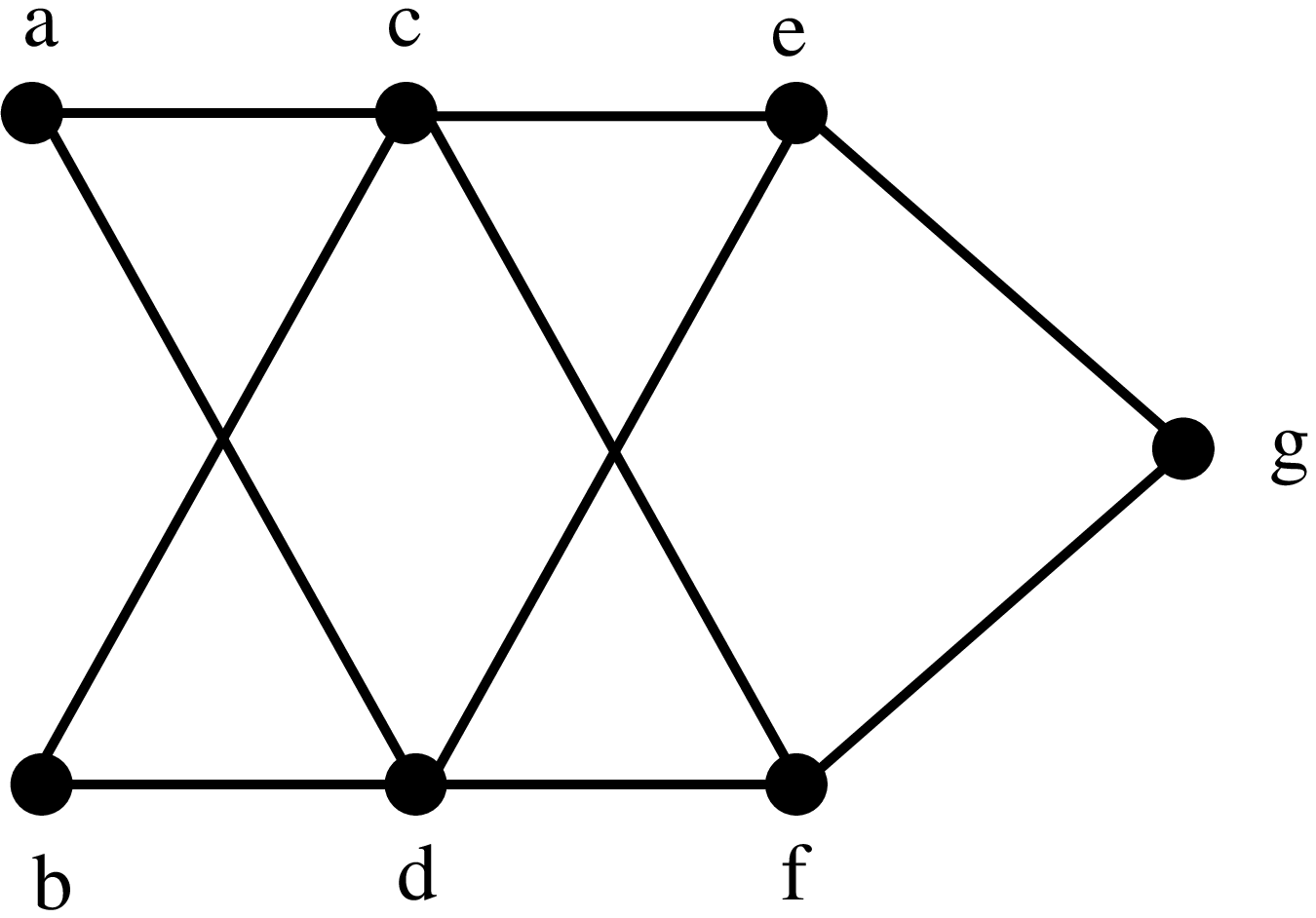}
\caption{}\label{fig:ord}
\end{center}
\end{figure}
%%% Local Variables: 
%%% mode: latex
%%% TeX-master: "aut2"
%%% End: 

%\begin{comment}
%\include{FR2}
%\end{comment}
\section{Generators for $\Aut(G)$ and Decomposition over Graph Automorphisms}
\label{sec:generators} 
Recall the convention of Section \ref{sec:preliminaries}:  $G$ denotes 
a partially commutative
group with commutation graph $\G$, and $X=V(\G)$. \\[1em]
\textbf{Notation.} We shall often abuse notation when discussing
elements of $X^{-1}$ by referring to $x^{-1}$ as a ``vertex'' of $\G$, when
we really mean that $x$ is a vertex. In particular, for $z=x^{-1}$
we refer to $z^\perp$ and  $\ad(z)$ when we mean $x^\perp$ and $\ad(x)$, 
respectively.

\subsection{Graph Automorphisms}\label{subsec:graphaut}
For any graph $\W$ let $\Iso(\W)$ denote the group
of graph automorphisms of $\W$. If $\W$ is labelled then
by an automorphism of $\W$ we mean a graph automorphism which preserves labels. 
We shall use the equivalence $\sim$ of Section \ref{subsec:admiss} 
to define a quotient graph of $\G$. 
Let $u, v\in X$. 
In \cite{DKR2} it is shown  that there is an edge of $\G$ joining
$u$ to $v$ if only if there is an edge joining $a$ to $b$, for all
$a\in [u]$ and $b\in [v]$. 
Therefore
there is a well-defined  graph with vertex set 
consisting of the equivalence classes of $\sim$
and an edge joining vertices $[u]$ and $[v]$ if and only if 
there is an edge of $\G$ joining $u$ and $v$. The resulting graph
has no multiple edges but may have loops. 

\begin{defn}\label{defn:comp}
The {\em compression} of the graph $\G$ is the labelled 
graph $\G^{\cmp}$ with
vertices $X^\cmp=\{[v]:v\in X\}$ and an edge joining
$[u]$ to $[v]$ if and only if $u$ is joined to $v$ by an edge of
$\G$. 
Vertices of $\G^{\cmp}$ are labelled as follows.
Let  $u\in X$ and $|[u]|=d$. 
\be
\item 
If $d=1$ then $[u]$ is labelled $(1,1)$.
\item 
If $d>1$ and $[u]=[u]_\perp$ then $[u]$ is labelled $(\perp,d)$.
\item
If $d>1$ and $[u]=[u]_\lk$ then $[u]$ is labelled $(\lk,d)$.
\ee
\end{defn}

We shall express each   automorphism $\phi$ of $\G$ as a product
$\phi=\a\b$, where $\a$ corresponds to a certain automorphism of $\G^\cmp$ and 
$\b$ is an automorphism of $\G$ which maps $[u]$ to itself, for all 
$u\in X$. First we make some definitions. 
If $\W$ and $\W^\prime$ are graphs without multiple
edges and
$f$ is a map from $V(\W)$ to $V(\W^\prime)$ then we say that $f$ 
{\em induces
a graph  homomorphism} if, for all $u,v\in V(\W)$, 
 $uf$ is joined to $vf$ whenever $u$ is joined to $v$.
It is easy to see (for details see  \cite{DKR2}) that 
the map from $X$ to $X^\cmp$  sending $u$ to $[u]$ 
induces a homomorphism $\cmp:\G\maps \G^\cmp$. Every automorphism
$\phi\in \Iso(\G)$ induces a label preserving automorphism $\phi^\cmp \in 
\Iso(\G^\cmp)$: sending
$[u]$ to $[u\phi]=[u]\phi$. 
In fact the map $\pi_\cmp:\Iso(\G)\maps \Iso(\G^\cmp)$, given by
$\phi\pi_\cmp=\phi^\cmp$, is a homomorphism (see \cite{DKR2}). 

On the other hand we may fix a total order on the 
elements of $[u]$, for all $u\in X$. Then, if $\phi^\cmp$ is an automorphism
of $\G^\cmp$, the label of $[u]$ is identical to the label of 
$[u]\phi^\cmp$, for all $u\in X$. Hence 
there is a unique order preserving bijection from $[u]$ to
$[u]\phi^\cmp$. The union of these bijections 
is  an automorphism  $\phi$ of $\G$; and we may define $\i_\cmp$ to 
be the map from $\Iso(\G^\cmp)$ to $\Iso(\G)$ given by $\phi^\cmp \i_\cmp
=\phi$. Then $\i_\cmp$ is a homomorphism and $\i_\cmp\pi_\cmp$ is the 
identity map of $\G^\cmp$. 
\begin{defn}\label{defn:Gcomp}
Define the \emph{compressed automorphism group} $\Iso_\cmp(\G)$ of $\G$ to be the 
subgroup $\Iso(\G^\cmp)\i_\cmp$ of $\Iso(\G)$. 

For $v\in X$ let $S_{[v]}$ denote the group of permutations
of $[v]$; so $S_{[v]}$ is a subgroup of $\Iso(\G)$ 
isomorphic to the symmetric group of degree $|[v]|$.
\end{defn}  
The definition of  $\Iso_\cmp(\G)$ depends on the choice of ordering of $[u]$, which
we regard as fixed for the remainder of the paper.  We then have the 
 following lemma.
\begin{lemma}[cf. {\cite[Proposition 2.52]{DKR2}}]\label{lemma:graphdecomp}
~
%\be[(i)]
%\item $\Iso_\cmp(\G)\cong \Iso(\G^\cmp)$.
%\item 
$\Iso(\G)=\left(\prod_{[v]\in X^\cmp} S_{[v]}
\right)
\rtimes  \Iso_\cmp(\G)$.
%\ee
\end{lemma}

Now we focus attention on automorphisms which
interchange connected components of $\G$. 
 First of all we fix 
notation for these connected components.  
In the following definition we adopt the convention that, if $m$ is a 
non-negative integer and $\Omega$ is a graph then $\Omega^m$ denotes
the disjoint union of $m$ copies of $\Omega$ (and is empty if $m=0$). 
\begin{defn}\label{defn:comps}
Let $\Omega_0$ denote the graph  consisting of a single vertex and no
edges. Suppose that there exist 
pairwise non-isomorphic graphs $\Omega_1, \ldots ,\Omega_d$, such that
every connected component of $\G$ with at least two vertices is isomorphic
to $\Omega_i$, for some $i\ge 1$, and that $d$ is minimal with this property.
Then 
\begin{equation}\label{eq:isomtype}
\G\cong \Omega_0^{m_0}\cup \Omega_1^{m_1}\cup \cdots \cup \Omega_d^{m_d},
\end{equation}
for some $m_i\in \ZZ$, with $m_0\ge 0$ and $m_i\ge 1$, for $i\ge 1$. 
In this case we say that 
the right hand side of \eqref{eq:isomtype} is the {\em isomorphism type}
of $\G$. 
\end{defn}

Suppose that $\G$ has isomorphism type 
$\Omega_0^{m_0}\cup \Omega_1^{m_1}\cup \cdots \cup \Omega_d^{m_d}$. Identify each
connected component of $\G$ with a 
particular copy of $\Omega_j$ (to which it is isomorphic) in the disjoint 
union $\Omega_j^{m_j}$. 
 To be explicit, let 
\[\G=\displaystyle{\cup_{j=0}^d\cup_{k=1}^{m_j}\G_{j,k}},\]
where $\G_{j,k}\cong \Omega_j$, for $k=1,\ldots, m_j$. 
Fix  an isomorphism from $\G_{j,k}$ to $\Omega_j$, for all $j$ and $k$. 
For $0\le j\le d$ and $1\le a<b\le m_j$ there is an 
  isomorphism of 
 $\Omega_0^{m_0}\cup \Omega_1^{m_1}\cup \cdots \cup \Omega_d^{m_d}$ interchanging
the $a$th and $b$th copy of $\Omega_j$ and fixing all other 
connected components pointwise. This induces, via the fixed isomorphisms
 of $\G_{j,k}$ and $\Omega_j$, 
an isomorphism $\w^j_{a,b}$ of $\G$, 
which interchanges $\G_{j,a}$ and $\G_{j,b}$ and leaves
all other components fixed. The subgroup of $\Iso(\G)$ generated by $\{\w^j_{a,b}:
1\le a<b\le m_j\}$ is then isomorphic to the symmetric group $S_{m_j}$ of
degree $m_j$.
\begin{defn}\label{defn:compperm}
Denote by $\Iso_{\sym}(\G_{j,*})$ the  subgroup of $\Iso(\G)$ generated by $\{\w^j_{a,b}:
1\le a<b\le m_j\}$. Denote by $\Iso_\cmp(\G_{j,k})$  the subgroup of 
$\Iso_\cmp(\G)$ consisting of 
 compressed automorphisms $\phi$ such that $\phi|_{\G_{j,k}}$
is an automorphism  of $\G_{j,k}$ and $x\phi =x$, for all 
$x\in X\bs X_{j,k}$.
\end{defn}
  Thus  $\Iso_\cmp(\G_{j,k})\cong  \Iso(\Omega_j^\cmp)$. 
Note that $\G^\cmp$ has isomorphism type 
$\Omega_0^{n_0}\cup \cup_{i=1}^d(\Omega_i^\cmp)^{m_i}$, 
where $n_0=0$, if $m_0=0$, and $n_0=1$, if $m_0>0$; so 
we obtain 
 the following decomposition. %\refof $\Iso_\cmp(\G)$. 
\begin{lemma}\label{lemma:permauts}
Let $\G$ have isomorphism type given by \eqref{eq:isomtype}. Then 
\[\Iso_\cmp(\G)=\displaystyle{\prod_{j=1}^d \left(\prod_{k=1}^{m_j}
\Iso_\cmp(\G_{j,k})\rtimes \Iso_{\sym}(\G_{j,*})\right)},\]
with $\Iso_{\sym}(\G_{j,*})$  isomorphic to the symmetric group of degree $m_j$
and $\Iso_\cmp(\G_{j,k})$ isomorphic to $\Iso(\Omega_j^\cmp)$. 
\end{lemma}
Each of the above subgroups of $\Aut(\G)$ is naturally isomorphic to
 a subgroup
of the automorphism group of $G$; which we shall now name. 
\begin{defn}\label{defn:graphaut}
Let $\Aut(G)$ be the automorphism group of the partially 
commutative group $G$ with commutation graph $\G$. 
An element $\phi\in \Aut(G)$ is  
\begin{itemize}
\item 
a {\em graph automorphism} if the 
restriction  $\phi{|_X}$ of $\phi$ to $X$ is an element of $\Iso(\G)$; and 
\item
a \emph{compressed graph automorphism} if  $\phi{|_X}$ is 
an element of $\Iso_\cmp(\G)$.
\item 
Denote by $\Aut^\G(G)$ and $\Aut^\G_\cmp(G)$  the subgroups of $\Aut(G)$ consisting of graph
automorphisms and compressed graph automorphisms, respectively. 
\item
For $v\in X$,
 denote by $S_{[v]}(G)$ the subgroup of $\Aut^\G(G)$ consisting of elements 
$\phi$ such that $\phi|_X\in S_{[v]}$. 
\item
Denote by $\Aut^\G_{\sym}(G_{j,*})$ the subgroup of automorphisms  $\phi$ of $\Aut(G)$ such
that $\phi|_X$ is an element of $\Iso_{\sym}(\G_{j,*})$; and   
\item
by 
$\Aut^\G_\cmp(G_{j,k})$ the subgroup of automorphisms $\phi$ such that $\phi|_X$ is an element
of $\Iso_\cmp(\G_{j,k})$. 
\end{itemize}
\end{defn}
\begin{rem}
The subgroup $\prod_{j=0}^{d} \Aut^\G_{\sym}(G_{j,*})$  of $\Aut(G)$ 
 is denoted $\Pi$ and called the 
 group of permutation automorphisms in \cite{Gilbert87}. %; 
\end{rem}

The following
proposition follows from Lemmas \ref{lemma:graphdecomp} and 
\ref{lemma:permauts}.
\begin{prop}\label{prop:graphaut}
Let $\G$ have isomorphism type given by \eqref{eq:isomtype}. Then 
\be[(i)] 
\item
$\Aut^\G(G)= (\prod_{[v]\in X^\cmp} S_{[v]}(G)
)\rtimes  \Aut^\G_\cmp(G)$, with $S_{[v]}(G)$ isomorphic to the symmetric group
of degree $|[v]|$,  and 
\item\label{it:graphaut2}
\[
\Aut^\G_\cmp(G)= \prod_{j=1}^d \left(\prod_{k=1}^{m_i}
\Aut^\G_\cmp(G_{j,k})
\rtimes \Aut^\G_{\sym}(G_{j,*})\right).\]
Moreover $\Aut^\G_{\sym}(G_{j,*})$  is isomorphic to the 
symmetric group of degree $m_j$
and $\Aut^\G_\cmp(G_{j,k})$ is isomorphic to $\Aut^{\Omega_j}_\cmp(G(\Omega_j))$.
\ee
\end{prop}
 In the sequel we shall the particular generators for $\Aut^\G_\cmp(G)$ 
of the next definition.  
\begin{defn}\label{defn:graphautgens}
Define the following sets of graph automorphisms. 
\be[(a)]
\item For $1\le j\le d$,  let $\cP^\G_{\cmp,j}%(G_{j,1})
$ be a generating set for 
$\Aut^\G_\cmp(G_{j,1})$.
\item For $0\le j\le d$ and $1\le a<b\le m_j$, 
the group automorphism induced by the graph automorphism $\w^j_{a,b}$ (defined
above) is also
called $\w^j_{a,b}$ and we define 
\[\cP^\G_{\sym,j}%(G_{j,*})
=\{\w^j_{a,b}|1\le a<b\le m_j\}\subseteq \Aut^\G(G).\] 
\item $\cP^\G_\cmp(G)=\cP^\G_{\sym,0}%(G_{0,*})
\cup\cup_{j=1}^d(\cP^\G_{\cmp,j}%(G_{j,1})
\cup \cP^\G_{\sym,j}
%(G_{j,*})
)$.
\ee
\end{defn}
When the meaning is clear we write $\cP^\G_\cmp$ instead of 
$\cP^\G_\cmp(G)$. 
\begin{lemma}\label{lem:autgen}
\be[(i)]
\item 
$\Aut^\G_{\sym}(G_{j,*})$ is generated by $\cP^\G_{\sym,j}%(G_{j,*})
$.
\item\label{it:autgen2}
$\Aut^\G_\cmp(G)$ is generated by $\cP^\G_\cmp(G)$. 
\ee
\end{lemma}
\begin{proof}
The lemma follows directly from the definitions and
Proposition \ref{prop:graphaut}\ref{it:graphaut2}.
\end{proof}
Clearly a presentation for $\Aut_\cmp^\G(G)$ may be constructed from the 
decomposition of Proposition \ref{prop:graphaut}, 
using the generators  $\cP^\G_\cmp(G)$, but we leave details to the appendix.
 
\subsection{Generators for $\Aut(G)$}
\label{section:gens}

\begin{defn}
Given $x\in X$, the automorphism of $G$ mapping $x$ to $x^{-1}$ and fixing 
all other generators is called an {\em inversion} and 
denoted $\i_x$. The set of all inversions is denoted $\Inv=\Inv(G)$. 
\end{defn}
\begin{defn}
For fixed $x,y\in X^{\pm 1}$ 
an automorphism sending $x$ to $xy$ and fixing all
elements of $X^{\pm 1}$ other than $x^{\pm 1}$ is denoted $\tr_{x,y}$ and called 
a {\em transvection}.   The set of all transvections $\tr_{x,y}$ such that
$x\in X^{\pm 1}$ and  $y\in X$ 
is denoted $\Tr=\Tr(G)$. 
\end{defn}

For distinct $x,y\in X$, there exists an element $\tr_{x^\e,y^\d}\in \Aut(G)$ if and only if $x^\perp\bs x
\subseteq y^\perp$.

\begin{defn}\label{defn:conj}
Let $\phi\in \Aut(G)$ be an automorphism and assume that, for 
all $x\in X$, there exists $g_x\in G$ such that $x\phi=x^{g_x}$. 
Then $\phi$   is called  a {\em basis-conjugating} 
automorphism.  
 The subgroup of $\Aut(G)$ consisting of   all basis-conjugating automorphisms 
is denoted $\Conj(G)$.
\end{defn}
The group of inner automorphisms $\Inn(G)$ is a normal subgroup of $\Conj(G)$.
\begin{defn}\label{defn:graphminus}
For $S\subseteq X$ define $\G_S$ to be $\G\bs S$, the graph obtained from $\G$ by removing all vertices of $S$ and all their incident edges. 
\end{defn}
\begin{defn}\label{defn:elia}
Let $x\in X$, let $C$ be the vertex set of 
a connected component of $\G_{x^\perp}$ and
let $\e=\pm 1$. 
The automorphism $\a_{C,x^\e}$ given by
\[
y\mapsto
\left\{
\begin{array}{ll}
y^{x^\e}, &\textrm{ if } y\in C\\
y, &\textrm{ otherwise}
\end{array}
\right.
\]
is called an {\em elementary conjugating automorphism} of $\G$.
 
The set of all  elementary conjugating automorphisms (over all connected components of $\G_{x^\perp}$ and all $x\in X$) is denoted $\LInn=\LInn(G)$.  
\end{defn}

\begin{theorem}[Laurence \cite{Laurence95}]\label{thm:laurence}
The group of basis-conjugating automorphisms 
$\Conj(G)$ is generated by the set $\LInn(G)$.
\end{theorem}

In \cite{Laurence95} it is shown that $\Aut(G)$ is generated by the
following automorphisms.
\be[(i)]
\item A fixed choice $\cP^\G$ of 
generators for the graph automorphisms $\Aut^\G(G)$.
\item The set of inversions  $\Inv$.
\item The set of transvections $\Tr$.
\item The set of elementary conjugating automorphisms $\LInn$.
\ee

%It is worth noting the following corollary of this fact.
%\ajd{Probably this corol should go.} 
%\begin{corol}\label{cor:autpara}
%If $H$ is a canonical parabolic subgroup of $G$ and $\phi$ is one
%of the generating  automorphisms of $G$ above,  
%such that $H\phi\subseteq H$ then $H\phi =H$. 
%\end{corol}

We shall construct  various decompositions of $\Aut(G)$ 
and in view of these decompositions we shall reduce to 
proper generating 
subsets of Laurence's generators. The 
first such reduction is the following.

\begin{prop}\label{prop:autgen1}
$\Aut(G)$ is generated by
$\Inv\cup\Tr\cup \LInn\cup \cP^\G_\cmp$. 
\end{prop}
%\begin{vnr}
\begin{proof}
 To see that these automorphisms 
 generate $\Aut(G)$, it suffices, using Lemma \ref{lem:autgen}\ref{it:autgen2},  
 to show that every automorphism $\phi\in \Aut^\G(G)$ belongs
to the subgroup generated by $\Inv$, $\Tr$ and 
$\Aut^\G_\cmp(G)$. % and $\cP_\sharp$. 
From 
Proposition \ref{prop:graphaut}, $\phi$ may be written as
$\phi=\a\b$, where 
$\a\in \prod_{[v]\in X^\cmp} S_{[v]}(G)$ and 
$\b\in \Aut^\G_\cmp(G)$. 
Hence it is enough to show that 
$S_{[v]}(G)\subseteq \la \Inv,\Tr\ra$. %, \cP_\sharp\ra$. 
As $[v]$ generates a free or free Abelian subgroup of $G$, for
all $x,y\in [v]$ and $\e=\pm 1$, the transvections $\tr_{x^{\e},y}$ and
inversions $\i_x$ and $\i_y$ are automorphisms of $G$ and 
belong to $\Inv \cup \Tr$. The permutation $\s_{x,y}$ 
sending
$x$ to $y$ and $y$ to $x$ and fixing all other generators can be
obtained as a word in these generators; 
$\s_{x,y}=\i_x\tr^{-1}_{x,y} \tr_{y,x}\tr_{x^{-1},y}$. As $S_{[v]}(G)$ is
generated by such elements it follows that   
 $S_{[v]}(G)\subseteq \la \Inv,\Tr\ra$ as required.

\end{proof}
%\end{vnr}

\subsection{Decomposition of $\Aut(G)$ over Graph Automorphisms}
\begin{defn}\label{defn:aut*}
Let $\oAut(G)$ denote the subgroup of $\Aut(G)$ generated
by the set $\cP_* = \Inv \cup \Tr \cup \LInn$. %$\cup \cP_\sharp$.
\end{defn}
Later we shall show that $\oAut(G)$ has a natural description in terms
of the stabiliser of the lattice $\cK$. Here we establish what we need
in order to obtain an initial decomposition of $\Aut(G)$ in terms of 
$\Aut^\G_\cmp(G)$. It is useful to establish the following
fact first.
\begin{lemma}\label{lem:fixad}
Let $x,y\in X$ and let $C$ be a connected component of 
$\G_{y^\perp}$. If $\ad(x)\nsubseteq C\cup y^\perp$ and 
$\ad(x)\cap C\neq \nul$ then $y\in \ad(x)$.  
\end{lemma} 
\begin{proof}
 Suppose  that
$y\notin \ad(x)$. 
Then there exists $u\in x^\perp\bs x$ such that $[y,u]\neq 1$.
Thus $u\notin y^\perp$ and $u\in x^\perp\bs x$; so $[u,v]=1$  for all
$v\in \ad(x)$. This means that $\ad(x)\bs y^\perp$ 
is contained in some connected
component of $\G_{y^\perp}$. If $\ad(x) \cap C\neq \nul$ it follows
that $\ad(x) \subseteq C\cup y^\perp$, 
%so %this case does not occur. 
%
%Therefore,
%if $C\cap \ad(x)\neq \nul$ and
%$\ad(x) \nsubseteq C\cup y^\perp$ then $y\in \ad(x)$. 
and the result follows.
\end{proof}
\begin{prop}\label{prop:aut*}
Let $\phi\in \oAut(G)$. Then, for all $x\in X$, there exists $f_x\in G$ such
that $G(\ad(x))\phi=G(\ad(x))^{f_x}$.
\end{prop}
%\begin{vnr}
\begin{proof}
It suffices to prove the statement in the case where $\phi$ is a generator
of $\oAut(G)$. First consider the case where $\phi$ is a elementary 
conjugating 
automorphism. 

 Suppose then that $y\in X$, $C$ is a connected
component of $\G_{y^\perp}$ and that  $\phi=\a_{C,y}$. 
 Now let $x\in X$. If  $C\cap \ad(x)=\nul$ then
$G(\ad(x))\phi=G(\ad(x))$, so we may assume that $C\cap \ad(x)\neq \nul$. 
If $\ad(x) \subseteq C\cup y^\perp$ then $G(\ad(x))\phi=G(\ad(x))^{y}$, as 
required. This leaves the case where   $C\cap \ad(x)\neq \nul$ and
$\ad(x) \nsubseteq C\cup y^\perp$. In this case 
%
%suppose first that
%$y\notin \ad(x)$. Then there exists $u\in x^\perp\bs x$ such that $[y,u]\neq 1$.
%Thus $u\notin y^\perp$ and $u\in x^\perp\bs x$; so $[u,v]=1$  for all
%$v\in \ad(x)$. This means that $\ad(x)\bs y^\perp$ is contained in some connected
%component of $\G_{y^\perp}$. As $\ad(x) \cap C\neq \nul$ it follows
%that $\ad(x) \subseteq C\cup y^\perp$, so this case does not occur. 
%
%Therefore,
%if $C\cap \ad(x)\neq \nul$ and
%$\ad(x) \nsubseteq C\cup y^\perp$ then 
$y\in \ad(x)$, from Lemma \ref{lem:fixad},  and  
$z\phi=z$ or $z\phi = z^y$, for all $z\in \ad(x)$, so 
$G(\ad(x))\phi=G(\ad(x))$. % in this case.
\begin{comment}
We consider the cases $x=y$, $y\in x^\perp\bs x$ and $y\in \ad(x)\bs x^\perp$ 
separately. If $x=y$ then $x\in \ad(x)$ implies that $G(\ad(x))\phi \subseteq
G(\ad(x))$; and as $\phi$ is an automorphism we have $G(\ad(x))\phi =G(\ad(x))$.
If $y\in x^\perp\bs x$ then $[u,y]=1$, for all $u\in \ad(x)$, so 
 again $G(\ad(x))\phi =G(\ad(x))$.
\end{comment}
%%%%%%%%%end insert

 The case where $\phi\in \Inv$ is straightforward
(and $f_x=1$, for all $x$, in this case). Suppose then that $\phi$ is
 a transvection; more precisely let $y,z\in X$ with 
$y^\perp\bs y\subseteq z^\perp$
and $\phi=\tr_{y^{\e},z}$, where $\e\in \{\pm 1\}$. 
Let $x\in X$. If $y\notin \ad(x)$ then $\phi$ is the identity on $G(\ad(x))$ so
we may assume that $y\in \ad(x)$. In this case we have $z\in \cl(z)\subseteq
\ad(y)\subseteq \ad(x)$. Hence $G(\ad(x))\phi=G(\ad(x))$, as required.
%Finally, if $\phi\in \cP_\sharp$ then $x\phi =x$, for all $x\in X_J$ and
%as $\ad(x)=X$, for all $x\in X_S$, the result holds in this case too. 
\end{proof}
%\end{vnr}
\begin{rem}\label{rem:aut*}
Note that the proof of this proposition shows that 
if $\phi$ is in the subgroup of $\oAut(G)$ generated by $\Inv$ and  $\Tr$ 
%and $\cP_\sharp$
 then
$G(\ad(x))\phi=G(\ad(x))$, for all $x\in X$.
\end{rem}
\begin{prop}\label{prop:autdecomp}
$\oAut(G)$ is a normal subgroup of $\Aut(G)$ and the latter decomposes as a semi-direct
product $\Aut(G)\cong\oAut(G)\rtimes \Aut^\G_\cmp(G)$.
\end{prop}
%\begin{vnr}
\begin{proof}
To see that $\oAut(G)$ is normal in $\Aut(G)$ it
   is only necessary to check that $\theta^{-1} \phi \theta\in \oAut(G)$
where $\theta \in \Aut^\G(G)$ and $\phi$ is a generator of $\oAut(G)$.
It is straightforward to check from the definitions that if 
$\theta \in \Aut^\G(G)$ then $\theta$ acts 
by conjugation 
on the generators 
of $\oAut(G)$ as follows. If $\i_z\in \Inv$ then
$\i_z^\theta=\i_{z\theta}$. If $x\in X^{\pm 1}$, $y\in Y$ and 
$\tr_{x,y}\in \Tr$  then
$\tr_{x,y}^\theta=\tr_{x\theta,y\theta}$. If $\a_{C,y}$ is an
elementary conjugating automorphism then, since $\theta$ restricts to
 a graph automorphism of $\G$, it follows that $C\theta$ is a connected 
component of $\G_{(y\theta)^\perp}$. Furthermore 
$\a_{C,y}^\theta=\a_{C\theta,y\theta}$. Therefore $\oAut(G)$ is normal. 
%If $\w\in \cP_\sharp$ then $[\theta, \w]=1$.

Next we show that $\oAut(G) \cap \la \cP^\G_\cmp\ra=\{1\}$. 
 From Proposition \ref{prop:aut*}, for all $x\in X$ and $\phi \in \Aut^*(G)$,
we have $x\phi=w^g$, for some $w\in G(\ad(x))$ and $g\in G$. 
This
means that the exponent sum of $y\in X$ in $x\phi$ 
is zero unless $y\in \ad(x)$.
\begin{comment}
From Remark \ref{rem:aut*} above, 
if $\phi$ is a transvection or inversion 
 then $\ad(x)\phi\in G(\ad(x))$, for
all $x\in X$. Moreover, if $\psi$ is a basis-conjugating automorphism then,
for all $y\in \ad(x)$, we have $y^\phi=w^g$, for some $w\in G(\ad(x))$ and
$g\in G$ (Proposition \ref{prop:aut*}).  
$x\psi$ belongs to a conjugate of $G(\ad(x))$, by some element of $G$.
Therefore, for any $x\in X$ and any $\theta\in \oAut(G)$ the image
$g=x\theta$ is a product of conjugates of elements of $G(\ad(x))$. This
means that the exponent sum of $y\in X$ in $g$ is zero unless $y\in \ad(x)$.
\end{comment}

We claim that if $\theta\in \Aut^\G(G)$ then, for all $x\in X$,
 $x\theta =y$ and 
 $\ad(x)\theta=\ad(y)\theta$, for some $y\in X$ with $h(y)=h(x)$. To see
this note that $\theta \in \Aut^\G(G)$ implies $\theta|_X$ is in
$\Iso(\G)$ so $\ad(x)\theta =\ad(x\theta)$. If $x\theta =y$ then
it follows from Lemma \ref{lem:ad1} and induction on the height $h(x)$ of
$x$ that $h(y)=h(x)$, and the claim follows.

Now if $\theta$ is a non-trivial element of $\Aut^\G_\cmp(G)$ there is 
$x\in X$ such that $[x]\theta=[y]$ with $[y]\neq [x]$. Without loss
of generality we may assume $x\theta =y$; so the exponent sum of
$y$ in $x\theta$ is non-zero. 
As $h(x)=h(y)$ it follows
 from Lemma \ref{lem:ad1} that $[y]\cap \ad(x)=\nul$ and 
so $y\notin \ad(x)$. 
Combining this with the above we have 
$\oAut(G)\cap\Aut^\G_\cmp(G) =1$. As $\Aut(G)$ is generated by
$\Inv\cup\Tr\cup \LInn\cup \cP^\G_\cmp$ %and $\i_\cmp$ is an automorphism
%of $\Aut^\G_\cmp(G)$ to $\Aut^\G_\cmp(G)$ 
 it follows not only that $\oAut(G)$ is normal 
but that $\Aut(G)$ decomposes
as a semi-direct product, as claimed.
\end{proof}
%\end{vnr}

\subsection{Decomposition over Connected Components}

In this section we use the theory of automorphisms
of free products developed in \cite{FR1,FR2}, \cite{Gilbert87} and \cite{CollinsGilbert} to give a presentation
of $\Aut(G)$, and describe its structure, in terms of automorphisms  
of the groups corresponding to the connected components of $\G$. 
A presentation for the automorphism group of a 
free product in terms of presentations 
for automorphism groups of the factors  is given in \cite{FR1,FR2}  
and reformulated
in \cite{Gilbert87}. Using the latter we construct a 
presentation for $\Aut(G)$, in terms of presentations of automorphism groups
of factors of $G$,  appropriate to our particular setting.

\begin{comment}
We use a 
subset of the  set of generators of  \cite{Laurence95} described in Section
\ref{section:gens} above. 
 First we make a more explicit choice of the injective homomorphism 
$\i_\cmp$ defined in Proposition \ref{prop:autgen1}. 
\begin{defn}\label{defn:autcomp}
Fix a total order $<$ on $X$. An element $\phi\in \Aut^\G(G)$ is said
to be {\em admissible} if the following condition holds, for all $u\in X$.
If $v\in [u]$ then 
$u<v$ if and only if $\phi(u)<\phi(v)$. Let $\Aut^\cmp(\G)$ denote the
subgroup of admissible automorphisms of $\Aut^\G(G)$.
Then  the restriction of $\pi_\cmp$ to $\Aut^\G_\cmp(\G)$ 
is an automorphism onto $\Aut^\G(G^\cmp)$ and we 
may take $\i_\cmp$ to be the inverse of this restriction. 
\end{defn}
When 
$\Omega$ is a connected component of $\G$ we shall abuse notation by
regarding $\Aut^\cmp(\Omega)$ as a subgroup of $\Aut^\G(G)$. 
\end{comment}

\begin{defn}\label{defn:ourgens}
Let $\G$ have isomorphism type given by \eqref{eq:isomtype} and,
as in Section \ref{subsec:graphaut}, let 
$\G_{i,j}$ be the connected components of $\G$, where $\G_{i,j}\cong \Omega_i$,
for $1\le j\le m_i$ and  $0\le i\le d$. Let 
\begin{itemize}
\item $X_{i,j}$ be the vertex set
of $\G_{i,j}$, 
\item let $S=\{(0,j):1\le j\le m_0\}$, 
\item let 
$X_S=\cup_{s\in S}X_s$, the 
set of isolated vertices of $\G$,  and
\item  
let $J=\{(i,j):1\le i\le d, 1\le j\le m_i\}$. 
\end{itemize}
\be[(i)]
\item  Define the following sets of automorphisms which preserve 
subgroups generated by connected
components of $\G$ and fix all elements of $X_{i,j}$ when $j\neq 1$: let
\be[(a)]
\item 
$\Invi(G)=\{\i_x\in \Inv|x\in X_{i,1},0\le i\le d\}$;
%\item let $\Comp_i=\{\phi\in \Comp| % contained in prev section
\item  $\Tri(G)=\{\tr_{x,y}\in \Tr|x\in X_{i,1}^{\pm 1}, y\in X_{i,1}^{\pm 1}, 
1\le i\le d\}$;
\item  $\LInni(G)=\{\a_{C,x}\in \LInn| x\in X_{i,1}^{\pm 1},C\subseteq X_{i,1},
1\le i\le d\}$;
%\item $\cP_{*,\int}=\Inv\cup \Tri\cup \LInni$, 
\item $\cP_{\int}(G)=\Invi(G)\cup \Tri(G)\cup \LInni(G)$.
\ee
%where $\cP_{\cmp,d^\sharp}=\nul$, if $d^\sharp>d$.
\item Define the following sets of automorphisms which do not 
preserve subgroups generated by connected
components of $\G$. Let
\be
\item
 $\Tre(G)=\{\tr_{x,y}\in \Tr|x\in X_S^{\pm 1}, y\in X^{\pm 1}\}$; 
\item\label{it:LInn} $\LInne(G)=\{\a_{C,y}\in \LInn|C=X_j,y\in X_k^{\pm 1},  j\in J,k\in S\cup J, k\neq j\}$;
\item $\cP_{\ext}(G)=\Tre(G)\cup\LInne(G)$.
\ee
\ee
Finally define $\cP(G)=\cP^\G_\cmp(G)\cup \cP_{\int}(G)\cup \cP_{\ext}(G)$. 
\end{defn}
When the group $G$ in question is clear from the context we often 
drop the argument $G$ from
 these definitions, writing $\cP$ for $\cP(G)$, and so on.
\begin{rem}
\be
\item
 The conditions on $\Tr$ imply that $\Tre$ is empty unless there exists 
  an isolated vertex  $x$ 
 of $\G$, in which case there is an automorphism $\tr_{x,y}\in \Tre$, for
all $y\in X\bs \{x\}$.  
\item 
If $s\in S$ and $X_s=\{x\}$, $z\in X^{\pm 1}$, $z\neq x^{\pm 1}$,  
then $\Aut(G)$ contains
the automorphism $\a_{X_s,z}$, but also contains $\tr_{x^{\pm 1},z}$; and 
$\a_{X_s,z}=\tr_{x,z}\tr_{x^{-1},z}$. Hence we make the restriction $j\in J$;
i.e. 
$|X_j|\ge 2$, in Definition \ref{defn:ourgens} \ref{it:LInn} above.
\item In \cite{Gilbert87} elements of $\Tre\cup \LInne$ %$\cP_E$ 
are called 
Whitehead automorphisms. 
\ee
\end{rem}
\begin{prop}\label{prop:autgen}
The set $\cP$ generates $\Aut(G)$.
\end{prop}
%\begin{vnr}
\begin{proof}
In the light of %Lemma \ref{lem:autgen} and 
Proposition \ref{prop:autgen1}, 
it suffices to 
 show that every automorphism in $\Inv\cup \Tr\cup \LInn$ 
belongs to the subgroup generated by $\cP$.  
%$\cP^\G_\cmp$, $\cP_\sharp$, $\Inv_F$, 
%$\Tr_F$, $\LInn_F$, 
%$\Tr_E$ and 
%$\LInn_E$. 
%If $i\in S$ and $x=x_i^{\pm 1},$ $x_i\in X_S$ then 
%$\i_x\in \Inv_i\subseteq \Inv_F$ and
%$\tr_{x,y}\in \Tr_E$, for all $y\in X^{\pm 1}$. In this case there 
%are no automorphisms of the form $\a_{C,x}$ in $\LInn_F$ and for 
%all $j\in J$ we have $\a_{{X_j},x} =\tr_{
%
For all $i, j$ with $1\le i\le d$ and 
$1<j\le m_i$ the automorphism $\w^i_{1,j}$
belongs to $\cP^\G_{\sym,i}%(G_{i,*})
\subseteq \cP^\G_\cmp$  
 and  for all $x_j, y_j\in X_{i,j}$ there are 
$x_1,y_1\in X_{i,1}$ such that $x_1=x_j\w^i_{1,j}$ and 
$y_1=y_j\w^i_{1,j}$. Then 
we have $\i_{x_j}=(\w^i_{1,j})^{-1}\i_{x_1}\w^i_{1,j}$,
$\tr_{x_j^{\e},y_j}=(\w^i_{1,j})^{-1}\tr_{x_1^{\e},y_1}\w^i_{1,j}$ and 
$\a_{C_j,x_j}= (\w^i_{1,j})^{-1}\a_{C_1,y_1}\w^i_{1,j}$, where 
$C_j=C_1\w^i_{1,j}$. As $\i_{x_1}$, $\tr_{x_1^{\e},y_1}$ and 
$\a_{C_1,y_1}$ are all elements of $\cP$ it follows that 
  $\Inv\cup \Tr\cup \LInn$ is contained in the 
subgroup generated by $\cP$, as required.
%{\bf The proof that $\oLInn$ is enough is left for now: till the
%details of $\oLInn$ have been sorted out.} Given this fact the result 
%follows from
%the definitions.
\end{proof}
%\end{vnr}
We extend the notation of Definition \ref{defn:graphaut},  
for graph automorphisms of the factors
of $G$, to cover all automorphisms
of the factors. 
\begin{defn}\label{defn:factgen}
For $i=0,\ldots ,d$, let $\Aut(G_{i,1})$ denote the subgroup
of $\Aut(G)$ consisting of elements $\phi\in \Aut(G)$ such that 
$x\phi=x$, for all $x\in X\bs X_{i,1}$ and 
$G(\G_{i,1})\phi\subseteq G(\G_{i,1})$. 
\end{defn}
Before choosing generators and relators for $\Aut(G)$ note that 
\[\Aut^\G_\cmp(G_{i,1})\le \Aut(G_{i,1})\cong \Aut(G(\Omega_i))\]
and that, from Proposition \ref{prop:autgen1}, $\Aut(G_{i,1})$ is generated
by $\left(\cP_{\int}\cap \Aut(G_{i,1})\right)\cup \cP^\G_{\cmp,i}$. 
\begin{defn}\label{defn:factpres}
Choose presentations 
\begin{gather*}
\la \cP^\G_{\sym,i}%(G_{i,*})
| \cR^\G_{\sym,i}%(G_{i,*})
\ra 
\textrm{ for }\Aut^\G_{\sym}(G_{i,*}),\, 0\le i\le d,\textrm{ and } \\[.5em]
\la \cP^\G_{\cmp,i}|\cR^\G_{\cmp,i}\ra \textrm{ for }\Aut^\G_\cmp(G_{i,1}),\, 1\le i\le d. 
\end{gather*} 
(For notational convenience set $\cP^\G_{\cmp,0}=\cR^\G_{\cmp,0}=\nul$.)
For $0\le i\le d$, let 
\[\cP_i=\left(\cP_{\int}\cap\Aut(G_{i,1})\right)\cup \cP^\G_{\cmp,i},\] 
so $\cP_i$ is a set of generators
for $\Aut(G_{i,1})$. Choose a presentation $\la \cP_{i}|\cR_{i}\ra$ for $\Aut(G_{i,1})$ such
that $\cR_{i}\supseteq %\cR_{\sym,i}(G_{i,*})\cup
\cR^\G_{\cmp,i}$, the relators chosen (in the appendix) for $\Aut^\G_\cmp(G_{i,1})$. 
\end{defn}

\begin{prop}\label{prop:presentation}
$\Aut(G)$ has a presentation $\la \cP|\cR\ra$, where $\cP$ is 
given in Definition \ref{defn:ourgens} and $\cR$ is defined in
Definition \ref{defn:ourrels} below.
\end{prop}
\begin{defn}\label{defn:ourrels}
Let $\cR$ be the union of the following sets.
\be[(i)]
\item\label{it:ourrels1} $\cR^\G_{\sym,i}%(G_{i,*})
$, for $i=0,\ldots, d$.
\item\label{it:ourrels2} $\cR_{i}$, for $i=0,\ldots, d$. 
%\item $\cR_\sharp=\cR_{*,i}\cup \cR_{\sym,d^\sharp}$ if $d^\sharp>d$ and 
%$\nul$ otherwise.
%\item The set $\cR^\prime_\act$ consisting of
%\be[(i)]
%\item
% $[\phi,\psi]=1$, for all $\phi\in \cP_i$, $\psi\in \cP_j$, $i\neq j$ and 
%\item
% all relations of $\cR_\act(G)$ (see Definition \ref{defn:rc}) such
%that $\theta\in \cP_\sym$ and the automorphisms $\tr_{x,y}$
%and $\a_{C,y}$ mentioned there belong to $\Tr_i$ or $\LInn_i$, for some
%$i$.
%\ee 
\item\label{it:ourrels3} The sets 
\begin{align*}
\cW_j&=\{[\w^j_{a,b},p]:p\in \cP_{j},2\le a<b\le m_j\}\\
&\cup \{[p, \w^j_{1,a}q\w^j_{1,a}]: p,q\in \cP_{j},2\le a\le m_j \}\\
&\cup \{[\w^j_{1,a}p\w^j_{1,a},\w^j_{1,b}q\w^j_{1,b}]: 
p,q\in \cP_{j},2\le a<b\le m_j  \},
\end{align*}
for $j=0,\ldots ,d$
\item \label{it:ourrels4}
$\cD=\{[p,q]|
p\in \cP_{i}\cup \cP^\G_{\sym,i}%(G_{i,*})
, q\in \cP_{j}\cup \cP^\G_{\sym,j}%(G_{j,*})
, 
\textrm{ with } 0\le i<j\le d\}$.
\begin{comment}
The sets 
\be[(i)]
\item $\cD_\cmp$ (Definition \ref{defn:graphaut}); 
\item  $\cD_*=\{[p,q]|p\in \cP_{*,i}, q\in \cP_{*,j}\textrm{ with } 1\le i<j\le d^\sharp\}$ and 
\item  $\cD_\act=\{[p,q]|
(p,q)\in \cP_{*,i}\times(\cP_{\cmp,j}\cup \cP_{\sym,j}%(G_{j,*})
)
\textrm{ or } 
(p,q)\in (\cP_{\cmp,j}\cup \cP_{\sym,j}%(G_{j,*})
) \times\cP_{*,i}
\textrm{ with } 1\le i<j\le d^\sharp\}$.
\ee
Define $\cD=\cD_\cmp\cup \cD_*\cup \cD_\act$. 
\end{comment}
\item The set of relations $\cR$i, for i=1,\ldots ,11, below. 
\ee
\end{defn}
Note that if $\tr_{x,y}\in \Tre$ then necessarily $x\in X_i^{\pm 1}$, 
for some $i\in S$  and 
$y\in X_j^{\pm 1}$, where $j\in S\cup J$ with  $j\neq i$. Similarly, if  $\a_{C,x}\in \LInne$ then
$x\in X_i^{\pm 1}$, for some $i\in S\cup J$ and $C=\G_j$, where $j\in J$,
with  $j\neq i$.
In the relations below all the  transvections $\tr_{\cdot,\cdot}$ and 
elementary conjugating automorphisms $\a_{\cdot,\cdot}$, that 
are  mentioned explicitly, belong to $\Tre$ or $\LInne$, respectively. 
The relations are defined for
all $u,v,x,y,z\in X^{\pm 1}$ and $i,j,k,l\in S\cup J$, for
which the preceding conditions hold. 
To ease the description of conditions placed on such automorphisms, 
for $a\in \cup_{j\in J} G(\G_j)\cup X_S^{\pm 1}$,  
we define
\begin{equation*}\label{eq:vat} 
\vat a= 
\left\{
\begin{array}{ll}
j & \textrm{ if } a\in G(\G_j)\\
x & \textrm{ if } a=x^\e, \textrm{ where } x\in X_S, \e={\pm 1}
\end{array}
\right. 
.
\end{equation*} 

For $y\in X_j^{\pm 1}$, where $j\in J$, we 
 write
$\g_y(j)$ for the automorphism conjugating every element of $X_j$  by $y$ 
and fixing all elements of $X\bs X_j$. That is
\[x\g_y(j)=
\begin{cases} 
x^y& \textrm{ if } x\in X_j\\ 
x &\textrm{ if } x\in X\bs X_j
\end{cases},\] so
$\g_y(j)$ is equal to the product over all connected components $C$ of
$(\G_j)_{y^\perp}$ of the automorphisms $\a_{C,y}$. 
\be[{$\cR$}1., ref={$\cR$}\arabic*]
\item\label{it:R1} $[\tr_{x,y},\tr_{u,v}]=1$, if either
\be[(i)]
\item $u=x^{-1}$ %and $v=y^{-1}$ 
or 
\item 
$\vat x \neq \vat u$, $\vat x\neq \vat v$ and $\vat y\neq \vat u$.  
%$x\neq u^{\pm 1}$, $x\neq v^{\pm 1}$ and $y\neq u^{\pm 1}$.
\ee 
\item\label{it:R2} $[\tr_{x,y}^{-1},\tr_{u,x}^{-1}]=\tr_{u,y}^{-1}$,  
if $\vat x\neq \vat u$ and $\vat y \neq \vat u$. 
%($x\neq u^{\pm 1}$, $x\neq y^{\pm 1}$ and $y\neq u^{\pm 1}$).
\item\label{it:R3} $\tr_{x,y}^{-1}\tr_{y,x}\tr_{x^{-1},y}=
\w^0_{i,j}\w^0_{1,j}\i_z\w^0_{1,j}$, 
where $x\in X_{0,i}^{\pm 1}\subseteq X_S^{\pm 1}$ and $y\in X_{0,j}^{\pm 1}\subseteq X_S^{\pm 1}$, $i\neq j$ 
and $X_{0,1}=\{z\}$. (If $j=1$ the right hand side of this relation is
 replaced by  $\w^0_{i,1}\i_z$.)
\item\label{it:R4} $[\a_{X_i,x},\a_{X_j,y}]=1$, if 
$\vat x,\vat y\notin \{i,j\}$, $i\neq j$ and $i,j\in J$.
%$x,y \notin X_i\cup X_j$,  $i\neq j$ and $i,j\in J$. 
\item\label{it:R5} $[\a_{X_j,x},\a_{X_i,y}\a_{X_j,y}]=1$, if 
$i\neq j$ 
and %either 
%\be[(i)] 
%\item 
$\vat x = i$% or
.
%\item $x,y\in X_k$, $[x,y]=1$, $k\notin \{i,j\}$ and  $i,j,k\in J$.
%\ee
\item\label{it:R6} $[\tr_{x,y},\a_{X_l,z}]=1$, if $\vat y\neq l$ and 
$\vat x\neq \vat z$.
\item\label{it:R7} $[\tr_{x,y}^{-1},\a_{X_l,x}^{-1}]=\a_{X_l,y}^{-1}$, 
if $\vat y\neq l$.
\item\label{it:R8} $[\tr_{x,y},\a_{X_i,z}\tr_{x,z}]=1$, if $\vat y=i$.
\item\label{it:R9} $\tr_{x,y}\a_{X_i,x}=\a_{X_i,x}\tr_{x^{-1},y}^{-1}\g_y(i)^{-1}$,
if $\vat y=i$. 
\item \label{it:last}
Let $x\in X_S^{\pm 1}$, $y,z\in X^{\pm 1}$, $i\in S\cup J$ be such that   
$\vat y=\vat z =i$,  with $\al(y)\cap \al(z)=\nul$ and  
$[y,z]=1$. Let  $u\in X$ and  $j\in J$, where $i\neq j$.  Then 
\be[(i)]
\item\label{it:last2} 
$\tr_{x,u}^{-1}=\tr_{x,u^{-1}}$%
%if  $x\neq u$ then 
%$\tr_{x,u}^{-1}=\tr_{x,u^{-1}}$
;
\item \label{it:last1}
$[\tr_{x,y},\tr_{x,z}]=1$;
\item\label{it:last3}
$\a_{X_j,u}^{-1}=\a_{X_j,u^{-1}}$
and 
\item\label{it:last4}
%if $i\neq j$ 
%then 
 $
[\a_{X_j,y},\a_{X_j,z}] =1.
$
\ee
\item \label{it:lastdash}
Let $y\in X^{\pm 1}$ and $\theta\in \cP^\G_\cmp\cup \cP_{\int}$ 
and let $y_1\cdots y_k$ be a word representing
$y\theta$, with $y_i\in X^{\pm 1}$. 
\be[(i)]
\item 
Let $x,z\in X_S^{\pm 1}$ such that  $z=x\theta$. Then
\[
\tr_{x,y}\theta =\theta \tr_{z,y_k}\cdots \tr_{z,y_1},
\]
 for 
all $\tr_{x,y}\in \Tre$.
\item
Let $i,j\in J$ and $\G_j\theta = \G_i$. If $\vat y\neq j$ then, with $C=V(\G_j)$ and $D=V(\G_i)$, 
\[
\a_{C,y}\theta =\theta \a_{D,y_k}^{\e_1}\cdots \a_{D,y_1}^{\e_1},
\]
 for all 
$\a_{C,y}\in \LInne$. 
\ee
\ee 
The proof of this theorem is left to the appendix.

In the case where $m_0=0$, that is, no component of $\G$ is an isolated vertex, the set $\Tre$ is empty and the 
the relations of this presentation reduce to the union of the sets  
$\cup_{i=1}^d\cR_{\sym,i}%(G_{i,*})
$, $\cup_{i=1}^d\cR_i$, 
$\cW$ and  $\cD$, given in \ref{it:ourrels1}--\ref{it:ourrels4} of Definition \ref{defn:ourrels}, together
with 
%$\cR^\prime_\act$,
\ref{it:R4}, \ref{it:R5}, \ref{it:last}\ref{it:last3}, 
\ref{it:last4} and 
\ref{it:lastdash}(ii). In this case $\Aut(G)$ decomposes as a semi-direct
product $\Aut(G)=\la\LInne\ra \rtimes \la \cP^\G_\cmp\cup \cP_{\int}\ra$; and 
$\la \LInne\ra$ is called the Fouxe-Rabinovitch kernel and 
denoted $\FR(G)$ (see \cite{CollinsGilbert} for
more details).
%below here is old stuff
 The structure of 
$\Aut(G)$ is then given by 
the following (special case of a) theorem from \cite{CollinsGilbert}.
\begin{theorem}[{\it cf.} \cite{CollinsGilbert}, Theorem C]
\label{theorem:FRker}
Suppose that no component of $\G$ is an isolated vertex. 
Define $\bar G=G_1\times \cdots \times G_n$ and $\FR(G)= \la \LInne\ra$. 
Then 
$\FR(G)$ is the kernel of the
canonical map from $\Aut(G)$ to $\Aut(\bar G)$. Moreover $\FR(G)$ has a normal 
series 
\[
1<P_{n-1}<\cdots <P_2<\FR(G)
\]
such that, setting $\FR_i(G)=\FR(G)/P_i$, 
\be[(i)]
\item $\FR(G)=P_i \rtimes \FR_i(G)$, 
\item $\FR_i(G)=\FR(G_1\ast \cdots \ast G_i)$ and
\item all the $P_i$ are finitely generated.
\ee
\end{theorem}

In the light of the results of this section we may when necessary reduce to
the study of $\Aut^\G(G)$ where 
$\G$ is a connected graph. In particular, to give an explicit presentation
of $\Aut(G)$ it remains to determine the sets $\cR_i$ of Definition \ref{defn:factpres}.  
%%%%%%%%%%%%%%%%%%
\subsection{Conjugating Automorphisms}\label{sec:conj}

The subgroup of basis-conjugating automorphisms, which we consider here,  plays an important role in the structure of $\Aut(G)$ and has a rich and complex structure, even in
the case of free groups: see for example 
\cite{mccool86,GutierrezKrstic98, orlandi00, Bardakov03}.  
%We restrict attention in this section to the case where $\G$ is a connected
%graph. 
We shall consider several subgroups of the  basis-conjugating
automorphisms $\Conj(G)=\la \LInn(G)\ra$  which we now define.

Let $x\in X$ and, as usual, denote by $\G_x$ the full subgraph of $\G$
generated by $X\bs \{x\}$ and note that 
%. Given a connected component of $\G_x$ 
%with vertex set $C$ 
%let $\G^C_x$ denote the full subgraph on $C\cup x^\perp$.
%If 
if $y\in X$ lies in a connected component 
$C$ of $\G_x$ then $y^\perp \subseteq C\cup \{x\}$. 
\begin{defn}\label{defn:aConj}
Let $x\in X$ %, let $\Omega$ be the connected component of $\G$ containing $x$
  and let $C$ %\subseteq \Omega$
 be a connected component of $\G_x$. 
Then the 
automorphism $\beta_{C,x}$ given by
\[
y\beta_{C,x}=
\left\{
\begin{array}{ll}
y^x,& \textrm{ if } y\in C\\% \textrm{ or } y\notin \Omega\\
y, &\textrm{ otherwise }
\end{array}
\right.
\]
is called an {\em aggregate conjugating automorphism}. The
subgroup of $\Conj(G)$ generated by all aggregate conjugating automorphisms
is denoted $\aConj(G)$. 
\end{defn}
%%%%%%%%%%
\begin{comment}
In the following definition, if $\G$ is not connected then isolated 
vertices need careful treatment: because if $x\in X_S$ then $\ad(x)=X$. 
To avoid the difficulties this introduces the condition of the next definition
is imposed only in the connected component containing each vertex. 
\end{comment}
\begin{defn}\label{defn:NConj}
An element $\phi\in \Conj(G)$ is said to be a \emph{normal conjugating
automorphism} 
if, 
for every element $x\in X$, there exists $f_x\in G$ such that
$y\phi=y^{f_x}$, for all $y\in \ad(x)$. 
The subgroup of all
normal conjugating automorphisms is denoted $\NConj(G)$. 
\end{defn}

\begin{defn}\label{defn:VConj}
An element $\phi\in \Conj(G)$ is said to be a \emph{vertex conjugating
automorphism} 
if, for every element $x\in X$ there exists $f_x\in G$ such that
$y\phi=y^{f_x}$, for all $y\in [x]$. The subgroup of all
vertex conjugating automorphisms is denoted $\VConj(G)$. 
\end{defn}

If $\G$ is compressed ($\G=\G^\cmp$) then $\VConj(G)=\Conj(G)$. 

\begin{defn}\label{defn:singular}
An
elementary conjugating automorphism $\a_{C,u}$, where $u=x^{\pm 1}$, for
some $x\in X$ is called an {\em elementary singular 
conjugating automorphism} if 
 $C=\{y\}$, for some $y\in X$, and the set of all such elementary
conjugating automorphisms is denoted $\SLInn=\SLInn(G)$. 
 The subgroup of $\Conj(G)$ generated by $\SLInn(G)$  
%elementary singular conjugating automorphisms 
% together with inner automorphisms 
is called {\em singular} and denoted $\sConj(G)$. 
%If $C=[y]=\{y\}$ or $C=\{y\}$ and $[y]=\{x,y\}$  then $\a_{C,x}$ is called an {\em elementary isolated
%conjugating automorphism}. The set of all such automorphisms
%is denoted $\ILInn=\ILInn(G)$ and the subgroup of $\sConj$ generated
%by $\ILInn$ is called {\em isolated} and denoted $\iConj=\iConj(G)$. 
\end{defn}

\begin{defn}
Let $\Tr_\perp=\{\tr_{x^{\e},y^{\d}}\in \Tr|x\in y^\perp, \e,\d=\pm 1\}$ and 
$\Tr_\lk=\{\tr_{x^{\e},y^{\d}}\in \Tr|x\notin y^\perp, \e,\d=\pm 1\}$. 
%For
%$\tr_{x^{\e},y}\in \Tr$ write $\trs_{x^{\e},y}=\tr_{x^{-\e},y}$. 
\end{defn}

\begin{defn}\label{defn:isols}
\begin{itemize}
\item
If $x$ and $y$ are vertices of $X$ such that 
$x^\perp \cap y^\perp = y^\perp\bs y$ 
then we say that 
$x$ {\em dominates} $y$. 
\item
The set of all vertices dominated by $x$ is denoted 
$\Isol(x)=\{u\in X\,|\, x \textrm{ dominates } u\}$. 
\item 
The set of all dominated
vertices is denoted $\Isol(\G)=\cup_{x\in X}\Isol(x)$. 
\item
For fixed $y\in X$ the 
set of all $x$ such that $y\in \Isol(x)$ and $[y]\neq [x]$ 
is the \em{outer admissible} set
of $y$, denoted $\ado(y)$. 
\end{itemize}
\end{defn}
From the definition and Lemma \ref{lem:ad0} \ref{it:ad12}
it follows that $x$ dominates $y$ if and only if $[x,y]\neq 1$ and
$\ad(x)\subseteq \ad(y)$. Thus $\ado(y)=\{x\in \ad(y): x\notin [y]\cup y^\perp\}$.

If $\a_{C,x}\in \SLInn(G)$ then $C=\{y\}$ is a connected component
of $\G_{x^\perp}$ so $y^\perp\bs y\subseteq x^\perp$ and $y\notin x^\perp$.
Therefore $x$ dominates $y$ and   $\tr_{y,x}\in \Tr_\lk$ and 
$\a_{C,x}=\tr_{y,x}\tr_{y^{-1},x}$.
Hence $\sConj$ is the subgroup of $\Aut(G)$ generated by the set 
$\{\tr_{y,x}\tr_{y^{-1},x}|x \textrm{ dominates } y\}=\SLInn$.

\begin{defn}\label{defn:collectedconj}
Let $x,u\in X$ such that $x$ dominates $u$ and let $[u]\bs\{x\}=\{v_1,\ldots ,v_n\}$.
The conjugating automorphism 
\[\a_{[u],x}=\prod_{i=1}^n\a_{\{v_i\},x}\]
is called a {\em basic collected conjugating automorphism} and the
set of all    basic collected conjugating automorphisms is denoted 
$\CLInn=\CLInn(G)$. The subgroup of $\Conj(G)$ generated by
$\CLInn(G)$ is denoted $\CConj=\CConj(G)$. 
\end{defn}

\begin{defn}\label{defn:regvert}
\begin{itemize}
\item
The set of {\em regular elementary conjugating automorphisms}
is $\RLInn=\RLInn(G)=(\LInn(G)\cap \VConj(G))\backslash \SLInn(G)$. 
\item
The set of 
{\em basic  vertex 
conjugating automorphisms} is 
 $\VLInn=\VLInn(G)=\RLInn(G)\cup \CLInn(G)$. 
\end{itemize}
\end{defn}

We record some straightforward properties of  these definitions
in the following lemma. 
\begin{lemma}\label{lem:elconj} 
Let $\Gamma$ be a %connected 
 graph. 
\be[(i)]
\item\label{it:elconj1} 
\be
\item If $\G$ has an isolated vertex then $\Inn=\NConj$ and 
\item if $\G$ has no isolated vertex then $\aConj\le \NConj$.
\ee
In all cases 
\[\Inn\le \aConj\le \VConj\le \Conj\]
and 
\[\Inn\le \NConj\le \VConj\le \Conj.\]
\item\label{it:elconj2} $\VLInn\subseteq \VConj$. 
%\{\a_{u,x}\in \SLInn\,|\, |[u]|=1\}$  and 
\begin{comment} 
\item\label{it:elconj3} $\CConj\subseteq \sConj\cap \VConj$.
\end{comment}
\item\label{it:elconj4} %If $\G$ is not the path graph $P_2$ of length $2$ (i.e. $3$ vertices and $2$ edges) then 
If $\phi\in \sConj$ then $x\phi =x^{f_x}$, 
where $\al(f_x)\subseteq \ad(x)$, for all $x\in X$.
\ee
\end{lemma}
\begin{proof} 

~\ref{it:elconj1}~
It is immediate from the definitions that $\Inn \le \aConj$,  
$\Inn \le \NConj$ and $\VConj\le \Conj$. 
That $\aConj \le \VConj$ follows from the fact  that, if $x,y\in X$ then
$[y]\subseteq C\cup{x}$, for some connected component $C$ of $\G_x$. As
$[x]\subseteq \ad(x)$, for all $x$, it follows that $\NConj\le \VConj$.

If $x$ is an isolated vertex then $\ad(x)=X$, so for $\phi\in \NConj$ there
exists $f_x\in G$ such that $y\phi=y^{f_x}$, for all $y\in X$. Hence, 
in this case  $\NConj=\Inn$. 
Assume then that $\G$ has no isolated vertex. In this case, for all $x\in X$,
the connected component of $\G$ containing $x$ also contains $\ad(x)$. 
To see that $\aConj\le \NConj$ suppose that $u\in X$ and
consider the aggregate conjugating automorphism 
$\b=\b_{C,x}$, where $x\in X$. 
If $x\in u^\perp\backslash u$ then $v\b=v$, for all 
$v\in \ad(u)$, so assume that this is not the case. 
If $x\in \ad(u)$ then $x\notin u^\perp\backslash u$ implies that 
$\ad(u)\subseteq C^\prime\cup \{x\}$, for some component $C^\prime$ of
$\G_x$, so we may also assume that $x\notin \ad(u)$.

%Then 
%if  
%$\ad(u)\cap C\neq \nul$ and $\ad(u)\nsubseteq C$,  
%for some component $C$ of $\G_x$,% 
%and
Now let  $v$ and $w$ be distinct elements of  $\ad(u)$%. Then such that $v\in C$ and 
%$w\notin C$, then every path from $v$ to $w$ in $\G$ must contain
%$x$. However,  
 and $r$ be any element of $u^\perp\backslash u$. Then the 
path $v,r,w$ does not contain $x$; so $v$ and $w$ are  either 
both in $C$ or both outside $C$. %$x\in v^\perp\backslash v$ or $x=w$. 
Hence  $\b_{C,x}$ either fixes every element of $\ad(u)$, or acts as conjugation by $x$ on every element
of $\ad(u)$. Thus all elements of $\aConj$ are normal, as required.

\ref{it:elconj2}~ This follows directly from the definitions and the 
fact that the sets $[x]$ partition $X$, so that $\CLInn \subseteq \VConj$. 

\ref{it:elconj4}~ An induction on the length of $\phi$ as a word in the generators
$\SLInn$ is used. If $\phi$ is trivial
there is nothing to be proved, so assume inductively that the result
holds for words of length at most $n-1$ and that 
$\phi=\phi_0\phi_1$, where $\phi_0$ has length $n-1$ as a word
in $\SLInn^{\pm 1}$ and $\phi_1\in \SLInn^{\pm 1}$, say
$\phi_1=\a_{C,z}$, for some $z\in X^{\pm 1}$ and $C=\{y\}$. Then 
$x\phi_0=x^{f_x}$, where $\al(f_x)\subseteq \ad(x)$, for all $x \in X$. 
Let $x\in X$ and $u\in \ad(x)^{\pm 1}$. Then $u\phi_1=u$ unless $u=y^{\pm 1}$. In the
latter case $y\in \ad(x)$ so $z\in \ad(y)^{\pm 1}\subseteq\ad(x)^{\pm 1}$ and $u\phi_1=u^z$ implies
$\al(u\phi_1)\subseteq \ad(x)$. Thus we have $\al(f_x\phi_1)\subseteq \ad(x)$. 
Now $x\phi=(x\phi_1)^{f_x\phi_1}$
and since $x\phi_1=x^z$ if and only if $x=y^{\pm 1}$ it follows that 
$\al(x\phi)\subseteq \ad(x)$, as required. 
\end{proof}

We shall use the following definition of Laurence \cite{Laurence95}.
\begin{defn}\label{defn:conjlen}
Let $\phi$ be a conjugating automorphism and for each $x\in X$ let
$g_x\in G$ be such that $x\phi=g_x^{-1}\circ x\circ g_x$. The
{\em length} $|\phi|$ of $\phi$ is $\sum_{x\in X} \lg(g_x)$. 
\end{defn}
We shall prove, in Propositions \ref{prop:vconjgens} and 
in a subsequent paper, 
%\ref{prop:nconjgens} below,  
 versions of Theorem \ref{thm:laurence} (i.e. 
Theorem 2.2 of \cite{Laurence95}) appropriate to $\VConj$ and
$\NConj$ 
and to do so make use of Lemma 2.5 and Lemma 2.8 ({\em loc. cit.}) which we
state here for reference. 
\begin{lemma}[{\cite{Laurence95}[Lemma 2.5 \& Lemma 2.8]}]\label{lem:laurlem}
Let $\phi$ be a non-trivial element of $\Conj$ and, 
for each $x\in X$, 
let $g_x\in G$ such that 
$x\phi=g_x^{-1}\circ x \circ g_x$. Then 
\be[(i)]
\item
\label{it:laur25}
there exist $x,y\in X$ and  $\e\in \{\pm 1\}$ such that 
$x^{\e}g_x$ is a right divisor of $g_y$, and
\item
 \label{it:laur28}
if $y,z\in X\bs x^\perp$ such that $[y,z]=1$  and $x^{\e}g_x$ 
is a right divisor of $g_y$
then $x^{\e}g_x$ is a right divisor of $g_z$. 
\ee
\end{lemma}
(As can be seen from the example $\phi=\a_{C,x}^{-1}$ the variable $\e$ taking 
values $\pm 1$ is a necessary part of this lemma.)
\begin{lemma}\label{lem:vertshift} 
Let $\phi\in \VConj$ and for each $y\in X$ let $g_y\in G$ be such that
$y\phi=g_y^{-1}\circ y\circ g_y$. 
\be[(i)]
\item\label{it:vertshift1} If $[y]=[y]_\perp$ then $g_u=g_y$, for all $u\in [y]$. 
\item\label{it:vertshift2} If $[y]=[y]_\lk$ and $|[y]|\ge 2$ then there exists $v\in [y]$ 
and $m_y\in \ZZ$ 
such that $g_u=v^{m_y}\circ g_v$, for all $u\in [y]\bs\{v\}$. 
Moreover if $m_y\neq 0$ then $v$ is the
unique element of $[y]$ with this property and, setting $\e=-m_y/|m_y|$,
$S=[y]\bs\{v\}$ and $\a=\prod_{u\in S} \a_{\{u\},v^{\e}}$ we have
$\a\in \CLInn^{\pm 1}$ and $|\a\phi|<|\phi|$.
\ee
\end{lemma}
\begin{proof}
Since $\phi\in \VConj$, for all $y\in X$, there exists $f_y\in G$ such
that $u\phi=u^{f_y}$, for all $u\in [y]$, and we may
choose an $f_y$ of minimal length with this property. Fix $y\in X$. 
Then $u^{f_y}=u\phi=u^{g_u}$ so
$g_uf_y^{-1}\in C_G(u)$, for all $u\in [y]$. Therefore there are 
$a,b,c\in G$ such that $g_u=a\circ b$, $f_y=c\circ b$ and 
$g_uf_y^{-1}=a\circ c^{-1}\in C_G(u)$. As $g_u$ has no left divisor
in $C_G(u)$ this means that $a=1$ and so $f_y=c_u\circ g_u$, for
$c=c_u\in C_G(u)$. If $[y]=[y]_{\perp}$ then $C_G(u)=C_G(y)$, for all
$u\in [y]$, so in this case $g_y=f_y=g_u$, for all $u\in [y]$. 

Assume
then that $[y]=[y]_{\lk}$, with $|[y]|\ge 2$, 
and let $u,v\in [y]$, $v\neq u$, so $[u,v]\neq 1$.
Suppose $v\in \al(f_y)$. 
Then $f_y=c_v\circ g_v=c_v^\prime\circ v^m\circ g_v$, where 
$c_v^\prime \in G(v^\perp\bs v)$ and $m\in \ZZ$. Then
$u^{f_y}=u^{v^m g_v}$, since $v^\perp\bs v=u^\perp\bs u$. As
$g_v$ has no left divisor in $C_G(v)$ and $[v,u]\neq 1$ we have
$u^{v^m g_v}=g_v^{-1}\circ v^{-m}\circ u \circ v^m\circ g_v$, so
$g_u= v^m\circ g_v$. By choice of $f_y$ we have $c_v^\prime=1$, and
if $m\neq 0$ then no element $u\in [y]$, $u\neq v$, 
can be a left divisor of $v^m\circ g_v$, so the first statement of 
\ref{it:vertshift2} 
 as well as the uniqueness of $v$ follow.  
Moreover $v$ dominates $u$, for all $u\in [y]$, so the 
final statement of \ref{it:vertshift2}  also holds.
\end{proof}
\begin{prop}\label{prop:vconjgens}
$\VConj$ is generated by $\VLInn=\RLInn\cup
\CLInn$ and 
$\VConj\cap \sConj=\CConj$. 
\end{prop}
\begin{proof} That $\la \VLInn\ra\le \VConj$ is Lemma \ref{lem:elconj}
\ref{it:elconj2}. 
 For
the opposite inclusion  
 we use
induction on the length of an automorphism $\phi$ in $\VConj$. 
If $|\phi|=0$ then $\phi=1$ and there is nothing to prove. 
Assume that $|\phi|>1$ and that, for all 
conjugating automorphisms $\psi$ of shorter length, $\psi\in \VConj$ implies
$\psi \in \la\VLInn\ra$. 
If there exists $y\in X$ such that, $[y]=[y]_\lk$, $|[y]|\ge 2$ and, in  the notation of Lemma 
\ref{lem:vertshift},
$m_y\neq 0$, then it follows from that lemma and induction that $\phi\in \la \VLInn\ra$,
as claimed. Hence we assume that either $[y]=[y]_\perp$ or $m_y=0$, and so 
$g_y=g_u$, for all $u\in[y]$ and  for all $y\in X$. 
From Lemma \ref{lem:laurlem} \ref{it:laur25}  there exist $x, y\in X$, 
$\e\in\{\pm 1\}$  
such that $x\phi=g_x^{-1}\circ x\circ g_x$, $y\phi=g_y^{-1}\circ y\circ 
g_y$ and $x^\e g_x$ is a right divisor of $g_y$. 
Suppose that $[x,y]=1$. Then $[x\phi,y\phi]=1$; that is $[x^{g_x},y^{g_y}]=1$. 
If $g_y=a\circ x^\e\circ g_x$, for some $a\in G$, then this implies that $[x,y^{ax^\e}]=1$, from
which it follows that $[x,a]=1$. However, in this case $y^{g_y}$ is not reduced, a contradiction.
Therefore 
 $y\notin x^\perp$,  and so $u\notin x^\perp$, for all $u\in[y]$.%, for all $u\in [y]$.

Let $[y] =\{v_1,\ldots ,v_r\}$ and let $C_1,\ldots ,C_s$ be the components of 
$\G_{x^\perp}$ containing $v_1, \ldots ,v_r$. Then, from Lemma 
\ref{lem:laurlem} \ref{it:laur28},  $x^{\e}g_x$ is a right divisor  
 of $g_c$ for all $c\in C_1\cup \cdots \cup C_s$. 
Let $\a=\prod_{i=1}^s\a_{C_i,x^{-\e}}$. Then $|\a\phi|<|\phi|$. 
We claim that $\a\in \VConj$. Suppose not, so there is some $z\in X$ and 
elements $u,v\in [z]$ such that $u\in C_i$, for some $i$, but
$v\notin \cup_{i=1}^s C_i\cup \{x^\perp\}$. This implies that $u^\perp\bs u=v^\perp \bs v
\subseteq x^\perp$ and, as $u\in C_i$ implies $x\notin u^\perp$, so $x$
dominates $u$. Then $C_i=\{u\}$ so $u\in [y]$ and $[z]=[y]\subseteq 
\cup_{i=1}^s C_i$, a contradiction. Thus no such $z$ exists and $\a\in\VConj$.

If 
 $s=1$ and $|C_1|\ge 2$ then 
$\a\in \RLInn^{\pm 1}$. If $s=1$ and $|C_1|=1$ 
 then $x$ dominates $y$ and $\a\in \CLInn^{\pm 1}$. If $s>1$ then
$x^\perp\supseteq y^\perp\backslash y$ and $x$ dominates every element
of $[y]$. In this case $\a\in \CLInn^{\pm 1}$ again. Hence 
 by
 induction  $\phi\in \la \LInn_R\cup \LInn_C\ra$. 

Suppose then that
  $\phi \in \VConj\cap \sConj$. The first paragraph of the argument above goes through
with $\CLInn$ in place of $\VLInn$ and $\VConj \cap \sConj$ in place of $\VConj$.  
In the second paragraph, from Lemma \ref{lem:elconj} \ref{it:elconj4}
it now follows that $\ad(x)\subseteq \ad(y)=\ad(v_i)$; so $x$ dominates
$v_i$, for $i=1,\ldots , r$. Therefore $\a\in \CLInn^{\pm 1} 
\subseteq \sConj$ and, by induction on $|\phi|$ again, $\phi\in 
\la \CLInn\ra=\CConj$, as claimed.
\end{proof}

\begin{pap2}
Our next aim is to describe generators for $\NConj$ 
and to do so we need some definitions. 
\begin{defn}\label{defn:conjclos}
Let $x\in X$ and let $Y\subseteq X$. Let $C_1,\ldots ,C_r$ be a complete
list  of 
components of $\G_{x^\perp}$ 
which have non-empty intersection with $Y$. Let $v_1,\ldots ,v_s$ be
a complete list of elements of $X$ such that $\ad(v_i)$ has non-empty
intersection with $(Y\bs x^\perp)$. Define 
\begin{align*}
\cJ_x(Y)&=Y\cup \cup_{i=1}^r C_i 
%\{u\in X\,|\, u\in C \textrm{ for some connected component 
%} C \textrm{ of } \G_{x^\perp} \textrm{ such that } C\cap Y\neq \nul  \} 
\textrm{ and } \\
\cK_x(Y)&=Y\cup \cup_{i=1}^s \ad(v_i).
%\{u\in X\,|\,u\in \ad(v) \textrm{ for some } v 
%\textrm{ such that } \ad(v)\cap (Y\bs x^\perp)\neq \nul\}.
\end{align*}
Define a sequence of subsets $H_0, \ldots ,H_n,\ldots $ of $X$, dependent
on $x$ and $Y$, by 
\begin{align*}
H_0&=Y\textrm{ and }\\
H_{i+1}&=\cK_x(\cJ_x(H_i)),\textrm{ for } i\ge 0.
\end{align*}
Define
\[\cH_x(Y)=H_n, \textrm{ where } n \textrm{ is minimal such that } 
H_n=H_{n+1}.\]
\end{defn}
Note that, for all $x\in X$ and $Y\subseteq X$
\begin{itemize}
\item as $X$ is finite $\cH_x(Y)$ is defined; 
\item $\cH_x(Y)=Z\cup Z^\prime$, 
where $Z$ is a 
union of connected components of 
$\G_{x^\perp}$ and $Z^\prime \subseteq x^\perp$; 
\item
 from Lemma \ref{lem:fixad} it follows that if $Y\bs x^\perp$ is
a union of connected components of $\G_{x^\perp}$ then $\cK_x(Y)$ is
contained in $Y\cup \bigcup_{v\in \Isol(x)} \ad(v)$; and 
\item for all $v\in X$ either $\ad(v)\cap \cH_x(Y)=\nul$ or 
$\ad(v)\subseteq \cH_x(Y)\cup x^\perp$. 
\end{itemize}
\begin{lemma}\label{lem:hel}
Let $x$ and $y$ be elements of $X$.
\be[(i)]
\item\label{it:hel1} If $y\in x^\perp$ then $\cH_x(y)=\{y\}$. 
\item\label{it:hel2} If $y\notin x^\perp$ and $y\le_\cK z$, where $z\neq x$ is a 
$\cK$-maximal element of $X$, then $\cH_x(y)=\cH_x(z)$.
\ee
\end{lemma}
\begin{proof}
If $y\in x^\perp$ then $\cJ_x(y)=\cK_x(y)=\{y\}$, so also $\cH_x(y)=\{y\}$.

If $y\notin x^\perp$ and $y\le_\cK z$, where $z$ is $\cK$-maximal, then
$\ad(z)\cap (\{y\}\bs x^\perp)\neq \nul$, so 
$\ad(z)\subseteq \cK_x(\cJ_x(y))$. Hence $\cH_x(z)\subseteq \cH_x(y)$.
Moreover, for $z\neq x$, if $z\in x^\perp$ then 
$x\in z^\perp\bs z\subseteq y^\perp$
(Lemma \ref{lem:ad0} \ref{it:ad12}) 
so $y\in x^\perp$, a contradiction. Thus $z\notin x^\perp$ and
so $y\in \ad(z)$ and $\ad(z)\cap (\{z\}\bs x^\perp)\neq \nul$ which implies that 
$y\in \cK_x(\cJ_x(z))$ so $\cH_x(y)\subseteq \cH_x(z)$. 
\end{proof}
\begin{defn}\label{defn:NLInn}
Let $x\in X$ and $y\in X\bs x^\perp$. Let 
$\cH_x(y)\bs x^\perp=C_1\cup \cdots \cup C_t$, where $C_i$ is a connected component
of $\G_{x^\perp}$. The conjugating automorphism
\[
\atl_{y,x}=\prod_{i=1}^t \a_{C_i,x},
\]
is called a {\em basic normal   
conjugating automorphism}.  The set of all basic  normal conjugating 
automorphisms, as $x$ runs over $X$ and $y$ over $X\bs x^\perp$, 
is denoted $\NLInn=\NLInn(G)$. The set 
$\LInn(G)\backslash \NLInn(G)$ is denoted $\ULInn(G)$ and its elements
are called {\em unstable elementary conjugating automorphisms}.
\end{defn}

To describe the intersection of $\NConj$ and $\sConj$ some further
refinements of these definitions are necessary. 
\begin{defn}\label{defn:sol}
Let $x\in X$, and define  
\[
\Sol_0(x)=\{u\in X\,|\, \cH_x(u)\subseteq \Isol(x)\cup x^\perp\}\]
and 
\[
\Sol(x)=\{u\in \Sol_0(x)\,|\, u\textrm{ is } \textrm{ $\cK$-maximal and } 
u\notin x^\perp \}.
\] 
\end{defn}
\begin{defn}\label{defn:ILInn}
Let $\ILInn=\ILInn(G)$ denote the subset 
\[\{\atl_{y,x}\in \NLInn(G)\,|\, y\in\ \Sol(x)\}\]
of $\NLInn(G)$ 
and define $\iConj(G)=\la \ILInn(G)\ra$.
\end{defn}

\begin{lemma}\label{lem:normaldiv}
Let $\phi$ be a non-trivial element of  $\NConj(G)$ and for all $x\in X$ let $g_x$ be such that
$x\phi=g_x^{-1}\circ x \circ g_x$. Let $x,y\in X$ such that 
$x^{\e}g_x$ is a right divisor of $g_y$, for $\e\in \{\pm 1\}$. Then $x^{\e}g_x$ is a right divisor
of $g_z$, for all $z\in \cH_x(y)\bs x^\perp$.
\end{lemma}
\begin{proof}
It suffices to show that if $Y\subseteq X$  and $x^{\e}g_x$ is a right divisor
of $g_y$, for all $y\in Y\bs x^\perp$, then $x^{\e}g_x$ is a right divisor of 
$g_z$ given 
that either (i) $z\in \cJ_x(Y)\bs x^\perp$ or 
(ii) $Y\bs x^\perp$ is a union of connected components of $\G_{x^\perp}$ and 
$z\in \cK_x(Y)\bs x^\perp$. (i) follows from Lemma \ref{lem:laurlem}
\ref{it:laur28}. For (ii) suppose that $u\in X$, that 
$\ad(u)\cap (Y\bs x^\perp)\neq \nul$ with $w\in \ad(u)\cap Y$, 
$w\notin x^\perp$,  and that $z\in \ad(u)$, $z\notin Y\cup x^\perp$.
As $\phi\in \NConj$ there exists $f_u\in G$ such that $v\phi=v^{f_u}$, for
all $v\in \ad(u)$. %, and without loss of generality we may assume that
%$f_u$ has no left divisor in $u^\perp\bs u$.
 Then $w\phi=w^{f_u}=w^{g_w}$, 
where $g_w=g_w^\prime\circ x^{\e}\circ g_x$, for some $g_w^\prime \in G$. 
Therefore $g_wf_u^{-1}\in C_G(w)$; so $f_u=cg_w$, with $c\in C_G(w)$.
As $g_w$ has no left divisor in $C_G(w)$ we have $f_u=c\circ g_w=
c\circ g_w^\prime \circ x^{\e}\circ g_x$; 
so $x^{\e}g_x$ is a right divisor of $f_u$.
Now $z\phi=z^{f_u}=z^{g_z}$ so again $f_u=d\circ g_z$, where $d\in C_G(z)$.
Therefore $g_z=d^{-1}(cg_w^\prime x^{\e} g_x)$ and as $z\notin x^\perp$ we
have $x\notin \al(d)$, so that $x^{\e}$ does not cancel in reduction of 
the right hand side of this equality to a geodesic word. As $g_x$ has
no left divisor in $C_G(x)$ it follows that $x^{\e}g_x$ is a right divisor of
$g_z$, as required. 
\end{proof}
\begin{prop}\label{prop:nconjgens}
$\NConj(G)$ is generated by $\NLInn(G)$ and  
 $\NConj(G)\cap \sConj(G)=\iConj(G)$.
\end{prop}
\begin{proof}
To see that $\NLInn\subseteq \NConj$ suppose that $u,x,y\in X$  such that
$\ad(u) \cap \cH_x(y)\neq \nul$. Then $\ad(u)\subseteq \cH_x(y)\cup x^\perp$, 
by
definition, so $v\atl_{y,x}=v^x$, for all $v\in \ad(u)$. Thus $\atl_{y,x}
\in \NConj$, for all $x,y\in X$ with $y\notin x^\perp$. As 
 $\atl_{y,x^{-1}}= \atl_{y,x}^{-1}$ this implies that
 $\la \NLInn\ra\le \NConj$. 

Now 
assume that  $\phi\in \NConj$. 
If $|\phi|=0$ then $\phi=1$ and there is nothing to prove, so  
assume that $|\phi|>1$ and that for all 
conjugating automorphisms $\psi$ of shorter length $\psi\in \NConj$ implies
$\psi \in \la \NLInn \ra$.  
From Lemma \ref{lem:laurlem} \ref{it:laur25}  there exist $x,y\in X$, 
$\e\in\{\pm 1\}$  
such that $x\phi=g_x^{-1}\circ x\circ g_x$, $y\phi=g_y^{-1}\circ y\circ 
g_y$ and $x^{\e}g_x$ is a right divisor of $g_y$. Then 
 $y\notin x^\perp$ and, from Lemma \ref{lem:normaldiv}, $x^{\e}g_x$ is 
a right divisor of $g_z$, for all $z\in \cH_x(y)\bs x^\perp$. Therefore
$|\atl_{y,x^{-\e}}\phi|<|\phi|$ and so by induction  $\atl_{y,x^{-\e}}\phi
\in \la \NLInn\ra$, and the first statement of the proposition follows.

Now suppose that $\phi\in \NConj\cap \sConj$. Then $\al(g_z)\subseteq
\ad(z)$, for all $z\in X$ and, for all $z\in \cH_x(y)\bs x^\perp$, 
as $x^{\e}g_x$ is a right divisor of $g_z$, we have $x\in \ad(z)$. Therefore
$x$ dominates $z$, for all  $z\in \cH_x(y)\bs x^\perp$: that is 
$\cH_x(y)\subseteq \Isol(x)\cup x^\perp$. 
If $y\le_\cK x$ and $x$ is $\cK$-maximal then we claim that also $y$ is
$\cK$-maximal. To see this note that, since $[x,y]\neq 1$, if $\ad(y)\subseteq
\ad(x)$ then $x^\perp\bs x\subseteq y^\perp\bs y$. As $x$ dominates
$y$ we have $y^\perp\bs y\subseteq x^\perp\bs x$ in any case, so 
$y\le_\cK x$ implies $y^\perp\bs y=x^\perp\bs x$; so if $x$ is $\cK$-maximal
then so is $y$. In this case $\atl_{y,x}\in \ILInn$. Otherwise let 
$z\neq x$ be such that $y\le_\cK z$ and $z$ is $\cK$-maximal. Then 
%As $[x,y]\neq 1$ there exists $\cK$-maximal $z$ such that $y\le_\cK z$ and 
$\cH_x(y)=\cH_x(z)$ (Lemma \ref{lem:hel} \ref{it:hel2}) and so $\atl_{y,x}=\atl_{z,x}\in \ILInn$. Thus in all cases $\atl_{y,x}\in \ILInn$. Moreover, 
by definition $\ILInn\subseteq \sConj$, 
so $\atl_{y,x}=\atl_{y,x}\in \NConj\cap \sConj$.
Hence $\atl_{y,x^{-\e}}\phi\in \NConj\cap \sConj$ and its length is less than
the length of $\phi$, so 
the second statement follows, again by induction on the length of $\phi$. 
\end{proof}
\end{pap2}
%\subsection{Structure of $\aConj(G)$}
To describe the structure of  $\aConj(G)$ it is convenient to 
work with outer automorphisms. Denote the group 
$\Aut(G)/\Inn(G)$ of outer automorphisms by $\Out(G)$ as usual and 
given a subgroup $B$ of $\Aut(G)$ let $\overline{B}$ denote the group
$B\Inn(G)/\Inn(G)$. We write $\bar\b$ for the image of $\b\in \Aut(G)$ in $\Out(G)$ and 
$\g_x$ for the inner automorphism of $G$ mapping $g$ to $g^x$, for 
all $g\in G$. 
\begin{prop}
Let $G=G(\G)$, where $\G$ is a  connected graph.
Then $\aConj(G)$ is torsion-free and  $\aOonj(G)$ is free Abelian  and
a normal subgroup of  $\oOut(G)$. Moreover, if $c(x)$ is the number
of connected components of $\G_{x}$, for all $x\in X$, then 
the $\aOonj(G)$ has rank $\sum_{x\in X} (c(x)-1)$. 
\end{prop}
\begin{proof}
First we show that $\aOonj(G)$ is  a free Abelian group. % in $\oOut(G)$. 
Let 
$x\in X$ and suppose that $\G_x$ has connected components
$C_1,\ldots ,C_r$. If $y\in X$, $y\neq x$, then there is some 
$i$ such that $y^\perp \subseteq C_i\cup \{x\}$. Assume that
$\G_y$ has components $D_1,\ldots ,D_s$. We claim that there
is a $j$ such that 
\be[(i)]
\item\label{it:aconji} $D_j \supseteq C_k\cup \{x\}$, for all $k\neq i$, and 
\item\label{it:aconjii} $C_i \supseteq D_k\cup \{y\}$, for all $k\neq j$.
\ee
To see this choose 
$j$ such that $x\in D_j$, so $x^\perp \subseteq D_j\cup \{y\}$.
Let $u\in C_k$, $k\neq i$. Then there exists a path in $\G$ from
$u$ to $x$ and, as $y\in C_i$, this path may be chosen so that 
none of its vertices is $y$. Hence
$u$ and $x$ belong to the same component of $\G_y$. Thus, if $D_j$ is the
component of $\G_y$ containing $x$ then $D_j\supseteq C_k\cup \{x\}$, 
for all $k\neq i$. This shows that the first statement of the claim 
holds and the second follows by symmetry. 

The subgroup $\aOonj(G)$ is generated by the images $\bar\b_{C,x}$ 
in $\Out(G)$ of 
aggregate conjugating automorphisms
$\b_{C,x}$, where $x$ ranges over $X$ and $C$ ranges over
the  connected components of $\G_x$. 
Let $\bar\b_{C,x}$ and $\bar\b_{D,y}$ be generators of $\aOonj(G)$.
If $x=y$ then these two generators commute, so we assume $x\neq y$ 
and that components  $C_i$ and $D_j$ of $\G_x$ and $\G_y$, respectively,
have been chosen as in the claim above; so $x\in D_j$ and $y\in C_i$. 
If $C=C_i$ then let
$\b_1=\g_{x^{-1}}  \b_{C,x}$,  so $\bar\b_{C,x}=\bar\b_1$ and
$\b_1=\prod_{C_k\neq C}\b_{C_k,x}^{-1}\in \aConj(G)$. Otherwise
let $\b_1= \b_{C,x}$. In either case $u\b_1=u$, for all $u\in C_i$.
Similarly we may choose a representative $\b_2$ of $\bar\b_{D,y}$ such that 
$u\b_2=u$, for all $u\in D_j$. Then, from \ref{it:aconji} and
\ref{it:aconjii} above, it follows
that $\b_1\b_2=\b_2\b_1$ and 
so $\bar\b_{C,x}\bar\b_{D,y} =\bar\b_{D,y} \bar\b_{C,x}$, 
and  $\aOonj(G)$ is Abelian as claimed. 
%as required. 
%Since it follows directly from the definition that $\Conj(G)$ is torsion
%free this completes the proof. 

Now, for $i=1,\ldots ,r$, let $\bar\b_i=\bar\b_{C_i,x}\in \aOonj(G)$. As 
$\prod_{i=1}^{r}\bar\b_i=1$, given any element $\phi \in \aOonj(G)$ we may
write $\phi=\bar\g_0\bar\g_1$, where $\g_0=\prod_{i=1}^{r-1}\b_i^{m_i}$, for
some $m_i\in \ZZ$, and $\g_1$ is a product of generators $\b_{D,y}$, with
$y\neq x$. Let $y\in C_i$, where $1\le i\le r-1$. Then $y\g_0=y^{x^{m_i}}$ and
$y\g_1=y^h$, for some $h\in G$ such that $x\notin \al(h)$. Also $x\g_1=x^g$,
for some $g\in G$ such that $x\notin \al(g)$. Then
\[
y\g_0\g_1=(y^h)^{(x^{m_i})^g}.
\]
If $w$ is a geodesic word representing $h(x^{m_i})^g$ then the exponent
sum $|w|_{x}$ of $x$ in $w$ equals $m_i$; so $y\g_0\g_1=v^{-1}\circ y\circ v$,
where $|v|_{x}=m_i$. For $z\in C_r$ we have $z\g_0\g_1=z\g_1=u^{-1}\circ z
\circ u$, for some $u\in G$ such that $|u|_{x}=0$. If $\phi=1$ then 
$\g_0\g_1\in \Inn(G)$ and so it must be that $m_1=\cdots =m_{r-1}=0$. It 
follows by induction, on the minimal number of generators appearing
in a word representing $\phi$, that $\aOonj(G)$ is free Abelian
of rank $\sum_{x\in X} (c(x)-1)$, as claimed. 

To see that $\aConj(G)$ is torsion free it suffices  to note that 
$\Inn(G)$ is torsion free; since $\Inn(G)\cong G/Z(G)$, which is 
a partially commutative group. In fact $G/Z(G)\cong G(\G_Z)$, where 
$Z$ is the subset of $X$ consisting of vertices connected to all vertices
of $\G$ and $\G_Z$ is the full subgraph of $\G$ on $X\backslash Z$. 
As both $\Inn(G)$ and $\aOonj(G)=\aConj(G)/\Inn(G)$ are torsion-free,
 so is $\aConj(G)$.  

To show that $\aOonj(G)$ is normal in $\oOut(G)$ we shall show that 
if $\bar\b_{C,x}$ is an arbitrary generator of $\aOonj(G)$  
and  
$\phi$  is a generator of $\oOut(G)$, which
is not in  $\aOonj(G)$,  then $\phi^{-1}\bar\b_{C,x}\phi
\in \aOonj(G)$. We consider the cases where $\phi$ is 
the image in $\oOut(G)$ of 
an inversion and  
a transposition %and a conjugating automorphism, 
separately.

Let $\phi=\bar\i_z$; an inversion. Straightforward
checking shows that 
\be[(a)]
\item if $x=z$ then $\i_z^{-1}\b_{C,x}\i_z=\b_{C,x}^{-1}$ and
\item if $x\neq z$ then $\i_z^{-1}\b_{C,x}\i_z=\b_{C,x}$.
\ee
Hence the result holds in this case.

Next suppose that  $\phi=\bar\tr_{v,z}$, where $v=y$ or $y^{-1}$. If $y=x$ then we have
$x^\perp\bs x\subseteq z^\perp$ which implies that $\G_x$ is connected
and so $\b_{C,x}$ is an inner automorphism; as are all its conjugates. 
Thus we assume that 
$y\neq x^{\pm 1}$ and $\G_x$ is not connected. 
Let $C_1$ be the component of $\G_x$ containing $y$. Then
$y^\perp\bs y\subseteq z^\perp$ and $y^\perp\subseteq C_1\cup \{x\}$.
If $z\in C_2$, for some component $C_2$ of $\G_x$ with $C_2\neq C_1$ then
$z^\perp\subseteq C_2\cup \{x\}$ so $y^\perp\bs y\subseteq 
(C_1 \cup \{x\})\cap (C_2\cup \{x\})=\{x\}$; in which case 
$y^\perp=\{x,y\}$ and $x\in z^\perp$. These conditions imply
that $\tr_{v,z}^{-1}\b_{C,x}\tr_{v,z}=\b_{C,x}$, so we may 
now assume that $z\in C_1$. Assume in addition that the connected
components of $\G_x$ are $C_1,\ldots ,C_r$. If $C=C_i$, where 
$i\neq 1$  then  $\phi^{-1}\b_{C,x}\phi =\b_{C,x}$. If $C=C_1$ then
set $\b_1=\prod_{i=2}^r\b_{C_i,x}^{-1}$, so $\bar\b_{C,x}=\bar\b_1$ and
$\phi^{-1}\b_1\phi=\b_1$. Thus $\overline{\phi^{-1}\b_{C,x}\phi}=
\overline{\b_1}=\overline{\b_{C,x}}$, and the result follows.
\end{proof}

The previous proposition can not be extended to disconnected graphs or to 
$\Conj(G)/\Inn(G)$, in place of $\aConj(G)/\Inn(G)$, as the following
examples show. 
\begin{expl}
In the graph $\G$  of Figure \ref{fig:aconj}, let $C$ be the component
of $\G_x$ containing $a$ and let $D$ be the component of $\G_y$ containing $a$. 
 Then $\aConj$ contains $\b_{C,x}$ and $\b_{D,y}$ and $a[\b_{C,x},\b_{D,y}]=a^{[x,y]}$, while
$b[\b_{C,x},\b_{D,y}]=b$. Therefore $[\b_{C,x},\b_{D,y}]\notin \Inn$, so $\aOonj$ is non-Abelian.

\end{expl}
\begin{figure}
\begin{center}
\psfrag{a}{$a$}
\psfrag{b}{$b $}
\psfrag{c}{$c $}
\psfrag{d}{$d $}
\psfrag{e}{$e $}
\psfrag{f}{$f $}
\psfrag{g}{$g $}
\psfrag{h}{$h $}
\psfrag{x}{$x $}
\psfrag{y}{$y $}
\includegraphics[scale=0.2]{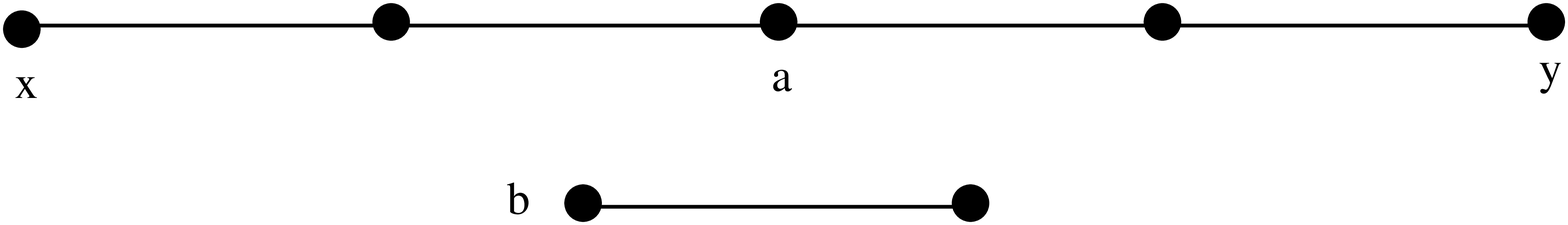}
\end{center}
\caption{$\aOonj(G)$ is non-Abelian}\label{fig:aconj}
\end{figure}

\begin{expl}
Let $\G$ be the graph of Figure \ref{fig:conj}. Then $\G_{x^\perp}$ has
a component $C=\{c,d,e,y\}$ and $\G_{y^\perp}$ has a component 
$D=\{a,b,c,x\}$. Let $\a=\a_{C,x}$ and $\b=\a_{D,y}$. The images of 
$c$ and $g$ under $[\a,\b]$ are $c^{[x,y]}$ and $g$, respectively.
Therefore $\a\b\neq \b\a$ modulo $\Inn(G)$. In this example 
$\Inn(G)=\aConj(G)$, so in general $\Conj(G)/\aConj(G)$ is also non-Abelian.
\end{expl}
\begin{figure}
\begin{center}
\psfrag{a}{$a$}
\psfrag{b}{$b $}
\psfrag{c}{$c $}
\psfrag{d}{$d $}
\psfrag{e}{$e $}
\psfrag{f}{$f $}
\psfrag{g}{$g $}
\psfrag{h}{$h $}
\psfrag{x}{$x $}
\psfrag{y}{$y $}
\includegraphics[scale=0.3]{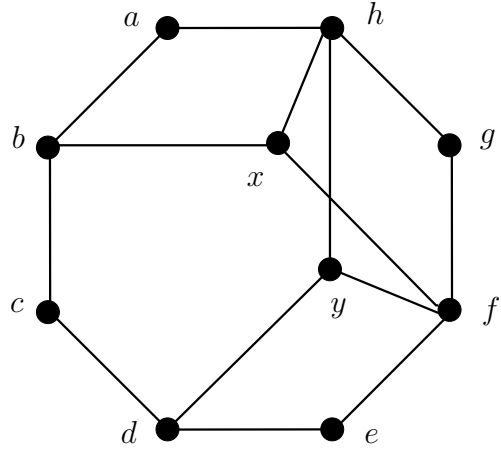}
\end{center}
\caption{$\Conj(G)/\Inn(G)$ is non-Abelian}\label{fig:conj}
\end{figure}

%%% Local Variables: 
%%% mode: latex
%%% TeX-master: "aut2.tex"
%%% End: 
\section{The stabilisers $\St(\cK)$ and $\cSt(\cK)$.}\label{section:st}
\begin{defn}\label{defn:st}
Define 
\[
\St(\cK)=\{\phi\in \Aut(G)|G(Y)\phi=G(Y), \textrm{ for all } 
Y\in \cK\}.
\]
\end{defn}
Then, from Lemma \ref{lem:ad0} \ref{it:ad9} it follows that $\phi\in \St(\cK)$ if
and only if $G(\ad(x))\phi=G(\ad(x))$, for all $x\in X$. 
Also, from Remark \ref{rem:aut*}  the subgroup
of $\oAut(G)$ generated by $\Inv$ and $\Tr$ is contained in $\St(\cK)$.

\begin{defn}\label{defn:stconj}
 Define 
\[
\cSt(\cK)=\{\phi\in \Aut(G)|G(Y)^\phi=G(Y)^{f_Y}, 
\textrm{ for some } f_Y\in G, \textrm{ for all } 
Y\in \cK\}.
\]
\end{defn}
We shall make use of the following fact in the proof of the
next proposition.
\begin{lemma}\label{lem:paraint}
Let $H$ and 
$K$ be canonical parabolic subgroups of $G$ and let $\theta\in \Aut(G)$ be such that
$H\theta=H^f$ and $K\theta=K^g$, for some $f, g\in G$. 
%which  has no left divisor in $K$.  
Then there exists $h\in H$ such that 
 $(H\cap K)\theta=(H\cap K)^h$. 
\end{lemma}
\begin{proof}
First suppose that $f=1$ and that 
 $g$ has no left divisor in $K$. In this case it follows from 
\cite[Corollary 2.4]{DKR2}, that if $u\in K$  then there exist words
$a,b$ (dependent on $u$) such that 
$g=a\circ b$ and $u^g=b^{-1}\circ u\circ b$. Thus, if $w\in H\cap K^g$, then
for some $u\in K$, we have $w=u^g=b^{-1}\circ u\circ b$. This implies
that $u\in H\cap K$, so $w=u^g\in (H\cap K)^g$. That is, $H\cap K^g\subseteq
(H\cap K)^g$. 

Now $(H\cap K)\theta\subseteq H\cap K^g\subseteq (H\cap K)^g$. Moreover, from
the hypothesis on $H$, $K$ and $\theta$, we have $H\theta^{-1}=H$ and 
$K\theta^{-1}=K^h$, where $h=(g\theta^{-1})^{-1}$. Applying the previous
argument gives $(H\cap K)\theta^{-1}\subseteq (H\cap K)^h$, so 
$H\cap K\subseteq (H\cap K)\theta^{h\theta}$, from which we obtain 
$(H\cap K)^g\subseteq (H\cap K)\theta$.   

In the general case let $\g_{f^{-1}}$ denote conjugation by $f^{-1}$ and 
let $\phi =\theta \g_{f^{-1}}$. Then $H\phi=H$ and $K\phi=K^{gf^{-1}}$. 
Let $gf^{-1}$ have minimal form $gf^{-1}=a\circ b$, where $a\in K$ and 
$b$ has no left divisor in $K$. Then $K\phi=K^b$, so from the first case
above $(H\cap K)\phi =(H\cap K)^b$. Hence $(H\cap K)\theta=(H\cap K)^{bf}$, 
as required. 
\begin{comment}
As $H\cap K$ is a canonical parabolic subgroup it suffices to show that, for all 
 $x\in X\cap H\cap K$ there is some $s\in H\cap K$ such that  
$x\theta=s^g$. Given such an $x$ there are elements $z\in H$ and 
$w\in K$ such that $x\theta =z=w^g$. 
 From \cite[Corollary 2.4]{DKR2} there
are elements $a,b,c,d,e, h, u,v\in G$ such that
$g=a\circ b\circ c\circ h$, $w=u^{-1}\circ v\circ u$, $u=d\circ a^{-1}$, 
$dh=d\circ h$ and $w^g=(dh)^{-1}\circ e\circ (dh)$, with $\al(e)=\al(v)$, 
$v$ cyclically minimal, $e=v^b$, $\al(b)\subseteq \al(v)$ and 
$[\al(b\circ c),\al(d)]=[\al(c),\al(v)]=1$. As $g$ has no   left divisor
in $K$ we have $a=b=1$, so $e=v$, $g=ch$ and $x\theta =(d^{-1}vd)^{g}$. 
Thus $(H\cap K)\theta =(H\cap K)^g$, as required.
\end{comment}
\end{proof}
\begin{theorem}\label{theorem:Stconj}
$\cSt(\cK)=\oAut(G)$ and 
the group $\Aut(G)$
can be decomposed into the internal semi-direct product of 
$\cSt(\cK)$ and the finite subgroup $\Aut^\G_\cmp(G)$:  
\[\Aut(G)=\cSt(\cK)\rtimes \Aut^\G_\cmp(G).\]
\end{theorem}
\begin{vnr}
\begin{proof}
That $\Aut(G)=\cSt(\cK)\rtimes \Aut^\G_\cmp(G)$ follows immediately 
from Proposition \ref{prop:autdecomp} once the first statement has been proved. 
From Lemma \ref{lem:paraint} and the definition of admissible sets,   
   $\phi\in \cSt(\cK)$ if
and only if $G(\ad(x))\phi=G(\ad(x))^{f_x}$, for some $f_x\in G$, for all $x\in X$. 
It follows then from Proposition \ref{prop:aut*} that 
$\oAut(G)\subseteq \cSt(\cK)$. 

For the reverse inclusion note that from Proposition \ref{prop:autdecomp} every
element $\phi$ of $\Aut$ can be expressed as $\phi=\a\b$, with $\a\in \oAut$ and
$\b\in \Aut^\G_\cmp(G)$. If $\phi\in \cSt(\cK)$ then, since $\oAut(G)\subseteq \cSt(\cK)$
we have $\a^{-1}\phi \in \cSt(\cK)\cap \Aut^\G_\cmp(G)$. However, from the
definitions of $\cSt(\cK)$ and $\Aut^\G_\cmp(G)$ this means that $\a^{-1}\phi=\beta
=1$, so 
$\phi=\a\in \oAut(G)$, as required. 
\end{proof}
\end{vnr}

\begin{theorem}\label{theorem:nconjnorm}
$\NConj(G)$ is a normal subgroup of $\cSt(\cK)$.
\end{theorem}
\begin{proof}
In the light of Lemma \ref{lem:elconj} we may assume that $\G$ has
no isolated vertex. 
Let $\phi\in \NConj(G)$ and $\psi\in \cSt(\cK)$. Let $x\in X$ and let
$g_x$ and $h_x$ be elements of $G$ such that $u\phi=u^{g_x}$, for
all $u\in \ad(x)$, and $G(\ad(x))\psi^{-1}=G(\ad(x))^{h_x}$. 
For each $u\in \ad(x)$
let $w_u$ be the minimal form of an element of $G$  such that 
$w_u^{h_x}=u\psi^{-1}$, so $\al(w_u)\subseteq \ad(x)$
 and $w_u\psi=u^{h_x^{-1}\psi}$. Then 
$u\psi^{-1}\phi\psi=u^{f_x}$, where $f_x=(h_x^{-1}g_x(h_x\phi))\psi$, so 
$f_x$ is dependent only on $x$ and $\psi^{-1}\phi\psi\in \NConj(G)$. 
\end{proof}

\begin{lemma}\label{lem:stintconj}
\be[(i)]
\item
$\St(\cK)\cap \Conj(G)=\sConj(G)$. 
\item $\St(\cK)\cap \VConj(G)=\CConj(G)$.
%\begin{pap2} and $\St(\cK)\cap \NConj(G)=\iConj(G)$\end{pap2}.
\item
If $\phi\in \Aut(G)$ then $\phi\in \sConj(G)$ if and only if 
$x\phi=x^{f_x}$ where $\al(f_x)\subseteq \ad(x)$, for all $x\in X$. 
\ee
\end{lemma}
\begin{proof}~
If $\a\in \SLInn$ then $G(\ad(x))\a=G(\ad(x))$, for 
all $x\in X$, so $\sConj\subseteq \St(\cK)\cap \Conj(G)$. For the converse 
 use induction on $|\phi|$, where 
 $\phi\in \St(\cK)\cap \Conj(G)$. If $|\phi|=0$ then $\phi=1$ and 
so belongs to $\sConj$. Assume then that $|\phi|>0$. 
In this case, from Lemma \ref{lem:laurlem} \ref{it:laur25},  there exist $u_1,u_2\in X$ such that
$u_i\phi=u_i^{w_i}$, reduced as written, for some $w_1,w_2\in G$, and
$u_1w_1$ is a right divisor of $w_2$. It follows, as in the proof 
of Proposition \ref{prop:vconjgens}, that 
$u_1\notin u_2^\perp$. 
 As $\phi\in \St(\cK)$
we have $w_2\in G(\ad(u_2))$ so $u_1w_1\in G(\ad(u_2))$. In particular
$u_1\in \ad(u_2)$ so $u_2^\perp\bs u_2\subseteq u_1^\perp$. Therefore
$\tr_{u_2,u_1}\in \Tr_\lk$ and $\b=\tr_{u_2,u_1}\tr_{u_2^{-1},u_1}\in \sConj
\subseteq \St(\cK)\cap \Conj(G)$. Therefore $\b\phi\in \St(\cK)\cap \Conj(G)$ and
 $|\b\phi|<|\phi|$  so, by induction, $\b\phi\in \sConj(G)$. 
This gives $\phi\in \sConj(G)$, 
as required. From this and Lemma \ref{lem:elconj} \ref{it:elconj4}
 the last statement of the 
lemma follows immediately.

That  $\St(\cK)\cap \VConj(G)=\CConj(G)$ 
%\begin{pap2}
% and $\St(\cK)\cap \NConj(G)=\iConj(G)$
%\end{pap2} 
follows immediately from Proposition \ref{prop:vconjgens}
%\begin{pap2} and 
%\ref{prop:nconjgens}\end{pap2}. 
\begin{comment}
Now $\St(\cK)\cap \NConj(G)\subseteq \sConj(G)$ so equals 
$\NConj(G)\cap \sConj(G)$. If $\phi=\a_{C,x}$, where $C=\{y\}$ and 
$\phi\in \NConj(G)$ then, for all $u\in \ad(y)$, $u\phi=u^g$, for some 
$g\in G$. As $y\phi=y^x$ and $x^\phi =x$ and $[x,y]\neq 1$ it follows that
$g=vx$, for some $v\in C(x,y)$. Then $v\in C(u)$, for all $u\in \ad(y)$, and 
hence $v$ may be taken to be $1$ and  $g=x$. 
Therefore every other element of $\ad(y)$ must commute with
$x$ and so $\phi\in \iConj(G)$.
\end{comment}   
\end{proof}

The following question now arises naturally.  
\begin{que}\label{que:aut*decomp} Let $\G$ be a connected graph. Is 
$\cSt(\cK)=\St(\cK)\NConj(G)$? 
\end{que}
It seems on first sight very plausible that the answer is ``yes'', but, as the 
subsequent example shows, it turns out to be ``no'' and in fact, 
in general $\St(\cK)\Conj(G)\subsetneq \cSt(\cK)$.   
\begin{expl}\label{ex:diagex}
Take $G$ to be the group $G(\G)$ where $\G$ is the graph of 
Figure \ref{fig:diagex}. 
\begin{figure}
\begin{center}
\psfrag{a}{$a$}
\psfrag{b}{$b$}
\psfrag{c}{$c$}
\psfrag{r}{$r$}
\psfrag{s}{$s$}
\psfrag{t}{$t$}
\psfrag{v}{$v$}
\includegraphics[scale=0.4]{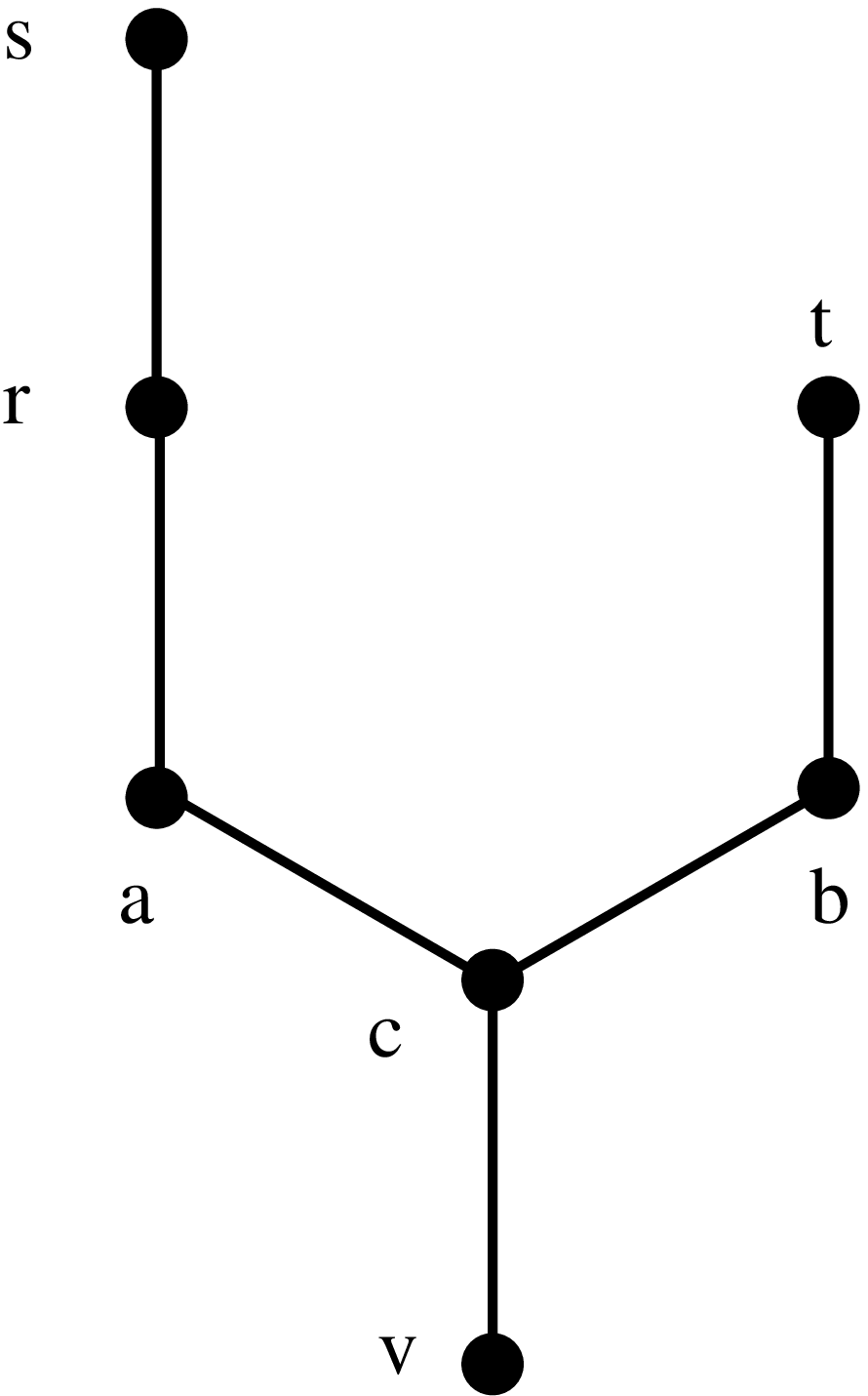}
\end{center}
\caption{Example \ref{ex:diagex}}\label{fig:diagex}
\end{figure}
Denote the components of $\G_{v^\perp}$ %by $C$ the component  $C=\{a,r,s\}$of $\G_{v^\perp}$ 
by  $C=\{a,r,s\}$ and $D=\{b,t\}$, 
 let $\a=\a_{C,v}$, $\tr=\tr_{v,a}\tr_{v,b}\tr_{v,a^{-1}}$, and set
$\phi=\a\tr$. 
\begin{equation*}
z\phi=
\left\{
\begin{array}{ll}
z, & \textrm{ if } z=b, c, t\\
vb^a, &\textrm{ if } z=v\\
z^{vb^a} & \textrm{ if } z\in C
\end{array}
\right.
\end{equation*}

The answer  to question \ref{que:aut*decomp} above is ``no'' in 
this example, as  $\phi$ cannot be written as $\g\d$, where $\d\in \Conj$ 
and $\g\in St(\cK)$, as we shall demonstrate. 
 Suppose then that $\phi=\g\d$, where  $\d\in \Conj$ 
and $\g\in St(\cK)$. 
The set $\cK$ consists of 
$\ad(v)=\{a,b,c,v\}$, $\ad(s)=\{a,r,s\}$, $\ad(t)=\{b,c,t\}$ and four more
sets $\ad(z)=\{z\}$, where $z=a, b, c$ and $r$. As $\g$ maps the subgroup generated by
$\ad(a)=\{a\}$ to itself  we have $a\g=a^{\pm 1}$. As $a\d=a^g$, for some
$g\in G$, it must be that $a\g=a$. Similarly $b\g=b$, $c\g=c$ and $r\g=r$.  
Combined with the expression for $\phi$ above  we obtain 
$a\d=a^{vb^a}$, 
$b\d=b$ and $c\d=c$.  As $c\d=c$, we have $C_G(c)\d=C_G(c\d)=C_G(c)$,  so
 $G(c^\perp)\d=G(c^\perp)$: that is $G(a,b,c,v)\d=G(a,b,c,v)$. Moreover,  
 as $\d$ acts on generators by conjugation,  
$\d$ must map $G(a,b,v)$ to itself; so $v\d=v^g$, for some $g\in G(a,b,v)$,
and $\d$ restricts to an automorphism of $G(a,b,v)$.
Applying Lemma \ref{lem:laurlem} to the restriction of $\d$ 
to $G(a,b,v)$ we see that
either $b^\e$ or $a^\e{vb^a}$ is a right divisor of $g$, or that $v g$ is a right
divisor of $vb^a$, in which case $g=b^a$. In the latter case consider
the automorphism $\a_{C,v^{-1}}\d$. This maps $a$ to $a^{b^a}$, $b$ to 
itself, $c$ to itself  and $v$ to $v^{b^a}$. Restricting to 
$G(a,b,v)$ again gives a contradiction to Lemma \ref{lem:laurlem}. 
Thus we may assume that either $a^\e vb^a$ or $b^\e$ is
a right divisor of $g$. Suppose that $m$ is maximal 
such that 
$a^{\e m}vb^a$ is a right divisor
of $g$; say $g=g_0\circ a^{\e m}vb^a$. If $m\ge 1$ then
 $\d_0=\a_{v,a^\e}^{-m}\d$ maps $a$ to $a^{vb^a}$, fixes $b$ and $c$ and 
maps $v$ to $v^{g_0vb^a}$. As this contradicts  Lemma \ref{lem:laurlem}
we have $m=0$. Similarly, if  
$b^{\e}$ is a right divisor of $g$ then we obtain a  contradiction. 
 Hence no such conjugating automorphism $\d$ exists. 

It is also possible to show that $\tr\a\notin \Conj(G) \St(\cK)$. Moreover the 
example shows that replacing $\cSt(\cK)$, $\St(\cK)$ and $\NConj(G)$ with
their canonical images in $\Out(G)$ the equality of Question 
\ref{que:aut*decomp} still
fails.
\end{expl}

If there are no  dominated vertices in $\G$, that is $\Isol(\G)=\nul$, then following holds. Here  
\[
\St(\cL)=\{\phi\in \Aut(G)|G(Y)\phi=G(Y), \textrm{ for all } 
Y\in \cL\},
\]
a subgroup of $\Aut(G)$ 
defined originally in \cite{DKR5}, where it was shown to be an arithmetic
group.

\begin{lemma}\label{lem:isolnul}
 Let $\G$ be a graph such that  $\Isol(\G)=\nul$. Then 
\be[(i)]
\item\label{it:solnul1} $\Conj(G)\cap \St(\cK)=\sConj(G)=\{1\}$ and 
$\Conj(G)=\VConj(G)=\NConj(G)$ is normal in $\cSt(\cK)$ and
\item\label{it:solnul2} $\St(\cL)=\St(\cK)$.
\ee
\end{lemma}
\begin{proof}~ 
\ref{it:solnul1}
In this case $\CConj(G)=\sConj(G)=\{1\}$ so $\VConj\cap \St(\cK)=1$. 
To see that $\NConj=\Conj$, note first that,  
from Lemma \ref{lem:ad0} \ref{it:ad2}, 
it follows that $\ad(x)=\cl(x)$, 
for all $x\in X$. Let $x,y\in X$ and let $C$ be a component of 
$\G_{y^\perp}$. If $\ad(x)\cap C\neq \nul$ then, from Lemma \ref{lem:fixad},
 either $\ad(x)\subseteq C\cup y^\perp$ or $y\in \ad(x)$. If $y\in \ad(x)
=\cl(x)$ then $\cl(x)\subseteq y^\perp$; so either $\ad(x)\cap C=\nul$ or
 $\ad(x)\subseteq C\cup y^\perp$. Therefore, either 
$u\a_{C,y}=u^y$, for all $u\in\ad(x)$, or $u\a_{C,y}=u$, for all $u\in \ad(x)$;
and it follows that  $\NConj=\Conj$. 

\ref{it:solnul2} From \cite[Lemma 2.4]{DKR3} it follows that 
 if $Y\in \cL$ then  $Y=\cup_{y\in Y}\cl(y)$. 
Therefore, for all $\phi\in \Aut(G)$, 
$\phi\in \St(\cL)$ if and only if 
 and $G(\cl(y))\phi=G(\cl(y))$, for all $y\in Y$. 
Given that $\ad(x)=\cl(x)$, for all $x\in X$, the result follows from the remark following the definition of $\St(\cK)$ above. 
%from Theorem \ref{theorem:aut*decomp}.
\end{proof}
\begin{theorem}\label{theorem:nolocisol} 
The following are equivalent. 
\be[(i)]
\item $\Isol(\G)=\nul$.
\item  
$\cSt(\cK)=\NConj(G)\rtimes \St(\cL).$
\item  
$\cSt(\cK)=\Conj(G)\rtimes \St(\cL).$
\item
 $\cSt(\cK)=\Conj(G)\rtimes \St(\cK)$. 
\ee
\end{theorem}
\begin{proof}
In view of Lemma \ref{lem:isolnul} it suffices to show that each of the 
last 
 three statements implies the first.
To see that the second or third statement implies the first, 
suppose  $\cSt(\cK)$ decomposes as the given internal semi-direct
product. If $y\in \Isol(x)$, for some $x,y\in X$, then
$\tr=\tr_{y,x}\in \cSt(\cK)$, so
 $\tr=\a\l$, for some $\a\in \Conj$ and
$\l\in \St(\cL)$. Then, for $z\in X\bs y$ we have
 $z=z\tr = z\a\l=z^g\l=z\l^{g\l}$, for some $g\in G$. 
As $z\l\in G(\cl(z))$ it follows that $z\l=z\circ w$, for some
$w\in G(\cl(z))$, so $(zw)^{g\l}=z$, 
from which, counting exponents of letters, we infer
that $w=1$. 
Hence $g\l\in G(z^\perp)$, so
$g\in G(z^\perp)$, which implies that $z\a=z$, and consequently $z\l=z$. 
Now $y\a=y^h$ and $z\a=z\l=z$, for some $h\in G$ and all $z\in X\bs y$. 
 As $\l\in \St(\cL)$ we have $y\l\in G(\cl(y))$ and, since $z\l=z$ for all  
$z\neq
y$, we have $y\l=y^\e w$, for some $w\in G(\cl(y))$, $\e=\pm 1$. 
However this means that 
$yx=y\tr=y\a\l=(y^{\e}w)^{h\l}$ and, as $x\notin y^\perp$, 
the exponent sum of $x$ on the left hand side of this expression is zero, 
while on the right it is one. Hence no such $x,y$ exist and $\Isol(\G)=\nul$.

To see that the fourth statement implies the first: from the fourth statement
it follows that $\Conj(G)\cap \St(\cK)=\{1\}$, so $\SLInn=\nul$ and this
implies that $\Isol(\G)=\nul$. 
%from Theorem \ref{theorem:aut*decomp}.
\end{proof}
\subsection{Balanced graphs}\label{section:bal}
 Although $\Isol(\G)=\nul$ is a necessary condition for 
the intersection of $\Conj(G)$ and $\St(\cK)$ to be trivial,  
the class of graphs 
for which $\cSt(\cK)=\Conj(G)\St(\cK)$ is much wider than those without
dominated vertices: it can, as we shall show, be characterised using
 the following definition.
\begin{defn}
A graph $\G$ is called {\em balanced} if %either $\Isol(\G) =\nul$ or
the following condition holds for all $v\in \Isol(\G)$. 
Either
\be
\item $\ado(v)=\nul$, or
\item there exists a connected component $C_v$ of $\G_{v^\perp}$ such that
$\ado(v)\subseteq C_v$. 
\ee
\end{defn}

In this section we shall use the following extensions  of the terminology
for   
 transvections and conjugating automorphisms. 

\begin{defn}\label{defn:comptr}
If $\tr_{x,y_i}$ is a transvection, for $x, y_i\in X^{\pm 1}$, 
and $w=y_1\cdots y_n$ 
is a geodesic word in $G$ then $\tilde{\tr}_{x,w}=\tr_{x,y_n}\cdots \tr_{x,y_1}$ is
called a {\em composite transvection} and the set of all composite  
transvections is denoted $\Trt=\Trt(G)$.
\end{defn}
\begin{defn}
\begin{itemize}
\item
If $L$ consists of a union $L=\cup_{i=1}^r$ of  connected components $C_i$ 
of $\G_{x^\perp}$ then
 $\a_{L,x^\e}=\prod_{i=1}^r \a_{C_i,x^\e}$ is called an {\em extended  
elementary conjugating automorphism}. 
 The set of all extended elementary conjugating automorphisms  is denoted 
$\WLInn=\WLInn(G)$.
\item
Let $y\in X^{\pm 1}$ and 
 $\a_{L,y}\in \WLInn(G)$. If  $\ad(y)\cap L=\nul$ then 
$\a_{L,y}$ is called a {\em tame} elementary
conjugating automorphism of $G$. The set of 
all tame elementary conjugating automorphisms is denoted $\TLInn(G)$. 
\end{itemize} 
\end{defn}

\begin{lemma}\label{lem:xycomps}
Let $x$ and $y$ be elements of $X$ such that $x$ dominates $y$ and let $C$ 
be a component of $\G_{y^\perp}$. 
\be[(i)]
\item\label{it:xycomps1} If $x\notin C$ then $C$ is a component of $\G_{x^\perp}$.
\item\label{it:xycomps2}  If $x\in C$ and the components of $\G_{x^\perp}$ which meet $C$ are 
$C_1,\ldots, C_r$ then $C=[(C_1\cup \cdots \cup C_r)\cup x^\perp]\bs y^\perp$. 
\ee
\end{lemma}
\begin{proof} ~
\ref{it:xycomps1} 
If $C=\{u\}$, for some $u\in X$, 
then $y$ dominates $u$ so $u^\perp\bs u\subseteq y^\perp$. In this case,
if $x\in u^\perp$ then $x=u$, since $x\notin y^\perp$, but this contradicts
$x\notin C$. Thus $x\notin u^\perp$ and
 $u^\perp \cap x^\perp=y^\perp\cap u^\perp\cap x^\perp=u^\perp\bs u$;  
so $x$ also dominates $u$ and $C$ is
 a component of $\G_{x^\perp}$. If $C$ contains two elements $u$ and $v$ then there
is a path $p$ 
from $u$ to $v$ which does not meet $y^\perp$. 
If $u$ and $v$ belong to different components of $\G_{x^\perp}$ then $p$ meets
$x^\perp$, and as $x\notin y^\perp$ this means that $x\in C$, a contradiction. Hence
$C\subseteq C^\prime$, for some component $C^\prime$ of $\G_{x^\perp}$. 
As $y^\perp\bs y \subseteq x^\perp$, every component of $\G_{x^\perp}$ containing 
at least $2$ elements is contained in
some component of $\G_{y^\perp}$, so $C=C^\prime$. 

\ref{it:xycomps2}  Suppose that $u\in C_i$, for some 
$i\in \{1,\ldots ,r\}$. Either $u$ belongs to $y^\perp$ or to
some component of $\G_{y^\perp}$. However $y^\perp\bs y\subseteq x^\perp$ and 
$u\notin x^\perp$, so 
$u\notin y^\perp\bs y$. 
As $\{y\}$ is a connected component of $\G_{x^\perp}$, which does
not meet any component of $\G_{y^\perp}$, the vertex $u\neq y$. 
 Hence $u$ belongs to some component $C^\prime$
of $\G_{y^\perp}$. If $x\notin C^\prime$ then, from \ref{it:xycomps1}, $C^\prime=
C_i$, in which case $C^\prime=C$ and $x\in C^\prime$, a contradiction. Hence $x\in
C^\prime$ and $C=C^\prime$; so $C_i\subseteq C$, for all $i$. By definition $C\subseteq
\cup_{i=1}^r C_i\cup x^\perp$, and the  result follows. 
\end{proof}

\begin{lemma}\label{lem:trlinncom}
Let $y\in X$, $v,x\in X^{\pm 1}$,  $\a=\a_{L,y}\in \WLInn$ and 
$\tr=\tr_{v,x}\in \Tr$. 
\be[(i)]
\item\label{it:trlinncom1} If either $v\in L$ and $x\in L\cup y^\perp$ or $v\notin L$, 
$v\neq y^{\pm 1}$ and $x\notin L$ then
\[\a\tr=\tr\a.\]
\item\label{it:trlinncom2} If $v\in L$ and $x\notin L\cup y^\perp$ then $v\in \Isol(y)$ and 
\[\a\tr=\tr_{v,y}\tr\tr_{v,y}^{-1}\a.\]
\item\label{it:trlinncom3} If  $v\notin L$, 
$v\neq y^{\pm 1}$ and $x\in L$  then $v\in \Isol(y)$ and 
\[\a\tr=\tr_{v,y}^{-1}\tr\tr_{v,y}\a.\]
\item\label{it:trlinncom4} If $y=v^{\pm 1}$ and  $x\notin L$ then $L$ is  a union of connected 
components of $\G_{x^\perp}$ and, setting 
%the components of $\G_{x^\perp}$ meeting 
%$L$ non-trivially are $D_1,\ldots , D_r$, $\cK=\cup_{i=1}^r D_i$ and 
$\b=\a_{L,x}$, 
 
\[\a\tr=
\begin{cases}
\tr\b\a,& \textrm{ if } v=y\\
\tr\a\b^{-1}, &\textrm{ if } v=y^{-1}
\end{cases}
.\]
%\item  If $y=v^{\pm 1}$ and  $x\in L$ then $x\in [v]$, $L=\{x\}$ and 
%\[\a\tr=\tr\a\a_{v,x}.\]
\ee
\end{lemma}
\begin{proof}~ 
\ref{it:trlinncom1}
 If $v\in L$ and  $x\in L\cup y^\perp$ then  
\begin{equation}\label{eq:vinL}
z\a\tr=z\tr\a=
\left\{
\begin{array}{ll}
z, &\textrm{ if } z\notin L\\
(vx)^y, &\textrm{ if } z=v \\
z^y, &\textrm{ if } z\in L \textrm{ and } z\neq v^{\pm 1} 
\end{array}
\right.%\} =z\tr\a.%\in \St(\cK)\Conj.
.
\end{equation}
If  $v\notin L$, 
$v\neq y^{\pm 1} $ and $x\notin L$ then 
\begin{equation}\label{eq:yneqv}
z\a\tr=z\tr\a=
\left\{
\begin{array}{ll}
z, &\textrm{ if } z\notin L, z\neq v^{\pm 1} \\
vx, &\textrm{ if } z=v \\
z^y, &\textrm{ if } z\in L
\end{array}
\right. 
.
\end{equation}

\ref{it:trlinncom2}
 In this case $[x,v]\neq 1$, as $v\in L$ and $x\notin y^\perp$; 
so  $x\in \ado(v)$. As
$v\in L$ and $x\notin L\cup y^\perp$, all paths from $v$ to $x$ must 
intersect $y^\perp$, so $v^\perp\bs v\subseteq y^\perp$; and $v\notin y^\perp$,
so $v\in \Isol(y)$. 
 Then $z\a\tr$ is as given in \eqref{eq:vinL}, and is  equal to
$z\tr_{v,y}\tr\tr_{v,y}^{-1}\a$, for all $z\in X$.

\ref{it:trlinncom3}
 If $y\in v^\perp$ then, as $v\neq y^{\pm 1}$, $x\in y^\perp$, a contradiction.
Thus, as in the previous case, $y$ dominates $v$.  Then $z\a\tr$ 
is as given in \eqref{eq:yneqv}, and is equal to
$z\tr_{v,y}^{-1}\tr\tr_{v,y}\a$, for all $z\in X$.

\ref{it:trlinncom4}
 In this case $y$ is dominated by $x$ 
so, from Lemma \ref{lem:xycomps}, $L$ 
is a union of connected components of $\G_{x^\perp}$. Suppose 
$v=y^\e$, where $\e=\pm 1$. Then 
\[
z\a\tr=
\left\{
\begin{array}{ll}
z, &\textrm{ if } z\notin L, z\neq v^{\pm 1}\\
vx, &\textrm{ if } z=v \\
z^{(vx)^\e}, &\textrm{ if } z\in L
\end{array}
\right. 
\]
and this is equal to $z\tr(\b\a^\e)^\e$, for all $z\in X$. 
%\item As $\a\in \TLInn$ and $x\in L$ it follows that 
%$x\notin \out(v)\cup v^\perp$ and so $x\in [v]$. Therefore $L=\{x\}$ and
% the claim holds.  
\end{proof}

\begin{corol}\label{cor:atta}
Let $y\in X$ and $v\in X^{\pm 1}$, $v\neq y^{\pm 1}$. 
 Let $\a=\a_{L,y}\in \TLInn$ and let $\trt_{v,a}\in \Trt$. Then
\[\a\trt_{v,a}=\trt_{v,b}\a,\]
for some $\trt_{v,b}\in \Trt$.   
\end{corol}
\begin{proof}
Let $a=a_1\cdots a_n$, where $a_i\in X^{\pm 1}$, be a geodesic word 
representing $a$. 
 Then $\trt_{v,a}=
\tr_{v,a_n}\cdots \tr_{v,a_1}$. As $v\neq y^{\pm 1}$, 
$\a\tr_{v,a_i}=\tr_{v,y}^{-\e_i}\tr_{v,a_i}\tr_{v,y}^{\e_i}\a$, with 
$\e_i=0$ or $\pm 1$, for all $i$.  The corollary follows on setting $b$ 
equal to the word obtained by freely reducing  
$\prod_{i=1}^{n}y^{\e_i}a_iy^{-\e_i}$.
\end{proof}

\begin{corol}\label{cor:attab}
Let $v\in X$,  
  $\a=\a_{L,v}\in \TLInn$ and let $\tr=\trt_{v,a}\in \Trt$. Then
\[\a\tr=\tr\b,\]
for some $\b\in \la \TLInn\ra$.   
\end{corol}
\begin{proof}
Let $a=a_1\cdots a_n$, where $a_i\in X^{\pm 1}$, be a geodesic word 
representing $a$. By definition of $\Trt$ we have 
$a_i\in (\ad(v)\bs \{v\})^{\pm 1}$, for all $i$. 
Hence, by definition of $\TLInn$, $a_i\notin L$, for all $i$. 
 Thus 
\[\a\tr_{v,a_i}=\tr_{v,a_i}\a_{L,a_i}\a.\]
 Since 
$v\neq a_i^{\pm 1}$, $v\notin L$ and $a_j\notin L$, also
\[\a_{L,a_i}\tr_{v,a_j}=\tr_{v,a_j}\a_{L,a_i}\]
when $i\neq j$. 
Therefore 
\begin{align*}
\a\tr&=\a\tr_{v,a_n}\cdots \tr_{v,a_1}\\
&=\tr_{v,a_n}\cdots \tr_{v,a_1}\a_{L,a_n}\cdots \a_{L,a_1}\a\\
&=\tr\a_{L,a_n}\cdots \a_{L,a_1}\a.
\end{align*}
As $\a\in \TLInn$, we have $\ad(v)\cap L=\nul$ 
and, as $\tr_{v,a_i}\in \Tr$,  we have $\ad(a_i)\subseteq \ad(v)$ 
so $\ad(a_i)\cap L=\nul$. Hence
$\a_{L,a_i}\in \TLInn$; and the result follows. 
\end{proof}

\begin{prop}\label{prop:blintr}
Let $\G$ be a connected graph. Then 
$\la \Tr\cup \TLInn\ra=\la\Tr\ra\la \TLInn\ra$. 
\end{prop}
\begin{proof}
%If $\Sigma$ is a finite set then we denote by $\Sigma^*$ the free monoid
%on $\Sigma$; which consists of words on the alphabet $\Sigma$, and 
%by $\Sigma^+$ the subsemigroup of words of positive length. 
It suffices to prove the proposition holds with $\Trt$ in place 
of $\Tr$.  
First suppose that $u$ is a word on the generators $\TLInn$ and their
inverses and that $\tr\in \Trt$. It follows by a straightforward
induction on $|u|$ and Corollary \ref{cor:attab} that $u\tr=\tr u^\prime$
in $\cSt(\cK)$, for
some word $u^\prime$ over $\TLInn^{\pm 1}$. 

Now let $w$ be a word in the generators of $\la \Trt\cup \TLInn\ra$ and 
their inverses. If $|w|\le 1$ then $w\in \la\Trt\ra\la \TLInn\ra$. 
Assume inductively that for some $k\ge 1$ all words $w$ of length at most $k$
can be expressed as elements of $\la\Trt\ra\la \TLInn\ra$. 

Let $w$ be a word of length $k+1$ (in the given generators). Then
$w=w_0\xi$, for some word $w_0$ of length $k$ and generator $\xi\in 
(\Trt\cup \TLInn)^{\pm 1}$. By induction there exists words $w_1\in \la \Trt\ra$ and 
$w_2\in \la \TLInn\ra$ such that $w_0=w_1w_2$, in $\cSt(\cK)$. 
If $\xi\in \TLInn^{\pm 1}$ the proof is complete. 
Otherwise $\xi\in \Trt^{\pm 1}$ and, from the first part of the proof we
may rewrite $w_2\xi$ to a word $\xi^\prime w_2^\prime$, 
with $\xi^\prime\in \Trt^{\pm 1}$
and $w_2^\prime\in \la \TLInn\ra$, 
such that $w_2\xi=\xi^\prime w_2^\prime$ in $\cSt(\cK)$. 
Then $w=w_1\xi^\prime w_2^\prime\in \la\Trt\ra\la \TLInn\ra$, as required.
\end{proof}

\begin{theorem}\label{theorem:balstcon}
Let $\G$ be a connected graph and $G=G(\G)$. Then $\cSt(\cK)=\St(\cK)\Conj(G)$
if and only if $\G$ is a balanced graph. 
\end{theorem}
\begin{proof}
First suppose that $\G$ is a balanced graph. Let $\i=\i_y\in \Inv$, let
$\a_{L,x}\in \LInn$ and $\tr=\tr_{v,x}\in \Tr$, where $y\in X$ and 
$x, v\in X^{\pm 1}$.  Then $\a\i=\i\a$ unless $y=x^{\pm 1}$, in which case 
$\a\i=\i\a^{-1}$. Also $\tr\i=\i\tr$, unless $y=x^{\pm 1}$, in which case 
$\tr\i=\i\tr^{-1}$, or $v=y^{\pm 1}$, in which case $\tr\i=\i\tr_{v^{-1},x}$. 
It therefore suffices to show that elements of $\la \Tr\cup \LInn\ra$ belong
to $\St(\cK)\Conj(G)$.

First we show that, as $\G$ is balanced, $\la\Tr \cup \LInn\ra$ is generated by $\Tr\cup \Inn\cup\TLInn$. To see this 
suppose that $y\in \Isol(\G)$, $\ado(y)\neq \nul$ and $C_y$  is the component 
of $\G_{y^\perp}$ meeting $\ado(y)$. Let $L=X\bs (C_y\cup y^\perp\cup [y])$, so
$L$ is a union of connected components of $\G_{y^\perp}$ and $\a_{L,y}\in
\TLInn$. Let $D_y=[y]\bs y$ and note that, since $y\in \Isol(\G)$, we have
$[y]\bs y=[y]\bs y^\perp$ so $D_y$ is a union of connected components
of $\G_{y^\perp}$. If $D_y=\nul$ define $\b_y=1$ and otherwise define
$\b_y=\a_{D_y,y}$. Then   
$\a_{C_y,y}\b_y\a_{L,y}=\g_y\in \Inn(G)$, so the generators
$\a_{C_y,y}$ of this form are contained in the subgroup generated by
$\Inn$, $\TLInn$ and the set $\{\b_y\in \WLInn : y\in \Isol(\G),
 \out(y)\neq \nul\}$. 

Now suppose that $y\in \Isol(\G)$ and $[y]\neq \{y\}$. Then for all $v\in[y]$, 
$v\neq y$, we have
$\a_{v,y}=\tr_{v^{-1},y}\tr_{v,y}$, so $\b_y=\a_{D_y,y}=\prod_{v\in D_y} 
 \tr_{v^{-1},y}\tr_{v,y}$. Thus 
all generators $\b_{y}$ are contained in the subgroup generated by 
$\Tr$. It follows that every word on generators $\Tr\cup \LInn$ 
and their inverses may 
be replaced by a word on $\Tr\cup \Inn\cup\TLInn$ and their inverses. Thus
$\la \Tr\cup \LInn\ra=\la  \Tr\cup \Inn\cup\TLInn\ra$. As $\la \Inn \ra$ is
normal in $\Aut(G)$ it suffices to show that elements of $\la \Tr\cup \TLInn\ra$ belong
to $\St(\cK)\Conj(G)$: and this follows from Proposition \ref{prop:blintr}.

For the converse suppose that $\G$ is not a balanced graph. We shall show that 
the obstruction of Example \ref{ex:diagex} is also manifested in $\cSt(\cK)$.
 Indeed the argument is a generalised version of that example.  
 If $\G$ is not balanced then there exists a vertex $v\in\Isol(\G)$ such that 
$\ado(v)\neq \nul$ and there is no component $C$ of $\G_{v^\perp}$ such
that $\ado(v)\subseteq C$.  Let $v$ be such a vertex and let $a,b\in\ado(v)$
such that $a$ and $b$ are in different components $C$ and $B$, respectively,
of $\G_{v^\perp}$. 

Suppose that $a_0\in \ado(v)$ and $a_0<_\cK a$. 
%If $a_0\notin C$ then 
%all paths from $a$ to $a_0$ intersect $v^\perp$, so intersect $v^\perp\bs v$. 
 Then $a_0<_\cK a$ implies that $a^\perp\bs a\subseteq a_0^\perp$ and $a\in \ado(v)$ 
implies that there exists $u\in a^\perp\bs a$ such that $u\notin v^\perp$. Thus 
$a_0\in u^\perp$, so $a_0\in C$.  We may therefore assume that $a$ is 
$\cK$-minimal among elements of $\ado(v)$. Similarly we may assume $b$ is
$\cK$-minimal among elements of $\ado(v)$. 

Define $\phi=\a_{C,v}\tr_{v,a}\tr_{v,b}$ so 
\begin{equation*}
z\phi=
\left\{
\begin{array}{ll}
z, & \textrm{ if } z\notin C\cup\{v\}\\
vba, &\textrm{ if } z=v\\
z^{vba} & \textrm{ if } z\in C
\end{array}
\right.
.
\end{equation*}
Assume that there exist $\g\in \St(\cK)$ and $\d\in \Conj$ such that 
$\phi=\g\d$. 

Note that $Z(G(\ad(v)))$ is generated by $\ad(v)\cap (v^\perp\bs v)$ (which
may be empty). Let $c\in \ad(v)\cap (v^\perp\bs v)$. 
Then $\ad(c)\subseteq \ad(v)\cap (v^\perp\bs v)$ 
 and so if $z\in \ad(c)$ there exists $w_z\in G(\ad(c))$ such that 
 $z\g=w_z$.  
As $Z(G(\ad(v)))$ is Abelian so is $G(\ad(c))$ and  so
$w_z$ is cyclically reduced. 
From Lemma \ref{lem:comm}, 
there exists $g\in G$ such that $z\d=z^g$,  for all $z\in \ad(c)$. 
 Hence, if  $z\in \ad(c)$ then $z=z\phi=z\g\d=w_z\d=w_z^g$, 
with $w_z\in
G(\ad(c))$, so $g=1$ and $w_z=z$. 
 Therefore
$z\g=z\d=z$, for all $z\in Z(G(\ad(v)))$. 
 
As $a$ is $\cK$-minimal among  elements of $\ado(v)$ we have 
$\ad(a)\bs [a]\subseteq \ad(v)\cap (v^\perp\bs v)$. 
As $\d\in \cSt(\cK)$, there exists $g\in G$ such that 
$G(\ad(a))\d=G(\ad(a))^g$ and we may assume that $g$ has no left divisor in $G(\ad(a))$
or $G(\ad(a)^\perp)$. 
Let $z\in [a]$, so $z\phi=z^{vba}$. We have $z\g\in G(\ad(a))$ and so 
$z\g\d=u_z^g$, for some $u_z\in G(\ad(a))$. Therefore $u_z^g=z^{vba}$ and
so $z^{vbag^{-1}}\in G(\ad(a))$. As neither $v$ nor $b$ commute with $a$ or $z$
it follows that $g=g_1\circ vba$, and then $z^{g_1^{-1}}\in G(\ad(a))$. This holds
for all $z\in [a]$, and for any $u\in \ad(a)\bs [a]$ we have $u\d=u$, from
the paragraph above, so  $[u,g]=1$. Since
$g$ has no left divisor in $G(\ad(a)\cup \ad(a)^\perp)$, \cite[Corollary 2.5]{DKR4} 
implies 
that $g_1=1$ and $g=vba$. Now $z\d=z^{g_z}$, for some $g_z\in G$, so we
have $z^{g_z}=w_z^{vba}$, for some $w_z\in G(\ad(a))$. Again $z=w_z^{vbag_z^{-1}}$,
so $z\in \a(w_z)$ and $v,b\notin \ad(a)$, so $g_z=h_z\circ vba$, for some $h_z\in G(\ad(a))$,   and $w_z=z^{h_z}$. As
elements of $\ad(a)\bs [a]$ belong to the centre of $G(\ad(a))$, moreover
$h_z\in G[a]$. Therefore, for all $z\in[a]$,  $z\d=z^{h_zvba}$, for
some $h_z\in G[a]$. 

Similarly, we have 
$\ad(b)\bs [b]\subseteq \ad(v)\cap (v^\perp\bs v)$ and 
 there exists $g\in G$ such that
$G(\ad(b))\d=G(\ad(b))^g$ and  $g$ has no left divisor in $G(\ad(b))$
or $G(\ad(b)^\perp)$.  Let $z\in [b]$, so $z\phi=z$. 
We have $z\g\in G(\ad(b))$ and so 
$z\g\d=u_z^g$, for some $u_z\in G(\ad(b))$. Therefore $u_z^g=z$ and
$z^{g^{-1}}\in G(\ad(b))$.  Thus  $[z,g]=1$, 
which implies $[[b],g]=1$. 
For any $u\in \ad(b)\bs [b]$ we have $u\d=u$, 
so  $[u,g]=1$ and therefore  $g=1$.  
Now $z\d=z^{g_z}$, for some $g_z\in G$, so we
have $z^{g_z}=w_z$, for some $w_z\in G(\ad(b))$. Thus $g_z\in G(\ad(b))$,   and 
as elements of $\ad(b)\bs [b]$ belong to the centre of $G(\ad(b))$, moreover
$g_z\in G[b]$. Therefore, for all $z\in[b]$,  $z\d=z^{g_z}$, for
some $g_z\in G[b]$. 

Now let $z\in v^\perp\bs v$. Then $z\in C_G(a,b)$ so $z\d\in 
C_G(a^{h_avba})=C_G(a^{h_a})^{vba}\subseteq C_G(a)^{vba}=G(a^\perp)^{vba}$ and   
$z\d\in 
C_G(b^{g_b})\subseteq C_G(b)=G(b^\perp)$. If $w\in G(b^\perp)$ and $w=u^{vba}$, where
$u\in G(a^\perp)$,  then $a\notin b^\perp$ implies $a\notin \al(w)$
so $a$, and therefore also $b$ and $v$, 
cancel in reducing $u^{vba}$ to $w$. Neither $b$ nor $v$ belong 
to $\al(u)$, hence $[u,b]=[u,v]=1$ and 
$u\in G(v^\perp\bs v)$. Thus $u^{vba}=u$. It follows that 
$G(a^\perp)^{vba}\cap G(b^\perp)=G(v^\perp\bs v)$ and so  $z\d\in 
G(v^\perp\bs v)$, for all $z\in v^\perp\bs v$.

As $G(v^\perp\bs v)\d=G(v^\perp\bs v)$, for all $z\in \ad(v)$, we have 
$z\d\in G(\ad(v))$, so
$z^\d=z^{g_z}$, for some $g_z\in G(\ad(v))$. Now $\d$ satisfies the following
(with $w=vb$)  
\be
\item\label{it:vun1}  $z\d=z$, for all $z\in \ad(v)\cap v^\perp\bs v$. 
\item $z\d=z^{g_z}$, with $g_z\in G[b]$, for all $z\in [b]$. 
\item $z\d=z^{h_z w a}$, with $w=vb$ or $b$ and $h_z\in G[a]$, for all $z\in [a]$.
\item\label{it:vun4} $z\d=z^{g_z}$, with $g_z\in G(\ad(v))$, for all $z\in \ad(v)$.  
\ee
Let us call an element of $\Conj$ $v${\em -unlikely} if it satisfies all 
of these four properties. Amongst all $v$-unlikely basis conjugating automorphisms
choose one, which we shall now also call $\d$, of minimal length. As usual, 
for 
each $x\in X$ let $g_x\in G$ be such that $x\d=x^{g_x}$. 

From condition \ref{it:vun4}, $\ad(v)\d\subseteq G(\ad(v))$ and 
by direct calculation $\ad(v)\phi^{-1}\subseteq G(\ad(v))$. As $\g\in \St(\cK)$ 
this implies $\ad(v)\phi^{-1}\g=\ad(v)\d^{-1}\subseteq G(\ad(v))$. Hence
 $\d$ restricts to an automorphism
of $G(\ad(v))$ and, applying Lemma \ref{lem:laurlem} to this restriction, there exist elements
$x,y\in \ad(v)$ 
such that $x^{\e}g_x$ is  a right divisor of $g_y$. Moreover, 
$x,y\in \ado(v)\cup [v]$,  as 
the centre of $G(\ad(v))$, which is pointwise fixed by $\d$, 
is generated by $\ad(v)\cap (v^\perp\bs v)$. 
Suppose that $x,y\in C$ and let $D$ be the component of $\G_{x^\perp}$ 
containing $y$. As $x,y\in \ad(v)$, Lemma \ref{lem:xycomps} \ref{it:xycomps2} implies
that  $D\subseteq C$.  
Define $\d_0=\a_{D,x}^{-\e}\d$. For all $z\in X\bs D$ we
have $z\d_0=z\d$ and (applying Lemma \ref{lem:laurlem} again) $|\d_0|<|\d|$. 
If $a\notin D$ then clearly $\d_0$ is $v$-unlikely, contrary to the choice
of $\d$. If $a\in D$ then, for all $z\in [a]\cap D$, $z\neq x$, it follows that
 $x^\e g_x$ is  a right divisor of $h_zwa$,
  which implies $h_z=h_z^\prime x^\e h_z^{\prime\prime}wa$. 
Therefore $z\d_0=z^{h_z^\prime h_z^{\prime\prime}wa}$, 
for all $z\in ([a]\cap D)\bs\{x\}$, and again $\d_0$ is $v$-unlikely, a contradiction.
We may therefore assume that $\{x,y\} \nsubseteq C$. 

Assume that $y\notin C$ and that $D$ is the component of $\G_{x^\perp}$ 
containing $y$. Then $D\cap C=\nul$ and $g_z=g_z^\prime x^\e g_x$, 
for all $z\in D$. Again set  $\d_0=\a_{D,x}^{-\e}\d$ and $\d_0$ is 
$v$-unlikely with $|\d_0|<|\d|$. This contradiction shows that we
may assume $y\in C$ and $x\notin C$. Then $C$ is a component of $\G_{x^\perp}$
 and $x^\e g_x$ is  a right divisor of $h_zwa$, for all $z\in [a]$, 
as $a,y\in C$. As $\al(h_z)\subseteq [a]$ and 
 $x\notin C$ this implies $x=v$ or $b$. If $x=b$ then
$g_b=a$, a contradiction, so we have $x=v$, $w=vb$ and  $g_v=ba$. 

Let 
$\d_0=\a_{C,v}^{-1}\d$, so $z\d_0=z^{h_zba}$, for $z\in [a]$, and $z\d_0
=z\d$, for $z\notin C$. Again $\d_0$ is $v$-unlikely, contrary to minimality
of the length of $\d$.
In all cases we obtain a contradiction, so there exists no $v$-unlikely
automorphism $\d$, completing the proof that $\phi\notin \St(\cK) \Conj$. 

\end{proof}

%%% Local Variables: 
%%% mode: latex
%%% TeX-master: "aut2"
%%% End: 

%\include{projection}
%\begin{comment}
%\include{subgroups}
%\include{stabK}
%\include{uk}
%\end{comment}
%\include{appendix_graphaut}
%\include{generators_long}
%\include{ndg_version}
\section{Appendix}
\subsection{A presentation for the graph automorphisms of $G$}
A presentation for $\Aut_\cmp^\G(G)$ may be constructed using the wreath product structure, of the factors of the 
direct sum, in the decomposition of  Proposition 
\ref{prop:graphaut}\ref{it:graphaut2}. 
 First we establish presentations for these factors. 

Recall from Definition \ref{defn:graphautgens} that $\Aut^\G_\cmp(G_{j,1})$ has 
generating set  $\cP^\G_{\cmp,j}$. By definition $\Aut^\G_\cmp(G_{j,1})\cong
\Aut(\Omega_j)$, and so we may construct a presentation
\[\la \cP^\G_{\cmp,j}| \cR^\G_{\cmp,j}\ra \textrm{ for } \Aut^\G_\cmp(G_{j,1}).\]
Also $\cP^\G_{\sym,j}%(G_{j,*})
=\{\w^j_{a,b}|1\le a<b\le m_j\}$ is a generating set for $\Aut^\G_{\sym}(G_{j,*})$,
and   so we may choose a presentation 
\[
\la
\cP^\G_{\sym,j}|\cR^\G_{\sym,j}\ra \textrm{ for }\Aut^\G_{\sym}(G_{j,*}).
\]
Let
\begin{align*}
\cW^\G_j&=\{[\w^j_{a,b},p]:p\in \cP^\G_{\cmp,j},2\le a<b\le m_j\}\\
&\cup \{[p, \w^j_{1,a}q\w^j_{1,a}]: p,q\in \cP^\G_{\cmp,j},2\le a\le m_j \}\\
&\cup \{[\w^j_{1,a}p\w^j_{1,a},\w^j_{1,b}q\w^j_{1,b}]: 
p,q\in \cP^\G_{\cmp,j},2\le a<b\le m_j  \}.
\end{align*}
Let 
\[\cP^\G_j= \cP^\G_{\cmp,j}\cup \cP^\G_{\sym,j}\textrm{ and }
\cR^\G_j= \cR^\G_{\cmp,j}\cup \cR^\G_{\sym,j}\cup \cW^\G_j.
\]
\begin{prop}\label{prop:gaut_fac}
$\prod_{k=1}^{m_j}\Aut^\G_\cmp(G_{j,k})
\rtimes \Aut^\G_{\sym}(G_{j,*})$
 has presentation $\la \cP^\G_j| \cR^\G_j\ra$. 
\end{prop}
\begin{proof}
$\Aut^\G_{\cmp}(G_{j,k})\cong \Aut^\G_{\cmp}(G_{j,1})$, for $k=2,\ldots ,m_j$, 
and $\Aut^\G_{\sym}(G_{j,*})$ acts on $\prod_{k=1}^{m_j}\Aut^\G_\cmp(G_{j,k})$ by permuting the factors. 
Hence the group in question is
 a wreath product; and the given presentation is obtained from a standard
 construction. 
\end{proof}
As $\Aut^\G_\cmp(G)$ is a direct sum of the groups of the previous lemma 
a presentation can be written down immediately. In order to do so define
\[
\cD^\G=\{[p,q]:p\in \cP^\G_i, q\in \cP^\G_j, 1\le i<j\le d\}.
\]
From Proposition \ref{prop:graphaut}\ref{it:graphaut2} we obtain
the next corollary.
\begin{corol}\label{cor:gaut}
$\Aut^G_\cmp(G)$ has presentation $\la \cP^\G_\cmp
 |
\cR^\G_\cmp
\ra$, 
where $\cP^\G_\cmp=\cup_{j=1}^d \cP^\G_j$ and $\cR^\G_\cmp=\cup_{j=1}^d \cR^\G_j
\cup \cD^\G$.
\end{corol}

\subsection*{Proof of Theorem \ref{prop:presentation}}

\begin{proof}[Proof of Proposition \ref{prop:presentation}]
Let $\AA$ be the group with presentation $\la \cP|\cR\ra$. Identifying
each generator of $\cP$ with the elements of the same name in $\Aut(G)$
straightforward computation shows that  all the 
relators in $\cR$ hold in $\Aut(G)$; giving a canonical homomorphism 
$\Theta$ 
from $\AA$ to $\Aut(G)$. We shall use the presentation of $\Aut(G)$
given in \cite{Gilbert87}, which we shall call 
$\la \cQ|\cS\ra$,  to construct an inverse to $\Theta$. 

To define $\la \cQ|\cS\ra$ the automorphisms of a free product are divided into four types, in
\cite{Gilbert87}. The first two types are the permutation and factor autormorphisms.
The \emph{permutation automorhpisms} are those belonging to the subgroup
 $\Aut^\G_{\sym}(G)$. The \emph{factor automorphisms} are those automorphisms 
$\a$ such that 
$\a$ restricted to $G(\G_{j,k})$ is an automorphism of $G(\G_{j,k})$, 
for all $(j,k)\in S\cup J$.   Let $\Psi$ be the subgroup
generated by the permutation and factor automorphisms. 

The first step in the definition of $\la \cQ|\cS\ra$ 
is to choose a presentation for $\Psi$. In our case  
we extend the notation of Definition \ref{defn:factgen} to 
denote by $\Aut(G_{j,k})$ the subgroup of automorphisms $\phi$ such that 
$x\phi=x$, if $x\in X\bs X_{j,k}$ and $X_{j,k}\phi \subseteq G(\G_{j,k})$. 
Then 
$\Psi$ is generated by the subgroups $\Aut(G_{j,k})$ and 
$\Aut^\G_\sym(G_{j,*})$, for $0\le j\le d$,
$1\le k\le m_j$ (see Definition \ref{defn:graphaut}).
For fixed $j$, with $0\le j\le d$, let $\Psi_j$ be the subgroup of $\Psi$
generated by  $\Aut(G_{j,k})$, for 
$1\le k\le m_j$, and 
$\Aut^\G_\sym(G_{j,*})$. Then 
\[\Psi_j = \prod_{k=1}^{m_j}\Aut(G_{j,k})
\rtimes \Aut^\G_{\sym}(G_{j,*}).\] 
Let
 \begin{align*}
\cW_j&=\{[\w^j_{a,b},p]:p\in \cP_{j},2\le a<b\le m_j\}\\
&\cup \{[p, \w^j_{1,a}q\w^j_{1,a}]: p,q\in \cP_{j},2\le a\le m_j \}\\
&\cup \{[\w^j_{1,a}p\w^j_{1,a},\w^j_{1,b}q\w^j_{1,b}]: 
p,q\in \cP_{j},2\le a<b\le m_j  \}.
\end{align*}
(Thus $\cW_j\supseteq \cW^\G_j$.) Then
\begin{prop}\label{prop:aut_fac}
$\Psi_j$ has a presentation 
\[\la \cP_{j}\cup \cP^\G_{\sym,j} |
\cR_{j}\cup \cR^\G_{\sym,j}\cup \cW_j\ra.
\]
\end{prop}
\begin{proof}
$\Aut(G_{j,k})\cong \Aut(G_{j,1})$, for $k=2,\ldots ,m_j$, 
and $\Aut^\G_{\sym}(G_{j,*})$ acts on $\prod_{k=1}^{m_j}\Aut(G_{j,k})$ by permuting the factors. The result follows as in the proof of Proposition
\ref{prop:gaut_fac}. 
\end{proof}
%%%%%
As $\Psi=\prod_{j=0}^d\Psi_j$ we have
\begin{corol}\label{cor:psiaut}
$\Psi$ has presentation $\la \cQ_\Psi| \cS_\Psi\ra$ 
where $\cQ_\Psi=\cP_\int\cup \cP^\G_\cmp$ (see Definitions
\ref{defn:graphautgens} and \ref{defn:ourgens}) and 
 $\cS_\Psi=\cup_{j=0}^d (\cR_j \cup\cR^\G_{\sym,j}\cup \cW_j)\cup \cD$.
\end{corol}
\begin{proof}
This follows from Proposition \ref{prop:aut_fac} as 
\begin{align*}
\cup_{j=0}^d (\cP_{j}\cup \cP^\G_{\sym,j}) &= 
\cup_{j=0}^d (\cP_\int \cap \Aut(G_{j,1})) \cup \cup_{j=0}^d \cP^\G_{\cmp,j}
\cup \cup_{j=0}^d  \cP^\G_{\sym,j}\\
&=\cP_\int\cup \cP^\G_\cmp.
\end{align*}
\end{proof}

The generators $\cQ$  consist of $\cQ_\Psi$ together with  a 
set $\cQ_{\WH}$ of elements of $\la \LInne \cup \Tre\ra$,
called{\em Whitehead automorphisms},  which 
we now define. First, for $a\in \cup_{j\in J} G(\G_j)\cup X_S^{\pm 1}$  
we define 
\[\hat a= 
\left\{
\begin{array}{ll}
j & \textrm{ if } a\in G(\G_j)\\
x & \textrm{ if } a=x^\e, \textrm{ where } x\in X_S, \e={\pm 1}
\end{array}
\right. 
.
\] 
(Thus, in comparison to the notation of page \pageref{eq:vat}, 
 if $a\in G(G_j)$ or $a\in X_S$ 
then $\hat a=\vat a$, whereas if $a\in X_S^{-1}$
then $\hat a= \vat a^{-1}$.)
For $i, j\in J$ with $i\neq j$,  $a\in G(\G_j)\cup X^{\pm 1}_S$ and  $x\in X_S^{\pm 1}$,
 with $x^{\pm 1}\neq a$,  
extend the notation for transvections and locally inner automorphisms
to denote by 
\be
\item\label{it:trext} $\tr_{x,a}$ the automorphism $\tau$ such that $x\tau =xa$ and 
$y\tau=y$, for all $y\in X$, $y\neq x$ and 
\item\label{it:linnext}
 $\a_{X_i,a}$ the automorphism $\a$ such that $u\a=u^a$, for 
all $u\in X_i$ and $z\a=z$, for all $z\in X\backslash {X_i}$. 
\ee
A Whitehead automorphism is an element of 
$\la \LInne \cup \Tre\ra$, determined by an ordered pair $(A,a)$, where
$A$ is a subset of $J\cup X_S\cup X_S^{-1}$ and 
$a\in \cup_{j\in J}G(\G_j)\cup X_S^{\pm 1}$, satisfying
the condition that 
\begin{itemize}
\item 
$\hat a\in A$ and 
\item
if $a\in X^{\pm 1}_S$ then $a^{-1}\notin A$.
\end{itemize} 
Partitioning $A\bs \{\hat a\}$ as $A\bs \{\hat a\}=A_J\cup A_S$, where 
$A_J=(A\cap J)\bs \{\hat a\}$ and 
$A_S=(A\cap X_S^{\pm 1})\bs\{ a\}$, the pair $(A,a)$ determines
the automorphism 
\begin{equation}\label{eq:whaut}
\prod_{j\in A_J} \a_{X_j,a}\prod_{y\in A_S} \tr_{y,a}.
\end{equation}
The set $\cQ_{\WH}$ consists of all Whitehead automorphisms and 
the set $\cQ$ of generators of $\Aut(G)$ is the 
union $\cQ=\cQ_\Psi\cup \cQ_{\WH}$. 

The relators $\cS$ consist of the relations $\cS_\Psi$ together with
relators $\cS 1$--$\cS 9$ below, for which we need to introduce
some terminology. Recall that, 
for $h\in G(\G_i)$, where $i\in J$, we have  defined $\g_h(i)$ to be the automorphism
mapping $g$ to $g^h$, for all $g\in G(\G_i)$, 
and fixing all elements of $X_j$, 
where $j\neq i$. Clearly $\g_h(i)$ is a product of elements of 
$\LInni\subseteq \cQ_\Psi$.  
%For a Whitehead automorphism
%$(A,a)$ we write $\hat A=A\cup \{\hat a\}$. 
We use the same sublabelling, J, S, L, 
of relators as \cite{Gilbert87} and these relators apply to
all possible Whitehead automorphisms. In particular relators involving
elements of $X_S$ are defined only if $m_0>0$, in Definition 
\ref{defn:comps}. Given a set $W$, subsets $U,V$ of  $W$ and $x\in W$,
 we  write $U+V$, $V+x$ and $V-x$ to denote $U\cup V$, 
$V\cup \{x\}$ and $V\bs \{x\}$ respectively. 
\begin{description}
\item[$\cS 1$] ~
\begin{description}
\item[J]
$(A,a)^{-1}=(A,a^{-1})$, if $\hat a\in J$, and 
\item[S] $(A,a)^{-1}=(A-a+a^{-1},a^{-1})$, if $a\in X_S^{\pm 1}$.
\end{description}
\item[$\cS 2$] $(A,a)(B,b)=(B,b)(A,a)$, if $A\cap B=\nul$ and 
either 
\begin{description}
\item[J] $\hat a, \hat b\in J$, or 
\item[S]  $a, b\in X_S^{\pm 1}$, 
$a^{-1}\notin B$, $b^{-1}\notin A$, or
\item[L] $\hat a\in J$, $b\in X_S^{\pm 1}$, $b^{-1}\notin A$.  
\end{description} 
\item[$\cS 3$] $(A,a)(B,b)=(B,b)(A+ B-b,a)$, 
if $A\cap B=\nul$ and 
either 
\begin{description}
\item[S]  $a, b\in X_S^{\pm 1}$, 
$a^{-1}\notin B$, $b^{-1}\in A$, or
\item[L] $\hat a\in J$, $b\in X_S^{\pm 1}$, $b^{-1}\in A$.
\end{description} 
\item[$\cS 4$]  $(A,a)(B,a)=(A+ B,a)$, 
if $A\cap B=\{a\}$ and $a\in X_S^{\pm 1}$. 
\item[$\cS 5$] If $\hat a=\hat b\in J$ and $A\cap B=\{\hat a\}$ then 
\be[(i)]
\item $(A,a)(B,b)=(B,b)(A,a)$,
\item $(A,a)(B,a)=(A+ B,a)$ and
\item $(A,a)(A,b)=(A,ba)$.
\ee
\item[$\cS 6$] $\phi^{-1}(A,a)\phi=(A\phi,a\phi)$, for all $\phi\in \Psi$, 
with the natural interpretation of $A\phi$. 
\item[$\cS 7$] 
If $a, b\in X_S^{\pm 1}$, $a\in X_{0,s}$, $b\in X_{0,t}$, $s\neq t$, $b\in A$, $b^{-1}\notin A$, 
%$\w_{x,y}$ denotes the 
%element of $\cP_{\Pi,d^\sharp}$ interchanging elements $x$ and $y$ of $X_S$, 
 $v$ is the unique element of $X_{0,1}$ and $\rho$ denotes
 the word $\w^0_{1,s}\i_v\w^0_{1,s}\w^0_{s,t}$ in the generators $\cQ_\Psi$ 
(so $\rho$ is the cyclic permutation $(a,b^{-1},a^{-1},b)$) then 
\[(A,a)(A-a+a^{-1},b)=\rho (A-b+b^{-1},a).\]
\item[$\cS 8$] $(A,a)(B,b)=(B,b)(A,a)$, if $A\subseteq B$ 
and  $\hat b\notin A$ 
and either 
\begin{description}
\item[J]  $\hat a\in J$, or 
\item[S] $a\in X_S^{\pm 1}$, $a^{-1}\in B$.
\end{description}
\item[$\cS 9$] If $A\subseteq B$, 
$\hat a\in J$, $b\in X_S^{\pm 1}$ 
and  $b \in A$ then 
\[(A,a)(B,b)=(B,b)(B-A+\hat a+b^{-1},a^{-1})\g_{a^{-1}}(\hat a).\] 

\end{description}
In fact in \cite{Gilbert87} a larger set of generators is used involving
certain products of Whitehead automorphisms. However these
 additional generators  can all be
 removed, using Tietze transformations, to give the presentation 
$\la \cQ|\cS\ra$ above for $\Aut(G)$. 

Now let  $\Phi$ be the map from $\cQ$ to
$\AA$ defined as follows. Each element of $\cQ_\Psi$ is mapped 
to the element of the 
same name in the generators of 
$\AA$.   
 For $i\in J$ and  $x,y\in X^{\pm 1}$, with 
$x\neq y$, $x\notin X_i^{\pm 1}$ and $y\in X_S^{\pm 1}$,  the 
Whitehead automorphisms $(\{i,\hat x\},x)$ and $(\{\hat x, y\},x)$ map to 
$\a_{X_i,x}\in \LInne$ and $\tr_{y,x}\in \Tre$, respectively. 
To  define $\Phi$ on general Whitehead automorphisms 
first choose a geodesic word representing each element of $G(\G_i)$, 
for each $i\in J$.   If $g$ is represented by the geodesic 
word $a_1\cdots a_m$, with $a_i\in X_j^{\pm 1}$, $j\in J$, then 
the Whitehead automorphisms 
 $(\{i,j\},g)$ and $(\{j, y\},g)$ map to the words 
%$\a_{X_i,g}$ and $\tr_{x,g}$ map to 
$\a_{X_i,a_m}\cdots \a_{X_i,a_1}$ and $\tr_{x,a_m}\cdots \tr_{x,a_1}$,
over $\cP$, 
which we write as $\tilde{\a}_{X_i,g}$ and $\tilde{\tr}_{x,g}$, respectively, 
\emph{cf.} Definition \ref{defn:comptr}. 
(The $\tilde{}$ indicates that these are words over $\cP$, in the 
presentation of $\AA$, as opposed to elements of $\Aut(G)$.)  
 Finally $\Phi$ maps the Whitehead automorphism $(A,a)$ to
\begin{equation}\label{eq:a-whaut}
\prod_{j\in A_J} \tilde{\a}_{X_j,a}\prod_{y\in A_S} \tilde{\tr}_{y,a}
\end{equation}
(cf.  \eqref{eq:whaut}).
To see that this is a well defined map note that from 
\ref{it:R1}, \ref{it:R4} and \ref{it:R6} it follows that all terms of the product
\eqref{eq:a-whaut} commute with each other: so the order in which the
elements of $A$ appear in this product  does not affect the image
$(A,a)\Phi$ of $(A,a)$ in $\AA$. 

We claim that the natural extension of this map to $\Aut(G)$ determines
 a homomorphism $\Phi: \Aut(G)\maps \AA$. 
Clearly all the relators of $\cS_\Psi$ map to the identity of $\AA$. 
To prove the claim we need to check that the same is true of the 
relators $\cS$1--$\cS$9. 
First we 
establish a useful consequence of the relators \ref{it:last} 
%and \ref{it:lastdash} 
of $\AA$.
\be[{$\cR$}1.,ref={$\cR$}\arabic*]
\setcounter{enumi}{11}
\item\label{it:R12} Let $i,j\in J$, with $i\neq j$, and 
 let $x\in X_S^{\pm 1}$. If $a_1\cdots a_m=b_1\cdots b_n$ 
are geodesic words 
in $G(\G_i)$, 
with $a_i, b_i\in X_i^{\pm 1}$, then
\[ 
\a_{X_j,a_m}\cdots \a_{X_j,a_1}
=\a_{X_j,b_n}\cdots \a_{X_j,b_1}\quad
\textrm{ and }\quad\tr_{x,a_m}\cdots \tr_{x,a_1}=
\tr_{x,b_n}\cdots \tr_{x,b_1},
\] 
(where elements $\tr_{x,y}^{-1}$ of $\Tre^{-1}$ are written as 
$\tr_{x,y^{-1}}$).
\ee
(Therefore the definition of $\Phi$
is in fact independent of the choice of geodesic word for each element
of $G(\G_i)$.)

We now check $\cS$1--$S$9 in turn to see that they become relations of 
$\AA$. To fix notation let us assume that, whenever $(A,a)$ and $(B,b)$
 are Whitehead automorphims we have 
\begin{align*}
(A,a)\Phi&= \prod_{j\in A_J} \tilde{\a}_{X_j,a}\prod_{y\in A_S} \tilde{\tr}_{y,a}\textrm{ and }\\
(B,b)\Phi &= \prod_{k\in B_J} \tilde{\a}_{X_k,b}\prod_{z\in B_S} \tilde{\tr}_{z,b}. 
\end{align*}
As the terms of these products are products $\tilde \tr_{.,.}$ and 
$\tilde \a_{.,.}$  of  transvections
and locally inner automorphisms, it 
useful to establish versions of the relators \ref{it:R1}--\ref{it:lastdash}
for such automorphisms. With this in mind consider the analogues 
 of these
relators where generators $\tr_{a,s}$ and 
$\a_{X_n,s}$, with $a\in X_S^{\pm 1}$, $n\in J$ and $s\in X_J^{\pm 1}$,  are replaced by $\tilde \tr_{a,w}$ and $\tilde \a_{X_n,w}$, 
where $w$ may be any element of $G(\G_{\vat s})$ 
(and the conditions on the relators remain otherwise unchanged). 
This affects   $y$ and $v$ in \ref{it:R1}; $y$ in \ref{it:R2},
\ref{it:R7}, \ref{it:R9}
 and \ref{it:lastdash}; 
$x$ and $y$ in \ref{it:R4} and  \ref{it:R5}; 
$y$ and $z$ in \ref{it:R6} and  \ref{it:R8}; and $u$, $y$ and $z$ in \ref{it:last}.

Denote the $\tilde{}$ version of $\cR$j by $\cR$j\~{}. 
Then \ref{it:R1}\~{}, \ref{it:R4}\~{}, \ref{it:R6}\~{}, 
\ref{it:last}\~{} and \ref{it:lastdash}\~{} follow directly 
from
 the original versions. \ref{it:R2}\~{} follows using \ref{it:R2} and 
\ref{it:R1}. Similarly,  \ref{it:R5}\~{} follows from \ref{it:R5} and 
\ref{it:R4}; \ref{it:R7}\~{} follows from \ref{it:R7} and 
\ref{it:R6}; \ref{it:R8}\~{} follows from \ref{it:R8} and 
\ref{it:R6}. \ref{it:R9}\~{} follows from \ref{it:R9} and \ref{it:lastdash}. 
Therefore we may now assume each relator $\cR j$ is in fact the relator 
$\cR j\tilde{}$ (and drop the \~{}). 

Given the comment following \eqref{eq:a-whaut}, we have in $\AA$ 
\[
\left((A,a)\Phi
\right)^{-1} =
\prod_{j\in A_J} \tilde{\a}_{X_j,a^{-1}}\prod_{y\in A_S} \tilde{\tr}_{y,a^{-1}}.
\]
Therefore 
the relator $\cS$1 follows from %\ref{it:last} and 
\ref{it:R12}. 

To verify 
relators $\cS$2 we must check that, for all $j\in A_J$, $k\in B_J$, 
$y\in A_S$ and $z\in B_S$, we have 
\[
[\tilde{\a}_{X_j,a},\tilde{\a}_{X_k,b}]
=[\tilde{\tr}_{y,a},\tilde{\a}_{X_k,b}]=[\tilde{\a}_{X_j,a},\tilde{\tr}_{z,b}]
=[\tilde{\tr}_{y,a},\tilde{\tr}_{z,b}]=1,
\]
in $\AA$. Assume the conditions of $\cS$2 hold. 
 As $A\cap B=\emptyset$ we have in all cases $j\neq k$ and $y\neq z$;
 so $y=z^{-1}$ or $\vat y\neq \vat z$. 
In case \textbf{J} we have 
 $a\notin G(\G_k)$ and $b\notin G(\G_j)$, so $\vat a,\vat b\notin
\{k,j\}$, and $y,z\notin \{\vat a ,\vat b\}$. In 
case \textbf{S} we have again $\vat a,\vat b\notin  \{j, k\}$, $y\neq b$, and 
$y\neq b^{-1}$, as $b^{-1}\notin A$, and similarly $z^{\pm 1} \neq a$. Hence
  $y,z\notin \{\vat a ,\vat b\}$. 
 In case \textbf{L} we have $\vat a,\vat b\notin \{j,k\}$ and 
$y,z\notin \{\vat a ,\vat b\}$, as before.  
 Therefore relation \ref{it:R4} implies that $[\tilde{\a}_{X_j,a},\tilde{\a}_{X_k,b}]=1$;  relation \ref{it:R6} implies that 
$[\tilde{\tr}_{y,a},\tilde{\a}_{X_k,b}]=[\tilde{\a}_{X_j,a},\tilde{\tr}_{z,b}]
=1$ 
 and \ref{it:R1}(i) \& (ii) imply that $[\tilde{\tr}_{y,a},\tilde{\tr}_{z,b}]=1$.

 To see $\cS$3 holds let $A^\prime=A\bs\{b^{-1}\}$, so 
$(A,a)\Phi=\tilde{\tr}_{b^{-1},a}((A^\prime,a)\Phi)$. Since $\cS$2 maps to the a
relation of $\AA$, 
 $(A^\prime,a)\Phi(B,b)\Phi=(B,b)\Phi(A^\prime, a)\Phi$ and 
hence it suffices to show that 
\[\tilde{\tr}_{b^{-1},a}(B,b)\Phi=(B,b)\Phi (B-b+a,a)\Phi\tilde{\tr}_{b^{-1},a},\] 
that is
\[\tilde{\tr}_{b^{-1},a}\left(\prod_{k\in B_J}\tilde{\a}_{X_k,b}\prod_{z\in B_S}\tilde{\tr}_{z,b}
\right)
=\left(\prod_{k\in B_J}\tilde{\a}_{X_k,b}\prod_{z\in B_S}\tilde{\tr}_{z,b}
\prod_{k\in B_J}\tilde{\a}_{X_k,a}\prod_{z\in B_S}\tilde{\tr}_{z,a}\right)\tilde{\tr}_{b^{-1},a}.\]
From \ref{it:R7} and the conditions of $\cS$3 we have 
$\tilde{\tr}_{b^{-1},a}\tilde{\a}_{X_k,b}=\tilde{\a}_{X_k,b}\tilde{\a}_{X_k,a}\tilde{\tr}_{b^{-1},a}$ and 
using \ref{it:R4} we obtain 
\[\tilde{\tr}_{b^{-1},a}\left(\prod_{k\in B_J}\tilde{\a}_{X_k,b}\right)
=\left(\prod_{k\in B_J}\tilde{\a}_{X_k,b}\prod_{k\in B_J}\tilde{\a}_{X_k,a}
\right)\tilde{\tr}_{b^{-1},a}.\]
Similarly, using \ref{it:R2} and \ref{it:R1} we have
\[\tilde{\tr}_{b^{-1},a}\left(\prod_{z\in B_S}\tilde{\tr}_{z,b}
\right)
=\left(\prod_{z\in B_S}\tilde{\tr}_{z,b}\prod_{z\in B_S}\tilde{\tr}_{z,a}\right)
\tilde{\tr}_{b^{-1},a}.\]
Finally, \ref{it:R6} may be applied to give the required result. 

That  $\cS$4 holds after mapping to $\AA$ is another consequence of 
the  the remarks following \eqref{eq:a-whaut}. 

If the conditions of $\cS$5 hold then,
for all $y\in A_S$, $j\in A_J$ and $k\in B_J$ 
we have $\vat a,\vat b\notin\{j, k\}$ and
$\vat y\neq \vat b$; so from \ref{it:R6}, $[\tilde{\a}_{X_k,b},\tilde{\tr}_{y,a}]=1$. 
As also $\vat j\neq \vat k$, 
\ref{it:R4} applies to give $[\tilde{\a}_{X_j,a},\tilde{\a}_{X_k,b}]=1$. For
all $y\in A_S$ and $z\in B_S$ we have also $\vat y,\vat z\notin 
\{\vat \,\vat b\}$, and either $y=z^{-1}$ or $\vat y\neq \vat z$. Hence,
 from \ref{it:R1},
$[\tilde{\tr}_{y,a},\tilde{\tr}{z,b}]=1$. Therefore $\cS$5 (i)  holds after mapping
to $\AA$. $\cS$5 (ii) is dealt with using the remarks following \ref{eq:a-whaut}.
From the above the conditions of \ref{it:R6} also apply to give 
$[\tilde{\a}_{X_j,b},\tilde{\tr}_{y,a}]=1$, so $\cS$5 (iii)  holds after mapping
to $\AA$. 

If $\phi$ is a generator of  $\Psi$ then we have, from \ref{it:lastdash}, 
$\phi^{-1} \tilde{\a}_{X_j,a}\phi=\tilde{\a}_{X_j\phi,a\phi}$ and 
$\phi^{-1} \tilde{\tr}_{y,a}\phi = \tilde{\tr}_{y\phi,a\phi}$. Consequently the equality
of $\cS$6 holds on mapping to $\AA$. 

In the case  of $\cS$7, let $A_S^\prime=A_S\bs \{b\}$. Then 
we must show that 
\begin{equation*}
\begin{split}
\left(\prod_{j\in A_J} \tilde{\a}_{X_j,a}\prod_{y\in A_S} \tilde{\tr}_{y,a}
\right)&
\left(
\prod_{j\in A_J} \tilde{\a}_{X_j,b}\prod_{y\in A_S^\prime} \tilde{\tr}_{y,b}
\right)\tilde{\tr}_{a^{-1},b}=\\[.5em]
&\rho
\left(
\prod_{j\in A_J} \tilde{\a}_{X_j,a}\prod_{y\in A_S^\prime} \tilde{\tr}_{y,a}
\right)\tilde{\tr}_{b^{-1},a}.
\end{split}
\end{equation*}
If the conditions of $\cS$7 hold then we may 
apply \ref{it:R7}, \ref{it:R6}
\ref{it:R4}, \ref{it:R2} and \ref{it:R1} to the left hand side to obtain
\begin{align*}
\left(\prod_{j\in A_J} \right.&\left.\tilde{\a}_{X_j,a}\prod_{y\in A_S} \tilde{\tr}_{y,a}
\right)
\left(
\prod_{j\in A_J} \tilde{\a}_{X_j,b}\prod_{y\in A_S^\prime} \tilde{\tr}_{y,b}
\right)\tilde{\tr}_{a^{-1},b}\\[.5em]
&= \left(\prod_{j\in A_J} \tilde{\a}_{X_j,a}\prod_{y\in A_S^\prime} \tilde{\tr}_{y,a}
\right)\tilde{\tr}_{b,a}
\left(
\prod_{j\in A_J} \tilde{\a}_{X_j,b}\prod_{y\in A_S^\prime} \tilde{\tr}_{y,b}
\right)\tilde{\tr}_{a^{-1},b}\\[.5em]
&=\left(\prod_{j\in A_J} \tilde{\a}_{X_j,b}\prod_{y\in A_S^\prime} \tilde{\tr}_{y,b}
\right)\tilde{\tr}_{b,a}\tilde{\tr}_{a^{-1},b}.
\end{align*}
From \ref{it:R3} we have $\tilde{\tr}_{a,b}^{-1}
\tilde{\tr}_{b,a}\tilde{\tr}_{a^{-1},b}= \rho$, so 
$\tilde{\tr}_{b,a}\tilde{\tr}_{a^{-1},b}=\tilde{\tr}_{a,b}\rho$, and 
from \ref{it:lastdash} then 
$\tilde{\tr}_{b,a}\tilde{\tr}_{a^{-1},b}=\rho\tilde{\tr}_{b^{-1},a}$. 
 A final application of \ref{it:lastdash} then gives the required 
result. 

If the conditions of $\cS$8 hold then $A\subseteq B$ and from $\cS$4 and
$\cS$5 it follows that $(B,b)=(A+\hat b,b)(B\bs A,b)$. In case \textbf{S}
of $\cS$8, this means that
\begin{align*}
(B,b)(A,a) & = (A+\hat b,b)(B\bs A,b)(A,a)\\
& = (A+\hat b,b)(A,a)(B\bs A +A-a,b),\textrm{ using $\cS$3},\\
&=(A+\hat b,b)(A,a)(B-a,b).
\end{align*}
Now $\cS$1 implies that $(A,a)^{-1}=(A-a+a^{-1},a^{-1})$ and $\cS$3 implies
that  
\begin{align*}
(\{a,\hat b\},b)(A-a+a^{-1},a^{-1})&=(A-a+a^{-1},a^{-1})(A+\hat b,b),
\end{align*}
so 
\[(A+\hat b,b)(A,a)=(A,a)(\{a,\hat b\},b).\]
Therefore 
\begin{align*}
(B,b)(A,a)&=(A,a)(\{a,\hat b\},b)(B-a,b)\\
&= (A,a)(B,b),\textrm{ using $\cS4$ and $\cS$5}.
\end{align*}
Thus, case \textbf{S} of $\cS$8 follows from $\cS$1, $\cS$3, $\cS$4 and 
$\cS5$. Since the latter all hold after mapping into $\AA$, the 
same is true of $\cS$8, \textbf{S}. Hence it remains to consider 
$\cS$8, \textbf{J}. In this case we have, from $\cS$2, that 
\begin{align*}
(B,b)(A,a)&= (A+\hat b)(B\bs A,b)(A,a)\\
&=(A+\hat b,b)(A,a)(B\bs A, b).
\end{align*}
 As $\cS$2 holds in $\AA$ it therefore suffices to check
that 
\[(A+\hat b,b)(A,a)=(A,a)(A+\hat b,b)\] holds after mapping to $\AA$; that is 
\[
\begin{split}
\left(\tilde{\a}_{X_i,b}\prod_{j\in A_J} \tilde{\a}_{X_j,b}
\prod_{y\in A_J} \tilde{\tr}_{y,b}\right)&
\left(\prod_{j\in A_J} \tilde{\a}_{X_j,a}
\prod_{y\in A_J} \tilde{\tr}_{y,a}\right)\\[.5em]
&=
\left(\prod_{j\in A_J} \tilde{\a}_{X_j,a}
\prod_{y\in A_J} \tilde{\tr}_{y,a}\right)
\left(\tilde{\a}_{X_i,b}\prod_{j\in A_J} \tilde{\a}_{X_j,b}
\prod_{y\in A_J} \tilde{\tr}_{y,b}\right),
\end{split}
\] 
where $i=\hat a \in J$ and $\hat b\notin A$. Let $w$ denote the 
left hand side of the above expression. 
From \ref{it:R6}, we have $[\tilde{\tr}_{y,b},\tilde{\a}_{j,a}]=1$, for $y\in A_S$, 
$j\in A_J$; 
from \ref{it:R4}, we have $[\tilde{\a}_{j_1,b},\tilde{\a}_{j_2,a}]=1$, for $j_1\neq j_2\in 
A_J$; and from \ref{it:R1}, we have 
$[\tilde{\tr}_{y_1,b},\tilde{\tr}_{y_2,a}]=1$, for $y_1\neq y_2\in A_S$. Hence
\[w=
\tilde{\a}_{X_i,b}\prod_{j\in A_J} \tilde{\a}_{X_j,b}\tilde{\a}_{X_j,a}
\prod_{y\in A_J} \tilde{\tr}_{y,b}\tilde{\tr}_{y,a}.\]
From \ref{it:R5}, we have 
$\tilde{\a}_{X_i,b}\tilde{\a}_{X_j,b}\tilde{\a}_{X_j,a}=\tilde{\a}_{X_j,a}\tilde{\a}_{X_i,b}\tilde{\a}_{X_j,b}$ and 
from \ref{it:R4}, we have $[\tilde{\a}_{X_i,b},\tilde{\a}_{X_j,b}]=1$, for $j\in A_J$. 
Thus
\[w=
\left(\prod_{j\in A_J} \tilde{\a}_{X_j,a}\right)
\left(\prod_{j\in A_J} \tilde{\a}_{X_j,b}\right)
\tilde{\a}_{X_i,b}
\left(\prod_{y\in A_J} \tilde{\tr}_{y,b}\tilde{\tr}_{y,a}\right).
\]
Using \ref{it:R8} and \ref{it:R6} in a similar fashion, we finally obtain
\[w=\left(\prod_{j\in A_J} \tilde{\a}_{X_j,a}
\prod_{y\in A_J} \tilde{\tr}_{y,a}\right)
\left(\tilde{\a}_{X_i,b}\prod_{j\in A_J} \tilde{\a}_{X_j,b}
\prod_{y\in A_J} \tilde{\tr}_{y,b}\right),
\]
as required.

To establish that $\cS$9 maps to an equality in $\AA$, first consider
the special case where $\hat a \in J$,
$b\in X^{\pm 1}_S$, $\hat a\in B$ and   
$A=\{\hat a, b\}$, in which case $\cS$9 reduces to the statement
\[(\{\hat a, b\},a)(B,b)=(B,b)(B-b+b^{-1},a^{-1})\g_{a^{-1}}(\hat a),\]
which we call $\cS\textrm{9}'$. Then $\cS$9 follows from $\cS\textrm{9}'$
and $\cS$1--$\cS$8. To see this, let  $(A,a)$ and  $(B,b)$  be such that the 
conditions of $\cS$9 hold. Then 
$(A,a)=(\{\hat a,b\},a)(A-b,a)$, from $\cS$5, so
\begin{align*}
(A,a)(B,b)&=(\{\hat a,b\},a)(A-b,a)(B,b)\\
&=(\{\hat a,b\},a)(B,b)(A-b,a),\textrm{ from $\cS$8},\\
&=(B,b)(B-b+b^{-1},a^{-1})\g_{a^{-1}}(\hat a)(A-b,a), 
\textrm{ from $\cS\textrm{9}'$},\\
&=(B,b)(B-b+b^{-1},a^{-1})(A-b,a)\g_{a^{-1}}(\hat a),\textrm{ from $\cS$6},\\
&=(B,b)(B-A+\hat a +b^{-1},a^{-1})(A-b,a^{-1})(A-b,a)\g_{a^{-1}}(\hat a),\\
&\hspace{1em}\textrm{ from $\cS$5},\\
&=(B,b)(B-A+\hat a +b^{-1},a^{-1})\g_{a^{-1}}(\hat a),\textrm{ from $\cS$1}.
\end{align*}

Therefore it suffices to check that $\cS\textrm{9}'$ maps to an equality
in $\AA$. Suppose then that $\hat a \in J$,  
$b\in X^{\pm 1}_S$ and $\hat a \in B$. 
From \ref{it:R7}, for all $k\in B_J-\hat a$,  
\[\tilde{\tr}_{b^{-1},a}\tilde{\a}_{k,b^{-1}}=\tilde{\a}_{k,a}^{-1}\tilde{\a}_{k,b^{-1}}\tilde{\tr}_{b^{-1},a},\]
so 
\[\tilde{\a}_{k,b^{-1}}\tilde{\tr}_{b^{-1},a}=\tilde{\a}_{k,a}\tilde{\tr}_{b^{-1},a}\tilde{\a}_{k,b^{-1}}\]
and
\begin{align*}
\tilde{\tr}_{b^{-1},a}^{-1}\tilde{\a}_{k,b}&= \tilde{\a}_{k,b}\tilde{\tr}_{b^{-1},a}^{-1}\tilde{\a}_{k,a}^{-1}\\
&=\tilde{\a}_{k,b}\tilde{\a}_{k,a}^{-1}\tilde{\tr}_{b^{-1},a}^{-1},
\end{align*}
using \ref{it:last} and  \ref{it:R6}.
Similarly, for all $z\in B_S$, \ref{it:R2} implies that
\[\tilde{\tr}_{b^{-1},a}\tilde{\tr}_{z,b^{-1}}= \tilde{\tr}_{z,a}^{-1}\tilde{\tr}_{z,b^{-1}}\tilde{\tr}_{b^{-1},a},\]
so, using \ref{it:last} and \ref{it:R1}, we obtain
\[\tilde{\tr}_{b^{-1},a}^{-1}\tilde{\tr}_{z,b}= \tilde{\tr}_{z,b}\tilde{\tr}_{z,a}^{-1}\tilde{\tr}_{b^{-1},a}^{-1}.\]
Then, setting $i=\hat a$, 
\begin{align*}
(\{i, b\},a)&\Phi(B,b)\Phi =
\tilde{\tr}_{b,a}\tilde{\a}_{i,b}
\left(\prod_{k\in B_J-i} \tilde{\a}_{k,b} \prod_{z\in B_S} \tilde{\tr}_{z,b}
\right)\\
&=\tilde{\a}_{i,b}\tilde{\tr}_{b^{-1},a}^{-1}
\left(
\prod_{k\in B_J-i} \tilde{\a}_{k,b} \prod_{z\in B_S} \tilde{\tr}_{z,b}
\right)
\g_a(i)^{-1},\textrm{ from \ref{it:R9} and \ref{it:lastdash}}\\
&=\tilde{\a}_{i,b}
\left(
\prod_{k\in B_J-i} \tilde{\a}_{k,b}  \tilde{\a}_{k,a}^{-1}\prod_{z\in B_S}\tilde{\tr}_{z,b} \tilde{\tr}_{z,a}^{-1}
\right)
\tilde{\tr}_{b^{-1},a}^{-1}
\g_a(i)^{-1}, \textrm{ using the above},\\
&=\tilde{\a}_{i,b}\left(\prod_{k\in B_J-i} \tilde{\a}_{k,b} \prod_{z\in B_S} \tilde{\tr}_{z,b}
\right)\left(
\prod_{k\in B_J-i} \tilde{\a}_{k,a}^{-1}\prod_{z\in B_S}\tilde{\tr}_{z,a}^{-1}
\right)\tilde{\tr}_{b^{-1},a}^{-1}
\g_a(i)^{-1}, \\
&\hspace{1em}\textrm{ using \ref{it:R4} and \ref{it:R6}},\\
&=(B,b)\Phi \,(B+b^{-1}+b,a^{-1})\Phi\, \g_a(i)^{-1}\Phi,
\end{align*}
as required.

This concludes the proof that  
substitution of $q\Phi$ for $q$ in $s$, for all $q\in \cQ$ and 
all $s\in \cS$ results in the trivial element of $\AA$; so 
$\Phi$ is a homomorphism. From the definitions, 
$\Theta\Phi$ is the identity of $\Aut(G)$ and 
$\Phi\Theta$ is the identity of $\AA$, so $\AA\cong \Aut(G)$ and 
$\la \cP|\cR\ra$ is a presentation of $\Aut(G)$.
\end{proof}

%%% Local Variables: 
%%% mode: latex
%%% TeX-master: "aut2.tex"
%%% End: 
%%%%%%%%%%%%%%%%%%%%%%%%%%%%%%%%%%%%%%%%%%%%%%%%%
%
%
%%%%%%%%%%%%%%%%%%%%%%%%%%%%%%%%%%%%%%%%%%%%%%%%%
\bibliographystyle{plain}
\bibliography{aut}
\end{document}